\documentclass[pdfamsfonts]{amsart}
\usepackage[english]{babel}
\usepackage{inputenc}
\usepackage{mathabx}
\usepackage{amsmath,amsthm,amsfonts,amssymb,amscd}
\usepackage{mathpazo}
\usepackage{amssymb,amsfonts}
\usepackage[all,arc]{xy}
\usepackage{enumerate}
\usepackage{mathrsfs}
\usepackage{tgtermes}
\usepackage{marginnote}
\usepackage{mathtools}
\usepackage{amsthm,thmtools,xcolor}
\DeclarePairedDelimiter\ceil{\lceil}{\rceil}

\usepackage[all]{xy}
\usepackage{tikz-cd}
\usetikzlibrary{cd}
\numberwithin{equation}{section}
\theoremstyle{plain}
\newtheorem{thm}{Theorem}[section]

\newtheorem{prop}{Proposition}[section]
\newtheorem{lem}{Lemma}[section]

\newtheorem{thmx}{Theorem}

\theoremstyle{definition}

\theoremstyle{remark}
\newtheorem{rem}{\textit{Remark}}

\usepackage{color}

\usepackage[colorlinks = true,
linkcolor = blue,
urlcolor  = blue,
citecolor = blue,
anchorcolor = blue]{hyperref}
\usepackage{hyperref}
\makeatletter
\let\c@equation\c@thm
\makeatother
\numberwithin{equation}{section}
\DeclareMathOperator{\ex}{\mathrm{e}}
\DeclareMathOperator{\dist}{\mathrm{dist}}

\DeclareMathOperator{\supp}{supp}
\newcommand\underrel[3][]{\mathrel{\mathop{#3}\limits_{%
			\ifx c#1\relax\mathclap{#2}\else#2\fi}}}
\title[Fractional  KdV equation]{On the propagation of regularity for solutions of the fractional Korteweg-de Vries equation}

\author{Argenis. J. Mendez}

\address{Instituto Nacional  de  Matematica Pura e  Aplicada, Rio de Janeiro, RJ, Brasil}

\email{amendez@impa.br}

\thanks{This work was partially supported by  CNPq, Brazil.}


\date{February, 2018.}
\commby{Argenis Mendez}
\begin{document}
\begin{abstract}
%	The Korteweg-de Vries (KdV) equation, provides a relatively  description of motions of long waves in shallow water under gravity.
We consider the initial value problem  (IVP)  for the fractional Korteweg-de Vries equation (fKdV)
%In this article we  study  special regularity properties of solutions
%to the the fractional Korteweg-de Vries  equation (fKdV) 
\begin{equation}\label{abstracteq1}
\left\{
\begin{array}{ll}
\partial_{t}u-D_{x}^{\alpha}\partial_{x}u+u\partial_{x}u=0, & x,t\in\mathbb{R},\,0<\alpha<1, \\
u(x,0)=u_{0}(x).&  \\
\end{array} 
\right.
\end{equation}
It has been shown that the  solutions to  certain dispersive equations  satisfy the propagation of regularity phenomena. More precisely,  it deals in determine whether regularity of the initial data on the
right hand side of the real line  is propagated  to the left hand side by the flow solution. This property was found originally in solutions of Korteweg-de Vries (KdV) equation and it  has been also  verified in other dispersive equations   as the Benjamin-Ono (BO) equation.  

Recently, it has been shown that the solutions of the dispersive generalized Benjamin-Ono (DGBO) equation,  this is $\alpha\in (2,3)$ in \eqref{abstracteq1}; also satisfy the propagation of regularity phenomena. This is achieved by introducing a commutator decomposition to handle the dispersive part in the equation.  Following the approach used in the DGBO, we prove that the  solutions of the fKdV also satisfies the propagation of regularity phenomena. Consequently, this type of regularity travels with infinite speed to its left as time evolves.
\end{abstract}
\maketitle
\section{Introduction}
This paper provides a detailed study on  the propagation of regularity   satisfied   by solutions  for the so-called fractional Korteweg-de Vries (fKdV) equation
%This work is devoted to study the regularity of solutions to the initial value problem (IVP) associated to the fractional KdV equation 
\begin{equation}\label{e1}
\left\{
\begin{array}{ll}
\partial_{t}u-D_{x}^{\alpha}\partial_{x}u+u\partial_{x}u=0, & x,t\in\mathbb{R},\,0<\alpha<1, \\
u(x,0)=u_{0}(x),&  \\
\end{array} 
\right.
\end{equation}
where $u=u(x,t)$ represents a real valued function and the fractional derivative operator $D_{x}^{s}$ is defined  via its Fourier transform as 
\begin{equation*}
\widehat{D_{x}^{s}f}(\xi)=c_{s}|\xi|^{s}\widehat{f}(\xi)\quad \mbox{for}\quad s>0.
\end{equation*} 
 In the case $\alpha=1,$  the operator $D_{x}$  can be written as $D_{x}=\mathcal{H}\partial_{x}$ where $\mathcal{H}$ denotes the Hilbert transform,
\begin{equation*}
(\mathcal{H}f)(x)=\frac{1}{\pi}\lim_{\epsilon\rightarrow 0}\int_{|y|\geq \epsilon}\frac{f(x-y)}{y}\,\mathrm{d}y.
\end{equation*} 
 Also, for this particular value of $\alpha,$ the equation in \eqref{e1} becomes an integral equation widely studied and known in the literature as the Benjamin-Ono equation (BO) i.e,
  \begin{equation*}
  \partial_{t}u-\mathcal{H}\partial_{x}^{2}u+u\partial_{x}u=0, \quad  x,t\in\mathbb{R}.
  \end{equation*}
  
%  Following  \cite{KATO1} it is said that the IVP \eqref{e1} is \emph{locally well-posed} (LWP) in the Banach space $X$ if given  any initial data $u_{0}\in X,$ there exist  a time $T>0$ and a unique solution of the IVP \eqref{e1} satisfying
%  \begin{equation*}
%  u\in C\left([-T,T]:X\right)\cap Y_{T},
%  \end{equation*} 
%  where $Y_{T}$ is an auxiliary function set. In addition,
%  the map data solution, $u_{0}\longmapsto u,$ must be continuous. In the case that $T$ can be taken arbitrary large, it is said that  the IVP \eqref{e1} is  \emph{globally well-posed} (GWP) in the space $X.$
  
  The IVP \eqref{e1} has been the focus of attention of recent studies where the local well posedness theory has been approached as well as  the  existence of solitary wave solutions,  the stability properties of ground states and numerical simulations, these issues have been  largely carried out by  Linares et al. \cite{LIPICA2},  \cite{LIPICLA}, Frank and Lenzmann \cite{Lezman}, 
  Klein and Saut \cite{Kleins}, Molinet, Pilod and Vento \cite{mpv}, and Angulo  \cite{angulo}.
  
  The proof of our  main result is based on a weighted energy estimate, where the integrability  of the term $\partial_{x}u$ is  fundamental. So that, for  the applicability of this method we  will consider local well-posedness (LWP) results  in the Sobolev scale  with the minimal regularity that ensures $\partial_{x}u\in L^{1}([0,T]).$ To our knowledge the best result that fits with our requirements was obtained by Linares, Pilod and Saut \cite{LIPICLA}  when  studying  weakly  dispersive perturbations of Burger's equation. More precisely, they obtained  the following:
\begin{thm}\label{lt}
	Let $0<\alpha<1.$      Define $s(\alpha):=\frac{3}{2}-\frac{3\alpha}{8}$ and assume that  $s>s(\alpha)$.
	Then, for every  $u_{0}\in H^{s}(\mathbb{R}),$ there exists a positive  time $T=T(\|u_{0}\|_{s,2})$ (which can be chosen as a non-increasing functions  of its argument) and a unique  solution  $u$  to (\ref{e1}) such that 
	\begin{equation}
	u\in C\left([0,T]: H^{s}(\mathbb{R})\right)\quad \mbox{and}\quad  \partial_{x}u\in L^{1}\left([0,T]: L^{\infty}(\mathbb{R})\right).
	\end{equation}
\end{thm}
%\begin{proof}
%	See \cite[Theorem 1.5]{LIPICLA}.
%\end{proof}
It is also important to point out  that several improvements  have been made in the well-posedness theory of the IVP  \eqref{e1}. In this sense, we shall mention the recent work announced by  Molinet, Pilod and Vento \cite{mpv}, where  the authors prove that the IVP \eqref{e1} is locally well-posedness in the Sobolev space $H^{s}(\mathbb{R}),$  with $s>\frac{3}{2}-\frac{5\alpha}{4}.$ Additionally, in \cite{mpv}  is proved   Global well-posedness (GWP) in the energy space $H^{\alpha/2}(\mathbb{R}),$ as long as $\alpha>\frac{6}{7}.$  

Most of these results are based in a method introduced by  Molinet and Vento \cite{mv} to obtain energy
estimates at low regularity for strongly nonresonant dispersive equations. However, the LWP of the IVP \eqref{e1} in the space  $H^{s}(\mathbb{R}),$\, $s>\frac{3}{2}-\frac{5\alpha}{4},$ does not ensure the boundeness  of   $\|\partial_{x}u\|_{L^{1}_{T}L^{\infty}_{x}}$ which is crucial in our analysis as  mentioned above.

Since the space where we will consider solutions has already been described, we will proceed to describe  our main result. This  is inspired in a property   found originally  by Isaza, Linares and Ponce \cite{ILP1} in solutions of the KdV. More precisely, the authors prove that regularity on the right hand side  of the data travels forward in time with infinite speed. This  property  ''\emph{propagation of regularity phenomena principle}''  was also  studied  by Isaza et al. \cite{ILP2} in solutions of the Benjamin-Ono equation, where the term that providing  dispersion is  more difficult to handle  due to the presence of the Hilbert transform. 

Recently, in \cite{AJ}, we  studied the  propagation of regularity in solutions of the dispersion generalized Benjamin-Ono equation  i.e  
\begin{equation}\label{DGBO}
\partial_{t}u-D_{x}^{\alpha+1}\partial_{x}u+u\partial_{x}u=0, \quad  x,t\in\mathbb{R}, \alpha \in (0,1)
\end{equation}
where a combination of the techniques introduced in \cite{ILP1} and \cite{ILP2} together with  a    commutator decomposition provided by Ginibre and Velo \cite{GV2}; allowed us  to show that real solutions associated to the IVP \eqref{DGBO} also satisfies  the propagation of regularity phenomena. More precisely, this is summarized in the following:
\begin{thm}\label{DGBOT}
	Let $u_{0}\in H^{s}(\mathbb{R})$ with $s=\frac{3-\alpha}{2},$\, 
	and $u=u(x,t),$  be the corresponding solution  of the IVP (\ref{DGBO}).% provided by Theorem \ref{B.1}. 
	
	If  for some $x_{0}\in \mathbb{R}$  and for some $m\in\mathbb{Z}^{+},\,m\geq  2,$
	\begin{equation}\label{eq106}
	\partial_{x}^{m}u_{0}\in L^{2}\left(\left\{x\geq x_{0}\right\}\right),
	\end{equation}
	then for any $v>0,T>0,\epsilon>0$ and    \,$\tau>4\epsilon$
	\begin{equation}\label{eqc1}
	\begin{split}
	&\sup_{0\leq t\leq T}\int_{x_{0}+\epsilon-vt}^{\infty}(\partial_{x}^{j}u)^{2}(x,t)\mathrm{d}x
	+ \int_{0}^{T}\int_{x_{0}+\epsilon-vt}^{x_{0}+\tau-vt}\left(D_{x}^{\frac{\alpha+1}{2}}\partial_{x}^{j}u\right)^{2}(x,t)\mathrm{d}x\,\mathrm{d}t\\
	&		+\int_{0}^{T}\int_{x_{0}+\epsilon-vt}^{x_{0}+\tau-vt}\left(D_{x}^{\frac{\alpha+1}{2}}\mathcal{H}\partial_{x}^{j}u\right)^{2}(x,t)\mathrm{d}x\,\mathrm{d}t\leq c 
	\end{split}
	\end{equation}	
	for $j=1,2,\dots,m$ with $c=c\left(T;\epsilon;v;\alpha;\|u_{0}\|_{H^{s}};\left\|\partial_{x}^{m}u_{0}\right\|_{L^{2}((x_{0},\infty))}\right)>0.$
	
	If in addition to (\ref{eq106}) there exists $x_{0}\in \mathbb{R}^{+}$ 
	\begin{equation}\label{clave1}
	D_{x}^{\frac{1-\alpha}{2}}\partial_{x}^{m}u_{0}\in L^{2}\left(\left\{x\geq x_{0}\right\}\right)
	\end{equation}
	then for any $v\geq 0,\,\epsilon>0$ and $\tau>4\epsilon$
	\begin{equation}\label{l1}
	\begin{split}
	&\sup_{0\leq t\leq  T}\int_{x_{0}+\epsilon-vt}^{\infty}\left(D_{x}^{\frac{1-\alpha}{2}}\partial_{x}^{m}u\right)^{2}(x,t)\,\mathrm{d}x
	+\int_{0}^{T}\int_{x_{0}+\epsilon-vt}^{x_{0}+\tau-vt}\left(\partial_{x}^{m+1}u\right)^{2}(x,t)\,\mathrm{d}x\,\mathrm{d}t\\
	&+\int_{0}^{T}\int_{x_{0}+\epsilon-vt}^{x_{0}+\tau-vt}\left(\partial_{x}^{m+1}\mathcal{H}u\right)^{2}(x,t)\,\mathrm{d}x\,\mathrm{d}t\leq c
	\end{split}
	\end{equation}
	with $c=c\left(T;\epsilon;v;\alpha;\|u_{0}\|_{H^{s}};\left\|D_{x}^{\frac{1-\alpha}{2}}\partial_{x}^{m}u_{0}\right\|_{L^{2}((x_{0},\infty))}\right)>0.$
\end{thm}  
This property has been  found in solutions of several  dispersive models,  not only in the one dimensional setting, but in the  multidimensional one too. In this sense, Isaza et al. \cite{ILP4} proved  that real solutions of the Kadomtsev-Petviashvilli (KPII) equation satisfy the propagation of regularity phenomena. Later, Linares and Ponce \cite{LPZK} extended the study to solutions of the Zakharov-Kuznetsov equation in two  dimensions  resp. three dimensions.

Since this property is present in several  dispersive models, as the mentioned above,  it  seems  reasonable to ask if it also satisfied in models with less dispersion  than that in DGBO. 

So that,  a question that arises naturally from the theorem above, is to determine  if in the  case that less dispersion is considered, for example the IVP \eqref{e1}, do the solutions of the fKdV also satisfy the propagation of regularity phenomena.

Our main objective in this paper is to  give answer to this question.
More precisely, we obtain that the propagation of regularity phenomena is also satisfied by solutions of the fKdV. %Without more delays we proceed to describe our main result. This  deals with  the propagation of regularity principle  in solutions of the fKdV and it is summarized in the following:
\begin{thmx}\label{A}
	Let $u_{0}\in H^{s_{\alpha}+}(\mathbb{R})$ where $s_{\alpha}:=2-\frac{\alpha}{2},$ and  $u=u(x,t)$ be  the corresponding solution of the IVP	\eqref{e1}
	provided by Theorem \ref{lt}. Suppose that  for some  $x_{0}\in \mathbb{R}$ and   some  $m\in \mathbb{Z}^{+},\, m\geq 2,$
	\begin{equation}\label{eq2.3}
	\partial_{x}^{m}u_{0}\in L^{2}\left(\{x\geq x_{0}\}\right).
	\end{equation} 
%	If  ${\footnotesize\alpha\in \left[\frac{2}{2k+1},\frac{1}{k}\right)}$   for some positive integer $k\geq 1,$ 
	Then for any $v\geq0,\, T>0,\, \epsilon>0$ and $\tau>4\epsilon$ are satisfied the following relations:
	\begin{equation*}
	{\footnotesize
		\begin{split}
		&\sup_{0\leq t\leq  T} \int_{x_{0}+\epsilon-vt}^{\infty} \left(\partial_{x}^{j}D_{x}^{\frac{\alpha n }{2}}u\right)^{2}(x,t)\,\mathrm{d}x+ \int_{0}^{T}\int_{x_{0}+\epsilon-vt}^{x_{0}+\tau-vt}\left(D_{x}^{j+\alpha \left(\frac{n+1}{2}\right)}u\right)^{2}(x,t)\,\mathrm{d}x\,\mathrm{d}t\\
		&+\int_{0}^{T}\int_{x_{0}+\epsilon-vt}^{x_{0}+\tau-vt}\left(\mathcal{H}D_{x}^{j+\alpha\left(\frac{ n+1}{2}\right)}u\right)^{2}(x,t)\,\mathrm{d}x\,\mathrm{d}t\leq c,
		\end{split}}
	\end{equation*}
	\begin{equation*}
	{\footnotesize	\begin{split}
		&\sup_{0\leq t\leq  T} \int_{x_{0}+\epsilon-vt}^{\infty} \left(\partial_{x}^{j}D_{x}^{1-\frac{\alpha  }{2}}u\right)^{2}(x,t)\,\mathrm{d}x+ \int_{0}^{T}\int_{x_{0}+\epsilon-vt}^{x_{0}+\tau-vt}\left(\partial_{x}^{j+1}u\right)^{2}(x,t)\,\mathrm{d}x\,\mathrm{d}t\\
		&+\int_{0}^{T}\int_{x_{0}+\epsilon-vt}^{x_{0}+\tau-vt}\left(\mathcal{H}\partial_{x}^{j+1}u\right)^{2}(x,t)\,\mathrm{d}x\,\mathrm{d}t\leq c,
		\end{split}}
	\end{equation*}
	and
	\begin{equation*}
	{\footnotesize
		\begin{split}
		&\sup_{0\leq t\leq  T} \int_{x_{0}+\epsilon-vt}^{\infty} \left(\partial_{x}^{m}u\right)^{2}(x,t)\,\mathrm{d}x+ \int_{0}^{T}\int_{x_{0}+\epsilon-vt}^{x_{0}+\tau-vt}\left(D_{x}^{m+\frac{\alpha}{2}}u\right)^{2}(x,t)\,\mathrm{d}x\,\mathrm{d}t\\
		&+\int_{0}^{T}\int_{x_{0}+\epsilon-vt}^{x_{0}+\tau-vt}\left(\mathcal{H}D_{x}^{m+\frac{\alpha}{2}}u\right)^{2}(x,t)\,\mathrm{d}x\,\mathrm{d}t\leq c,
		\end{split}}
	\end{equation*}
	where  $j=2,3,\cdots,m-1$ and $n=0,1,\cdots, \ceil{\frac{2}{\alpha}}-1.$
	\\
	If in addition  to \eqref{eq2.3}
	\begin{equation*}
	{\small \begin{split}
		&D_{x}^{1-\frac{\alpha}{2}}\partial_{x}^{m}u_{0}\in L^{2}\left(\left\{x\geq x_{0}\right\}\right)\, \mbox{and}\,\,	D_{x}^{\frac{\alpha j}{2}}\partial_{x}^{m}u_{0}\in L^{2}\left(
		\left\{x\geq x_{0}\right\}\right),
		\end{split}}
	\end{equation*}
for 	$j=1,2,\cdots, \ceil{\frac{2}{\alpha}}-1,$	then for any $v\geq0,\, T>0,\, \epsilon>0$ and $\tau>4\epsilon$
	\begin{equation*}
	{\footnotesize
		\begin{split}
		&\sup_{0\leq t \leq T}\int_{\mathbb{R}}\left(\partial_{x}^{m}D_{x}^{\frac{\alpha j}{2}}u\right)^{2}(x,t)\,\mathrm{d}x+\int_{0}^{T}\int_{\mathbb{R}}\left(D_{x}^{m+\alpha\left(\frac{ j+1}{2}\right)}u\right)^{2}(x,t)\,\mathrm{d}x\,\mathrm{d}t\\
		&+\int_{0}^{T}\int_{\mathbb{R}}\left(\mathcal{H}D_{x}^{m+\alpha\left(\frac{ j+1}{2}\right)}u\right)^{2}(x,t)\,\mathrm{d}x\,\mathrm{d}t\leq c,
		\end{split}}
	\end{equation*}
	and 
	\begin{equation*}
	\footnotesize{
		\begin{split}
		&\sup_{0\leq t\leq  T} \int_{x_{0}+\epsilon-vt}^{\infty} \left(\partial_{x}^{m}D_{x}^{1-\frac{\alpha  }{2}}u\right)^{2}(x,t)\,\mathrm{d}x+ \int_{0}^{T}\int_{x_{0}+\epsilon-vt}^{x_{0}+\tau-vt}\left(\partial_{x}^{m+1}u\right)^{2}(x,t)\,\mathrm{d}x\,\mathrm{d}t\\
		&+\int_{0}^{T}\int_{x_{0}+\epsilon-vt}^{x_{0}+\tau-vt}\left(\mathcal{H}\partial_{x}^{m+1}u\right)^{2}(x,t)\,\mathrm{d}x\,\mathrm{d}t\leq c.
		\end{split}}
	\end{equation*}
\end{thmx}
\begin{rem}
	We shall remark the difference between  the index  $s(\alpha)$ and $s_{\alpha}$  used in Theorem \ref{lt} and Theorem \ref{A}, respectively.
	\end{rem}
The method of proof used in  Theorem \ref{A} follows in spirit the same lines than that used in the proof of Theorem \ref{DGBOT}. The main differences rely on the inductive argument. 

%We shall underline the main differences between the two results. comparison  the dispersion present in the DGBO the dispersion is  is less, so that the inductive argument used in \cite{AJ}  is not sufficient. 

In the case of the DGBO (see \eqref{DGBO}), the induction is  carried out in two steps one for a positive integer $m$ and another for $m+\frac{1-\alpha}{2},\, 0<\alpha<1.$ In contrast, in the fKdV is required a  bi-induction argument, the first for a positive integer $m,$  followed by a  a sequence of steps of the form $m+\alpha j/2,\, j=1,\cdots ,\ceil{\frac{2}{\alpha}}-1,$ where $\ceil{\cdot}$ denotes the greatest integer function, and  a  final one step for $m+1-\alpha/2.$ 

% However, this is not  our case, inasmuch as   for   $\alpha\lll 1,$ the  argument  used in the models above is not  sufficient. So that,   another procedure involving  more steps than those in BO and DGBO is  considered. In this sense,   a  "biinductive" process shall be considered, one for  the  integer $m,$ 
% 
Concerning the nonlinear part of the equation in \eqref{e1} the commutator expansion  \cite{GV1}, \cite{GV2} used in studying the propagation of regularity of the DGBO  becomes fundamental in our study,  as well as the recent  advances  presented by  Li \cite{Dongli} on the study of  commutators of Kato-Ponce type.   

The document is organized as follows. The first section is focused in the description  the  notation used thorough all  the document. Section 2 is dedicated to present several known results concerning commutator estimates joint with a particular decomposition  of these.  In section $3$ we make a review of the weighted functions  later used in  the proof of Theorem \ref{A}. Finally, section $4$  deals with the proof of our main result. 
	\section{Notation}
The following notation  will be used extensively  throughout this article. The operators $D_{x}^{s}=(-\partial_{x}^{2})^{s/2}$ and $J^{s}=(1-\partial_{x}^{2})^{s/2}$ denotes the  Riesz and Bessel potentials of order $-s,$ respectively. 

For $1\leq p\leq \infty,$\, $L^{p}(\mathbb{R})$  is the usual  Lebesgue space  with the norm  $\|\cdot\|_{L^{p}}=\|\cdot\|_{p},$  and the   besides for $s\in \mathbb{R},$  we consider the Sobolev space $L^{p}_{s}(\mathbb{R})$  is defined  via its usual norm $\|f\|_{s,p}=\|J^{s}f\|_{p}.$ In the  particular case  $p=2,$ the set $L^{2}_{s}(\mathbb{R})$ has a  Hilbert structure space  and we will denote it by  $H^{s}(\mathbb{R}).$

Let $f=f(x,t)$  be a function  defined for $x\in \mathbb{R}$ and $t$ in the time interval $[0,T],$ with $T>0$  or in the hole line $\mathbb{R}$. Then if $A$ denotes any of the spaces defined above, we define  the spaces  $L^{p}_{T}A_{x}$ and $L_{t}^{p}A_{x}$ by the norms 
\begin{equation*}
\|f\|_{L^{p}_{T}A_{x}}=\left(\int_{0}^{T}\|f(\cdot,t)\|_{A}^{p}\,\mathrm{d}t\right)^{1/p}\quad \mbox{and}\quad \|f\|_{L^{p}_{t}A_{x}}=\left(\int_{\mathbb{R}} \|f(\cdot,t)\|_{A}^{p}\,\mathrm{d}t\right)^{1/p},
\end{equation*}
for $1\leq p\leq \infty$ with the  natural modification in the case $p=\infty.$
Moreover, we use similar definitions for the  mixed spaces $L_{x}^{q}L_{t}^{p}$ and $L_{x}^{q}L_{T}^{p}$  with $1\leq p,q\leq \infty.$

For two quantities  $A$ and $B$, we denote  $A\lesssim B$  if $A\leq cB$ for some constant $c>0.$ Similarly $A\gtrsim B$  if  $A\geq cB$ for some $c>0.$  We denote  $A \sim B$  if $A\lesssim B$ and $B\lesssim A.$  The  dependence of the constant $c$  on other parameters  or constants  are usually clear  from the context  and we will often suppress this dependence. 

For $A,B$ operators we will denote  the commutator between $A$ and $B$ as $[A;B].$ 

The set of even numbers and odd numbers will be denoted by $\mathbb{Q}_{1}$ and $\mathbb{Q}_{2}$ respectively. Additionally, for  $m\in \mathbb{Z}^{+}$ the sets $\mathbb{Q}_{1}(m)$ and $\mathbb{Q}_{2}(m)$ will denote the set of even and odd numbers between $1$ and $m,$ respectively.

\section{ Inequalities \& commutators}
In this section we  collect several inequalities that will be used extensively throughout our work.

First,  we have     an extension of the Calderon Commutator  theorem \cite{calderon} established by B. Baj\v{s}anski et al. \cite{BC}. 
\begin{thm}
	Let $\mathcal{H}$  be the Hilbert transform. Then for any  $p\in (1,\infty)$ and any  $l,m\in \mathbb{Z}^{+}\cup\{0\}$ there exists  $c=c(p;l;m)>0$ such that 
	\begin{equation}\label{eq31}
	\|\partial_{x}^{l}[\mathcal{H}; \psi]\,\partial_{x}^{m}f\|_{p}\leq c\|\partial_{x}^{m+l}\psi\|_{\infty}\|f\|_{p}.
	\end{equation}
\end{thm}
The proof follows by results in \cite{BC}, for a different proof see [\cite{DMP}, Lemma 3.1].

In our analysis, when dealing with the nonlinear part of the equation \eqref{e1}; it  will be crucial the next inequalities concerning the Leibniz rule for fractional derivatives	  established in \cite{{GRF},{KATOP2},{KPV2}}. 

%However, to simplify the exposition  from now on  we restrict ourselves  to the one-dimensional case $x\in \mathbb{R}$ where in the next sections these results will be applied.
\begin{lem}\label{lema1}
	For $s>0,\, p\in [1,\infty)$
	\begin{equation}\label{eq5}
	\left\|D^{s}(fg)\right\|_{p} \lesssim \left\|f\right\|_{p_{1}}\left\|D^{s}g\right\|_{p_{2}}+\left\|g\right\|_{p_{3}}\left\|D^{s}f\right\|_{p_{4}}
	\end{equation}
	with
	\begin{equation*}
	\frac{1}{p}=\frac{1}{p_{1}}+\frac{1}{p_{2}}=\frac{1}{p_{3}}+\frac{1}{p_{4}},\quad p_{j}\in(1,\infty],\quad j=1,2,3,4.
	\end{equation*}
\end{lem}
%\begin{lem}\label{lem5}
%	Let $s=s_{1}+s_{2}\in (0,1)$  with $s_{1},s_{2}\in (0,s),$  and $p,p_{1},p_{2}\in (1,\infty)$ satisfy
%	\begin{equation*}
%	\frac{1}{p}=\frac{1}{p_{1}}+\frac{1}{p_{2}}.
%	\end{equation*} 
%	Then,
%	\begin{equation}\label{eq33}
%	\left\|D^{s}(fg)-fD^{s}g-gD^{s}f\right\|_{p}\lesssim \left\|D^{s_{1}}f\right\|_{p_{1}}\left\|D^{s_{2}}g\right\|_{p_{2}}.
%	\end{equation}
%	Moreover, the case $s_{2}=0$ and $p_{2}=\infty$ is allowed.
%\end{lem}
Recently  D.Li \cite{Dongli}  proved   new fractional Leibniz rules  for the nonlocal operator $D^{s},\,s>0,$  and related ones, including  various end-point situations. This type of estimate has proved to be a  very useful  tool in the study of propagation of regularity.
\begin{thm}\label{thm11}
	\underline{Case 1}: $1<p<\infty.$
	
	Let $s>0$ and $1<p<\infty.$ Then for any  $s_{1},s_{2}\geq 0$ with $s=s_{1}+s_{2},$  and any $f,g\in \mathcal{S}(\mathbb{R^{n}}),$ the following  hold:
	\begin{enumerate}
		\item If $1<p_{1},p_{2}<\infty$ with $\frac{1}{p}=\frac{1}{p_{1}}+\frac{1}{p_{2}},$ then
		\begin{equation}\label{commutator1}
		\begin{split}
		&\left\|D^{s}(fg)-\sum_{\alpha\leq s_{1}}\frac{1}{\alpha!}\partial^{\alpha}_{x}f D^{s,\alpha}g-\sum_{\beta\leq s_{2}}\frac{1}{\beta!}\partial^{\beta}_{x}g D^{s,\beta}f\right\|_{p}\lesssim \|D^{s_{1}}f\|_{p_{1}}\|D^{s_{2}}g\|_{p_{2}}.
		\end{split}
		\end{equation}
		\item If $p_{1}=p,\, p_{2}=\infty,$\, then
		\begin{equation}
		\begin{split}
		&\left\|D^{s}(fg)-\sum_{\alpha< s_{1}}\frac{1}{\alpha!}\partial^{\alpha}_{x}f D^{s,\alpha}g-\sum_{\beta\leq s_{2}}\frac{1}{\beta!}\partial^{\beta}_{x}g D^{s,\beta}f\right\|_{p}\lesssim \left\|D^{s_{1}}f\right\|_{p}\left\|D^{s_{2}}g\right\|_{\mathrm{BMO}},
		\end{split}
		\end{equation}
		where $\|\cdot\|_{\mathrm{BMO}}$  denotes the norm in the BMO space %\footnote{For any $f\in L^{1}_{loc}(\mathbb{R}^{n})$, the BMO semi-norm is given by 
%			\begin{equation*}
%			\|f\|_{\mathrm{BMO}}=\sup_{Q}\frac{1}{|Q|}\int_{Q}|f(y)-(f)_{Q}|\,\mathrm{d}y,
%			\end{equation*}
%			where $(f)_{Q}$ is the average of $f$ on $Q,$ and the supreme is taken over all cubes $Q$ in $\mathbb{R}^{n}.$ }.
		\item If $p_{1}=\infty,\, p_{2}=p,$\, then
		\begin{equation}\label{ce1}
		\begin{split}
		&\left\|D^{s}(fg)-\sum_{\alpha\leq s_{1}}\frac{1}{\alpha!}\partial^{\alpha}f D^{s,\alpha}g-\sum_{\beta< s_{2}}\frac{1}{\beta!}\partial^{\beta}g D^{s,\beta}f\right\|_{p}\lesssim \left\|D^{s_{1}}f\right\|_{\mathrm{BMO}}\left\|D^{s_{2}}g\right\|_{p}.
		\end{split}
		\end{equation}
		The operator  $D^{s,\alpha}$ is defined  via Fourier transform\footnote{The precise  form of the Fourier transform  does not matter.}
		\begin{equation*}
		\begin{split}
		&\widehat{D^{s,\alpha}g}(\xi)=\widehat{D^{s,\alpha}}(\xi)\widehat{g}(\xi),\\
		&\widehat{D^{s,\alpha}}(\xi)= i^{-\alpha}\partial_{\xi}^{\alpha}\left(|\xi|^{s}\right).
		\end{split}
		\end{equation*}
	\end{enumerate}
	\underline{Case 2}: $\frac{1}{2}<p\leq 1.$
	
	If $\frac{1}{2}<p\leq 1,\, s>\frac{1}{p}-1$ or $s\in 2\mathbb{N},$  then for any  $1<p_{1},p_{2}<\infty$  with  $$\frac{1}{p}=\frac{1}{p_{1}}+\frac{1}{p_{2}},$$ any  $s_{1},s_{2}\geq 0$ with  $s_{1}+s_{2}=s,$
	\begin{equation*}
	\begin{split}
	&\left\|D^{s}(fg)-\sum_{\alpha\leq s_{1}}\frac{1}{\alpha!}\partial^{\alpha}_{x}f D^{s,\alpha}g-\sum_{\beta\leq s_{2}}\frac{1}{\beta!}\partial^{\beta}_{x}g D^{s,\beta}f\right\|_{p}\lesssim \|D^{s_{1}}f\|_{p_{1}}\|D^{s_{2}}g\|_{p_{2}}.
	\end{split}
	\end{equation*}
\end{thm}
\begin{rem}
	As usual  empty summation  (such as $\sum_{0\leq\alpha<0}$) is defined as zero.
\end{rem}
\begin{proof}
	For a detailed proof of this theorem and related results, see \cite{Dongli}.
\end{proof}
%Next we have  the following commutator  estimates involving  non-homogeneous fractional derivatives, established by Kato and Ponce  .
%\begin{lem}[\cite{KATOP2}]
%	Let $s>0$ and $p,p_{2},p_{3}\in (1,\infty)$ and $p_{1},p_{4}\in (1,\infty]$	 be such that 
%	\begin{equation*}
%	\frac{1}{p}=\frac{1}{p_{1}}+\frac{1}{p_{2}}=\frac{1}{p_{3}}+\frac{1}{p_{4}}.
%	\end{equation*}
%	Then,
%	\begin{equation}\label{eq91}
%	\left\|\left[J^{s}; f\right]g\right\|_{p}\lesssim \|\partial_{x}f\|_{p_{1}}\|J^{s-1}g\|_{p_{2}}+\|J^{s}f\|_{p_{3}}\|g\|_{p_{4}}
%	\end{equation}
%	and 
%	\begin{equation}\label{eq90}
%	\|J^{s}(fg)\|_{p}\lesssim \|J^{s}f\|_{p_{1}}\|g\|_{p_{2}}+\|J^{s}g\|_{p_{3}}\|f\|_{p_{4}}.
%	\end{equation}
%\end{lem}
 Also, several commutator estimates have been obtained  by D. Li\,\cite{Dongli}. These, corresponds to a    family of refined Kato-Ponce type inequalities for the operator $D^{s}.$  
\begin{lem}\label{dlkp}
	Let $1<p<\infty.$  Let $1<p_{1},p_{2},p_{3},p_{4}\leq \infty$ satisfy 
	\begin{equation}
	\frac{1}{p}=\frac{1}{p_{1}}+\frac{1}{p_{2}}=\frac{1}{p_{3}}+\frac{1}{p_{4}}.
	\end{equation}
	Therefore,
	\begin{itemize}
		\item[(a)] If $0< s\leq 1,$  then 
		\begin{equation*}
		\|D^{s}(fg)-fD^{s}g \|_{p}\lesssim \|D^{s-1}\partial_{x}f \|_{p_{1}}\|g\|_{p_{2}}.
		\end{equation*}
		\item[(b)] If $s>1,$  then 
		\begin{equation}\label{kpdl}
		\|D^{s}(fg)-fD^{s}g\|_{p}\lesssim \|D^{s-1}\partial_{x}f\|_{p_{1}}\|g\|_{p_{2}}+\|\partial_{x}f\|_{p_{3}}\|D^{s-1}g\|_{p_{4}}.
		\end{equation}
	\end{itemize}
\end{lem}
In addition, a non sharp commutator estimate is required in our analysis.
\begin{lem} \label{kdvlem1}
	Let $\phi\in C^{\infty}(\mathbb{R})$ with  $\phi'\in C^{\infty}_{0}(\mathbb{R}).$  If  $f\in H^{s}(\mathbb{R}),\,s>0,$ then  for any $l>s+\frac{1}{2}$
	\begin{equation}
	\left\|\left[D_{x}^{s};\phi\right]f\right\|_{2}\lesssim\left\|\phi'\right\|_{l,2}\|f\|_{s-1,2}. 
	\end{equation}
\end{lem}
\begin{proof}
We will give an outline of the proof.

First, we  use the decomposition of the operator $D_{x}^{s}$ for $s>0,$ given by Bourgain and Li \cite{BOURGAINLI} in the study of commutator estimates.

In fact,
\begin{equation*}
D_{x}^{s}=J^{s}_{x}-\sum_{1\leq j\leq \frac{s}{2}} c_{s,j}J^{s-2j}_{x}+ K_{s}
\end{equation*}	
where $K_{s}$ is a bounded  integral  operator  satisfying $K_{s}:L^{p}(\mathbb{R})\longrightarrow L^{p}(\mathbb{R}), \, 1\leq p\leq \infty.$

By using this decomposition and a proof similar to that found in  \cite[Chapter 6,  Lemma 6.16 ]{Folland1},  the lemma follows.	
\end{proof}
%For a more detailed exposition on these estimates see section 5 in  \cite{Dongli}.
Also,   the following inequality of  Gagliardo-Nirenberg type is used.
\begin{lem}\label{lema2}
	Let  $1<q,p<\infty,\, 1<r\leq \infty$ and $0<\alpha<\beta.$  Then, 
	\begin{equation*}
	\|D^{\alpha}f\|_{L^{p}}\lesssim c \|f\|_{L^{r}}^{1-\theta}\|D^{\beta}f\|_{L^{q}}^{\theta}
	\end{equation*}
	with
	\begin{equation*}
	\frac{1}{p}-\alpha=(1-\theta)\frac{1}{r}+\theta\left(\frac{1}{q}-\beta\right),\quad \theta\in \left[\alpha/\beta,1\right].
	\end{equation*}
\end{lem}
\begin{proof}
	See \cite[chapter 4]{BL}.	
\end{proof}
%Now, we present a result that will help us  to establishing the propagation of regularity of solutions of (\ref{eq7}). A previous result was proved  by  Kenig et al.(c.f \cite{KLPV}, Corollary 2.1 ),  using  the fact  that $J^{r}$ ($r\in \mathbb{R}$) can be seen as a pseudo-differential operator. Thus,  this approach let to obtain an expression for $J^{r}$ in terms of a convolution with a certain kernel $k(x,y)$ which enjoys some properties of regularity and decay on localized regions in $\mathbb{R^{2}}$, in fact, this is known as the singular integral realization  of a pseudo-differential operator, whose proof can be found in \cite{stein2} Chapter 4.
Next we consider a result  that will be used widely when dealing with the nonlinear part of the fKdV. 
\begin{lem}\label{lemma1}
	Let $m\in \mathbb{Z}^{+}$ and $s\geq0.$   If  $f\in L^{2}(\mathbb{R})$ and $g\in L^{p}(\mathbb{R}),\,\,2\leq p\leq \infty,$ with
	\begin{equation}\label{cond1}
	\mathrm{dist}\left(\supp(f),\supp(g)\right)\geq \delta>0.
	\end{equation}
	Then
	\begin{equation*}
	\left\|g\,\partial^{m}_{x}D^{s}f\right\|_{L^{p}}\lesssim\|g\|_{L^{p}}\|f\|_{L^{2}}.
	\end{equation*}
\end{lem}
\begin{rem}
	A previous estimated was obtained by Kenig et al. \cite{KLPV} when studying the propagation of regularity (fractional case) but for the operator $J^{s}.$
	\end{rem}
\begin{proof}
	See \cite[Lemma 3.28]{AJ}.
\end{proof}
%\begin{lem}
%	If $f\in L^{2}(\mathbb{R})$ is such that $\varphi_{\epsilon}\partial_{x}^{m}f\in L^{2}(\mathbb{R})$ then 
%	for all $\epsilon'>2\epsilon$ 
%	\begin{equation*}
%	\varphi_{\epsilon'}D_{x}^{r}f\in L^{2}(\mathbb{R})\quad \mbox{for}\quad r\in (0,m].
%	\end{equation*}
%\end{lem}
%\begin{proof}[interpolation]
%	
%\end{proof}

	\subsection{Commutator Expansions }

In this section  we present a several auxiliary results obtained by Ginibre and Velo \cite{{GV1}}, \cite{GV2}  which has proved to be useful in the study of propagation of regularity.

%They include several commutator expansions together with its estimates. The basic problem is to handle  the non-local operator $D^{s}$ for non-integer $s$ and in particular  to obtain representations  of its commutator with multiplication operators by functions that exhibit as much locality as possible. 

Let $a=2\mu+1>1,$   let  $n$ be a nonnegative integer and $f$ be  a smooth  function with suitable  decay at infinity, for instance with $f'\in C^{\infty}_{0}(\mathbb{R}).$ 

We define the operator 

\begin{equation}\label{eq8}
R_{n}(a)=-\left[\mathcal{H} D^{a}; f\right]-\frac{1}{2}\left(P_{n}(a)-\mathcal{H}P_{n}(a)\mathcal{H}\right),
\end{equation}
where 
\begin{equation}\label{eq105}
P_{n}(a)=a\sum_{0\leq j\leq n}c_{2j+1}(-1)^{j}4^{-j}D^{\mu-j}f^{(2j+1)}D^{\mu-j},
\end{equation}
and the constants $c_{2j+1}$ are given by the  following  formula
\begin{equation}
c_{1}=1\quad\mbox{and}\quad c_{2j+1}=\frac{1}{(2j+1)!}\prod_{0\leq k<j}\left(a^{2}-\left(2k+1\right)^{2}\right).
\end{equation}

%It was shown in \cite{GV1} that the operator $R_{n}(a)$ can be represented in terms of anti-commutators \footnote{For any two  operators $P$ and $Q$ we denote the anti-commutator by $[P;Q]_{+}=PQ+QP.$}   as follows
%\begin{equation*}
%R_{n}(a)=\frac{1}{2}\left([H; Q_{n}(a)]_{+}+[D^{a};[H; h]]_{+}\right),
%\end{equation*}
%where  the operator $Q_{n}(a)$ is represented in the Fourier space variables by the integral kernel
%\begin{equation*}
%Q_{n}(a)\longrightarrow (2\pi)^{\frac{1}{2}}
%\widehat{h}(\xi-\xi')|\xi\xi'|^{\frac{a}{2}}2aq_{n}(a,t),
%\end{equation*}
%
%with $|\xi|=|\xi'|\ex^{2t}$  and 
%\begin{equation*}
%q_{n}(a,t)=\frac{1}{a}(a^{2}-(2n+1)^{2})c_{2n+1}\int_{0}^{t} \sinh^{2n+1}\tau\,\sinh((a(t-\tau)))\,\mathrm{d}\tau.
%\end{equation*}
%
%The function  $q_{n}(a,t)$ satisfies  the following estimates
%\begin{lem}
%	Let $n$ be a non-negative integer. The following estimates hold
%	\begin{enumerate}
%		\item For $2n+1\leq a\leq 2n+3$ (with equality for $a=2n+3$)
%		\begin{equation*}
%		2a|q_{n}(a,t)|\leq (2|\sinh t|)^{a}.
%		\end{equation*}
%		\item For $1\leq a\leq 2n+1$
%		\begin{equation*}
%		2a|q_{n}(a,t)|\leq C(2|\sinh t|)^{2n+1}.
%		\end{equation*}
%		\item For $1\leq a\leq 2n+3$
%		\begin{equation*}
%		2a|q_{n}(a,t)|\leq a(2n+3)^{-1}(2|\sinh t|)^{2n+3}.
%		\end{equation*}
%	\end{enumerate}	
%\end{lem}
%\begin{proof}
%	See Lemma 2.3 in  \cite{GV2}.
%\end{proof}
%
%Based on these estimates of the function $q_{n}(a,t)$, Ginibre and Velo \cite{GV2}
%obtain the following properties of boundedness and compactness of the operator $R_{n}(a).$

\begin{prop}\label{propo2}
	Let $n$  be a non-negative  integer,   $a\geq 1,\,$ and   $ \sigma\geq 0,$  be such that 
	\begin{equation}\label{eq21}
	2n+1\leq a+2\sigma\leq2n+3.
	\end{equation}
	
	Then 
	\begin{itemize}
		\item[(a)] The operator $D^{\sigma}R_{n}(a)D^{\sigma}$ is bounded in $L^{2}$ with norm 
		\begin{equation}\label{eq98}
		\left\|D^{\sigma}R_{n}(a)D^{\sigma}f\right\|_{2}\leq C(2\pi)^{-1/2}\left\|\widehat{\left(D^{a+2\sigma}f\right)}\right\|_{1}\|f\|_{2}.
		\end{equation}	
		If $a\geq 2n+1,$ one can take  $C=1.$
		
		\item[(b)] Assume in addition  that
		\begin{equation*}
		2n+1\leq a+2\sigma<2n+3.
		\end{equation*}
		Then the operator ${\displaystyle D^{\sigma}R_{n}(a)D^{\sigma}}$ is compact in $L^{2}(\mathbb{R}).$
	\end{itemize}
\end{prop}
\begin{proof}
	See  Proposition 2.2 in \cite{GV2}.
\end{proof}
\begin{rem}
 Proposition \ref{propo2} is a generalization from previous results, where   the derivatives of operator $R_{n}(a)$  are not  consider  (cf.  Proposition 1 in \cite{GV1}). 
\end{rem}
%
%\begin{proof}
%	
%	The proof follows of the ideas presented in  Proposition 2.12 in \cite{LIPICLA}.
%
%\end{proof}
%\section{Review: Local theory}
Also, a direct application of the commutator decomposition\eqref{eq8} is the smoothing effect associated to solutions of the IVP \eqref{e1}. 
\begin{prop}\label{fkdv1}
	Let $\varphi $ denote  a nondecreasing  smooth  function such that  $\supp\varphi'\subset (-1,2)$ and $\varphi|_{[0,1]}\equiv 1.$  For $j\in \mathbb{Z},$  we define  $\varphi_{j}(x)=\varphi(x-j).$  Let $u\in C\left([0,T]: H^{\infty}(\mathbb{R})\right)$ be a smooth solution  of (\ref{e1})  with $0<\alpha<1$.  Assume also that  $s\geq 0$  and $r>\frac{1}{2}.$  Then,
	\begin{equation}
	\begin{split}
	&\left(\int_{0}^{T}\int_{\mathbb{R}} \left(\left|D_{x}^{s+\frac{\alpha}{2}}u(x,t)\right|^{2}+\left|\mathcal{H}D_{x}^{s+\frac{\alpha}{2}}u(x,t)\right|^{2}\right)\varphi_{j}'(x)\,\mathrm{d}x\,\mathrm{d}t\right)^{1/2}\\
	&\lesssim \left(1+T+\|\partial_{x}u\|_{L^{1}_{T}L^{\infty}_{x}}+ T\|u\|_{L^{\infty}_{T}H^{r}_{x}}\right)^{1/2} \|u\|_{L^{\infty}_{T}H^{s}_{x}}.
	\end{split}
	 	 \end{equation}
	 	 \begin{proof}
	 	 As was mentioned above the proof use the decomposition \eqref{eq8} and a application of   Mikhlin's Theorem. For a detailed description of the proof see   \cite[Proposition 2.12]{LIPICLA}.
	 	 \end{proof}
\end{prop}
\begin{rem}
In the particular  case that $u_{0}$  is  in the Sobolev space $H^{s_{\alpha}}(\mathbb{R}),\, s_{\alpha}=2-\frac{\alpha}{2},$ the gain of local derivatives is summarized in the following inequality
	\begin{equation}\label{kse}
\begin{split}
&\left(\int_{0}^{T}\int_{-r}^{r} \left(\left|\partial_{x}^{2}u(x,t)\right|^{2}+\left|\mathcal{H}\partial_{x}^{2}u(x,t)\right|^{2}\right)\,\mathrm{d}x\,\mathrm{d}t\right)^{1/2}\leq C\left(r;T;\|u_{0}\|_{s_{\alpha},2}\right)
\end{split}
\end{equation}
 for any $r>0.$
 The proof follows combining  the  ideas  used in the proof of  Proposition \ref{fkdv1} and the  inequality 
 \begin{equation}
 \|u\|_{L^{\infty}_{T}H^{s_{\alpha}}_{x}}\lesssim \|u_{0}\|_{H^{s}_{x}}\ex^{\|\partial_{x}u\|_{L^{1}_{T}L^{\infty}_{x}}}.
 \end{equation}
 The last inequality above is widely know in the literature and its proof is based on  energy estimate   combined with commutator estimates Kato-Ponce type. For a detailed proof see for instance  \cite[Chapter 9, p.~221]{LIPO}
	\end{rem}

%\begin{lem}\label{SE1}
%	Consequence of the smoothing effect.
%	\begin{equation}
%	\begin{split}
%	&\left(\int_{0}^{T}\int_{\mathbb{R}} \left(|D_{x}^{r}u(x,t)|^{2}+|\mathcal{H}D_{x}^{r}u(x,t)|^{2}\right)\varphi_{j}'(x)\,\mathrm{d}x\,\mathrm{d}t\right)^{1/2}<\infty\quad r\in\left(0,s+\frac{\alpha}{2}\right)
%	\end{split}
%	\end{equation}
%\end{lem}
\section{Weighted  functions}
This section is devoted to describe the  the weighted  functions used in the proof of Theorem \ref{A}, as well as their  properties. 

Most of these functions where originally constructed  when studying the propagation of regularity phenomena  for solutions of the KdV \cite{ILP1}, \cite{KLPV}. %So that, much of these 

For  $\epsilon>0$ and $b\geq 5\epsilon$ define the  families of functions 
\begin{equation*}
\chi_{\epsilon,b},\;  \phi_{\epsilon,b},\; \widetilde{\phi_{\epsilon,b}},\; \psi_{\epsilon},\eta_{\epsilon,b}, \varphi_{\epsilon,b}\in C^{\infty}(\mathbb{R})
\end{equation*}
satisfying the following properties:
\begin{enumerate}
	\item ${\displaystyle 
		\chi_{\epsilon,b}'\geq 0,
	}$	
	\item ${\displaystyle 
		\chi_{\epsilon,b}(x)=
		\left\{
		\begin{array}{ll}
		0, & x\leq \epsilon \\
		1, & x\geq b,
		\end{array} 
		\right. 
	}$
	\item ${\displaystyle  \supp(\chi_{\epsilon, b})\subseteq [\epsilon,\infty);}$
	\item ${\displaystyle \chi_{\epsilon, b}'(x)\geq\frac{1}{10(b-\epsilon)}\mathbb{1}_{[2\epsilon,b-2\epsilon]}(x),}$
	\item ${\displaystyle \supp\left(\chi_{\epsilon,b}'\right) \subseteq [\epsilon,b].}$ 
	\item There exists real numbers $c_{j}$ such that
	\begin{equation*}
	\left|\chi_{\epsilon, b}^{(j)}(x)\right|\leq c_{j}\chi_{\epsilon/3, b+\epsilon}'(x),\quad \forall x\in \mathbb{R},\,j\in \mathbb{Z}^{+}.
	\end{equation*} 
	\item For $x\in (3\epsilon,\infty)$ 
	\begin{equation*}
	\chi_{\epsilon, b}(x)\geq\frac{1}{2}\frac{\epsilon}{b-3\epsilon}.
	\end{equation*}
		\item For $x\in \mathbb{R}$
	\begin{equation*}
	\chi_{\epsilon/3,b+\epsilon}'(x)\leq \frac{\epsilon}{b-3\epsilon}.
	\end{equation*}
	\item Also, given $\epsilon>0$ and $b\geq 5\epsilon$ there exist $c_{1},c_{2}>0$ such that 
	\[ \begin{array}{lcr}
	\chi_{\epsilon,b}'(x)\leq c_{1}\chi_{\epsilon/3,b+\epsilon}'(x)\chi_{\epsilon/3,b+\epsilon}(x), & \\
	\chi_{\epsilon,b}'(x)\leq c_{2} \chi_{\epsilon/5, \epsilon}(x). & \\
	\end{array}\] 
	\item For $\epsilon>0 $ given and $b\geq 5\epsilon,$ we define the functions 
	\begin{equation*}
	\eta_{\epsilon,b}=\sqrt{\chi_{\epsilon, b}\chi_{\epsilon,b }'}\quad\mbox{and}\quad \varphi_{\epsilon,b}=\sqrt{\chi_{\epsilon, b}'}.
	\end{equation*}
	\item ${\displaystyle \supp\left(\phi_{\epsilon,b}\right),\,\supp\left(\widetilde{\phi_{\epsilon, b}}\right)\subset \left[\epsilon/4,b\right],}$
	\item ${\displaystyle \phi_{\epsilon}(x)=\widetilde{\phi_{\epsilon, b}}(x)=1,\quad x\in [\epsilon/2,\epsilon],}$
	\item ${\displaystyle \supp(\psi_{\epsilon})\subseteq\left(-\infty,\epsilon/2\right]},$
	\item for $x\in \mathbb{R}$
	\begin{equation*}
	\chi_{\epsilon, b}(x)+\phi_{\epsilon, b}(x)+\psi_{\epsilon}(x)=1,
	\end{equation*}
	and 
	\begin{equation*}
	\chi_{\epsilon, b}^{2}(x)+\widetilde{\phi_{\epsilon, b}}^{2}(x)+\psi_{\epsilon}(x)=1.
	\end{equation*}
\end{enumerate}
The family ${\displaystyle \{\chi_{\epsilon, b}: \epsilon>0,\, b\geq 5\epsilon\}}$ is constructed as follows:
let  $\rho\in C^{\infty}_{0}(\mathbb{R}),\, \rho(x)\geq 0,$  even, with $\supp(\rho)\subseteq(-1,1)$ and $ \|\rho\|_{1}=1.$

Then defining
\begin{equation*}
\nu_{\epsilon,b}(x)=\left\{
\begin{array}{lll}
0, & x\leq 2\epsilon ,\\
\frac{x}{b-3\epsilon}-\frac{2\epsilon}{b-3\epsilon}, & 2\epsilon\leq x\leq b-\epsilon,\\
1,& x\geq b-\epsilon,
\end{array} 
\right. 
\end{equation*}
and
\begin{equation*}
\chi_{\epsilon,b}(x)=\rho_{\epsilon}*\nu_{\epsilon,b}(x)
\end{equation*}
where $\rho_{\epsilon}(x)=\epsilon^{-1}\rho(x/\epsilon).$
%xxxxxxxxxxxxxxxxxxxxxxxxxxNew Theoremxxxxxxxxxxxxxxxxxxxxxxxxxxxxxxxxxxxx

\section{Proof of Theorem \ref{A}}
%\begin{proof}[Proof of Theorem \ref{A}]
We shall use two induction arguments. One for $m\in \mathbb{Z}^{+}$ with $m \geq 2$ and the other one  for $m+\alpha k,$ where $k$ is a positive integer to be specified later in the proof.

Since the solutions of the IVP (\ref{e1}) are translation invariant, then without  loss of generality we will take $x_{0}=0.$ 

Also, it  will be assumed that  solutions of the IVP (\ref{e1}) are  real valued functions with   
as much  regularity as required. 
	\begin{flushleft}
\textbf{\underline{{\sc Case $m=2:$}}}
	\end{flushleft}
\begin{flushleft}
	\fbox{{\sc Step 1:}}
\end{flushleft}
%F  $\alpha$  satisfies the inequality   $\frac{2}{2k+1}\leq \alpha<\frac{1}{k}.$ The case  $\alpha\in \left[\frac{1}{k+1},\frac{2}{2k+1}\right)$ follows by using the same method of proof.
%%%%%%%%%%%%%%%%%%%%
Formally, we apply $\partial_{x}^{2}$ to the equation in (\ref{e1})  followed by a multiplication by $\partial_{x}^{2}u\chi_{\epsilon, b}^{2}$  to obtain 
\begin{equation*}
\begin{split}
\partial_{x}^{2}\partial_{t}u\partial_{x}^{2}u\chi_{\epsilon, b}^{2}-\partial_{x}^{2}D_{x}^{\alpha}\partial_{x}u \partial_{x}^{2}u\chi_{\epsilon, b}^{2}+\partial_{x}^{2}\left(u\partial_{x}u\right)\partial_{x}^{2}u\chi_{\epsilon, b}^{2}=0
\end{split}
\end{equation*} 
 after integrating in the position variable, we obtain the energy identity
\begin{equation*}
\begin{split}
&\frac{1}{2}\frac{\mathrm{d}}{\mathrm{d}t}\int_{\mathbb{R}}\left(\partial_{x}^{2}u\right)^{2}\chi_{\epsilon, b}^{2}\,\mathrm{d}x-\underbrace{\frac{v}{2}\int_{\mathbb{R}}\left(\partial_{x}^{2}u\right)^{2}\left(\chi_{\epsilon, b}^{2}\right)'\,\mathrm{d}x}_{B_{1}(t)}-\underbrace{\int_{\mathbb{R}}\left(\partial_{x}^{2}D_{x}^{\alpha}\partial_{x}u\right)\partial_{x}^{2}u\chi_{\epsilon, b}^{2}\,\mathrm{d}x}_{B_{2}(t)}\\
&+\underbrace{\int_{\mathbb{R}}\partial_{x}^{2}\left(u\partial_{x}u\right)\partial_{x}^{2}u\chi_{\epsilon, b}^{2}\,\mathrm{d}x}_{B_{3}(t)}=0.
\end{split}
\end{equation*}
\S.1   To handle $B_{1}$ first notice that  there exists  $c>0$  and  $r>0$ such that  
\begin{equation}\label{kdv5}
\left(\chi_{\epsilon,b }^{2}\right)'(x+vt)=2 \chi_{\epsilon, b}(x+vt)\chi_{\epsilon, b}'(x+vt)\leq c \mathbb{1}_{[-r,r]}(x)
\end{equation}
for  all $(x,t)\in \mathbb{R}\times[0,T].$ Thus, by Remark \ref{kse},  it follows that  
\begin{equation}\label{kdv6}
\begin{split}
\int_{0}^{T}|B_{1}(t)|\,\mathrm{d}t&\lesssim \int_{0}^{T}\int_{\mathbb{R}}\left(\partial_{x}^{2}u\right)^{2}\left(\chi_{\epsilon, b}^{2}\right)'\mathrm{d}x\mathrm{d}t\\
&\lesssim\int_{0}^{T}\int_{-r}^{r}\left(\partial_{x}^{2}u\right)^{2}\mathrm{d}x\mathrm{d}t\\
&\leq c\left(\|u_{0}\|_{s_{\alpha},2};r;T\right).
\end{split}
\end{equation}
\S.2  Now,  we  extract information from the term handling the dispersive part of the equation in \eqref{e1}.

Combining  integration  by parts  and Plancherel's identity  it follows that    
\begin{equation} 
\begin{split}
B_{2}(t)&=-\frac{1}{2}\int_{\mathbb{R}}\partial_{x}^{2}u\left[\mathcal{H}D_{x}^{\alpha+1}; \chi_{\epsilon, b}\right]\partial_{x}^{2}u\,\mathrm{d}x\\
&=-\frac{1}{2}\int_{\mathbb{R}}D_{x}^{2}u\left[\mathcal{H}D_{x}^{\alpha+1}; \chi_{\epsilon, b}\right]D_{x}^{2}u\,\mathrm{d}x.
\end{split}
\end{equation}
Then  using  \eqref{eq8} 
 \begin{equation}\label{eqdkv.2}
\begin{split}
\left[\mathcal{H}D_{x}^{\alpha+1};\chi_{\epsilon, b}\right]=-R_{n}(\alpha+1)-\frac{1}{2}P_{n}(\alpha+1)+ \frac{1}{2}\mathcal{H}P_{n}(\alpha+1)\mathcal{H}
\end{split}
\end{equation}
for some nonnegative  integer $n$  to be fixed. This will help to obtain the smoothing effect.

Indeed, replacing this decomposition into $B_{2}$ yields

\begin{equation*}
\begin{split}
B_{2}(t)&=\frac{1}{2}\int_{\mathbb{R}}D_{x}^{2}uR_{n}(\alpha+1)D_{x}^{2}u\,\mathrm{d}x+\frac{1}{4}\int_{\mathbb{R}} D_{x}^{2}uP_{n}(\alpha+1)D_{x}^{2}u\,\mathrm{d}x\\
&\quad -\frac{1}{4}\int_{\mathbb{R}} D_{x}^{2}u\mathcal{H}P_{n}(\alpha+1)\mathcal{H}D_{x}^{2}u\,\mathrm{d}x\\
&=B_{2,1}(t)+B_{2,2}(t)+B_{2,3}(t).
\end{split}
\end{equation*}
%The term which dictaminate $B_{2,1},$ since it let us to fix the quantity of terms necessary in the decomposition of the commutator.
The quantity of terms $n,$ will be   chosen  according to the following rule (see \eqref{eq21})
\begin{equation*}
2n+1\leq \alpha+5\leq 2n+3,
\end{equation*} 
which clearly implies $n=2.$

In view of \eqref{eq98} the remainder operator $R_{2}(\alpha+1)$ satisfies
\begin{equation}\label{kdvd1}
\left \|D_{x}^{2}R_{2}(\alpha+1)D_{x}^{2}f\right\|_{2}\lesssim \|f\|_{2}\left\|\widehat{D_{x}^{\alpha+5}\left(\chi_{\epsilon, b}^{2}\right)}\right\|_{1},
\end{equation}
for $f$ in a suitable class of functions.

For our proposes a  combination of \eqref{kdvd1}   with  H\"{o}lder's inequality  and Plancherel's identity is sufficient to obtain
\begin{equation*}
\begin{split}
B_{2,1}(t)&=\frac{1}{2}\int_{\mathbb{R}}uD_{x}^{2}R_{2}(\alpha+1)D_{x}^{2}u\,\mathrm{d}x\lesssim \|u_{0}\|_{2}^{2}\left\|\widehat{D_{x}^{\alpha+5}\left(\chi_{\epsilon, b}^{2}\right)}\right\|_{1}.
\end{split}
\end{equation*}
Therefore, integrating in time  yields
\begin{equation*}
\int_{0}^{T}|B_{2,1}(t)|\,\mathrm{d}t\leq c.
\end{equation*}
Replacing  $P_{2}(\alpha+1)$ into $B_{2,2}$ and $B_{2,3}$ produce 
\begin{equation*}
\begin{split}
B_{2,2}(t)&=\left(\frac{\alpha+1}{4}\right)\int_{\mathbb{R}}\left(D_{x}^{2+\frac{\alpha}{2}}u\right)^{2}\left(\chi_{\epsilon, b}^{2}\right)'\,\mathrm{d}x-c_{3}\left(\frac{\alpha+1}{16}\right)\int_{\mathbb{R}}
\left(D_{x}^{1+\frac{\alpha}{2}}u\right)^{2} \left(\chi_{\epsilon, b}^{2}\right)^{(3)}\,\mathrm{d}x\\
&\quad +c_{5}\left(\frac{\alpha+1}{64}\right)\int_{\mathbb{R}} \left(D_{x}^{\frac{\alpha}{2}}u\right)^{2}\left(\chi_{\epsilon, b}^{2}\right)^{(5)}\,\mathrm{d}x\\
&=B_{2,2,1}(t)+B_{2,2,2}(t)+B_{2,2,3}(t).
\end{split}
\end{equation*}
and 
\begin{equation*}
\begin{split}
B_{2,3}(t)&= \left(\frac{\alpha+1}{4}\right)\int_{\mathbb{R}}\left(\mathcal{H}D_{x}^{2+\frac{\alpha}{2}}u\right)^{2}\left(\chi_{\epsilon, b}^{2}\right)'\,\mathrm{d}x\\
&\, -c_{3}\left(\frac{\alpha+1}{16}\right)\int_{\mathbb{R}}
\left(\mathcal{H}D_{x}^{1+\frac{\alpha}{2}}u\right)^{2} \left(\chi_{\epsilon, b}^{2}\right)^{(3)}\mathrm{d}x+c_{5}\left(\frac{\alpha+1}{64}\right)\int_{\mathbb{R}} \left(\mathcal{H}D_{x}^{\frac{\alpha}{2}}u\right)^{2}\left(\chi_{\epsilon, b}^{2}\right)^{(5)}\mathrm{d}x\\
&=B_{2,3,1}(t)+B_{2,3,2}(t)+B_{2,3,3}(t).
\end{split}
\end{equation*}
Notice that  $B_{2,2,1}$ and $B_{2,3,1}$ are positive and represent the smoothing effect. Then we need to bound  the terms $B_{2,2,2},B_{2,2,3}, B_{2,3,2},$ and $B_{2,3,3}.$

After integrate in time the  terms $B_{2,2,2},B_{2,3,2},B_{2,2,3},$ and  $B_{2,3,3}$  can be handled by using the local theory i.e.
\begin{equation*}
\begin{split}
&\int_{0}^{T}|B_{2,2+l,2}(t)|\mathrm{d}t\lesssim \sup_{0\leq t \leq T}\|u(t)\|_{s_{\alpha},2} \quad\mbox{for}\quad l\in\{0,1\},
\end{split}
\end{equation*}
and 
\begin{equation*}
\begin{split}
&\int_{0}^{T}|B_{2,2+l,3}(t)|\mathrm{d}t\lesssim \sup_{0\leq t \leq T}\|u(t)\|_{s_{\alpha},2} \quad\mbox{for}\quad  l\in\{0,1\}.
\end{split}
\end{equation*}
\S.3 \quad Finally, we     handle $B_{3}$ firstly  applying  integration by parts as follows  
\begin{equation*}
\begin{split}
B_{3}(t)&=\frac{5}{2}\int_{\mathbb{R}}\partial_{x}u\left(\partial_{x}^{2}u\right)^{2}\chi_{\epsilon, b}^{2}\,\mathrm{d}x-\frac{1}{2}\int_{\mathbb{R}}u\left(\partial_{x}^{2}u\right)^{2}\left(\chi_{\epsilon,b}^{2}\right)'\,\mathrm{d}x\\&=B_{3,1}(t)+B_{3,2}(t).
\end{split}
\end{equation*}  

Since  $u$  satisfies  Strichartz  estimate i.e. $\partial_{x}u\in L^{1}\left([0,T]:L^{\infty}(\mathbb{R})\right)
$ in Theorem \ref{lt}, then 
\begin{equation}\label{kdvf1}
\begin{split}
|B_{3,1}(t)|&\lesssim \|\partial_{x}u(t)\|_{\infty}\int_{\mathbb{R}}\left(\partial_{x}^{2}u\right)^{2}\chi_{\epsilon, b}^{2}\,\mathrm{d}x.
\end{split}
\end{equation}
 We shall also point out that the integral expression on  the right hand side of \eqref{kdvf1}  will be estimated by using Gronwall's inequality.
 
An application of  Sobolev's embedding  produces 
\begin{equation}
\begin{split}
|B_{3,2}(t)|&\lesssim \|u(t)\|_{\infty}\int_{\mathbb{R}} \left(\partial_{x}^{2}u\right)^{2}\left(\chi_{\epsilon, b}^{2}\right)'\,\mathrm{d}x\\
&\lesssim\sup_{0\leq t \leq T} \|u(t)\|_{s_{\alpha},2}\int_{\mathbb{R}} \left(\partial_{x}^{2}u\right)^{2}\left(\chi_{\epsilon, b}^{2}\right)'\,\mathrm{d}x.
\end{split}
\end{equation}
 We finish this step    gathering the estimates above, this  together with an application of   Gronwall's inequality  yield
\begin{equation}\label{eqfkdv1}
\begin{split}
&\sup_{0\leq t\leq  T}\left\|\partial_{x}^{2}u\chi_{\epsilon, b}(\cdot+vt)\right\|_{2}^{2}+\left\|D_{x}^{2+\frac{\alpha}{2}}u\eta_{\epsilon,b}^{2}\right\|_{L^{2}_{T}L^{2}_{x}}^{2}+\left\|\mathcal{H}D_{x}^{2+\frac{\alpha}{2}}u\eta_{\epsilon,b}^{2}\right\|_{L^{2}_{T}L^{2}_{x}}^{2}\leq c^{*}_{2,1}
\end{split}
\end{equation}
for any $\epsilon>0,\, b\geq 5\epsilon$ and $v\geq 0.$

At this point several issues shall be clarified and fixed. First, we  indicate the  dependence on  the parameters involved in the constant $c^{*}_{2,1},$ this will be crucial later when we will consider the limit process. More precisely,  $c^{*}_{2,1}=c^{*}_{2,1}\left(\alpha; \epsilon; T;v; \|u_{0}\|_{s_{\alpha},2}; \left\|\partial_{x}^{2}u_{0}\chi_{\epsilon, b}\right\|_{2}\right).$ 

These families of constants are distinguished in our argument, so that we will differentiate them  by fixing  the following notation: for   $m,n\in\mathbb{Z}^{+},$ the number  $c^{*}_{m,n}$  corresponds to the case  of $m-$derivatives  in the $n$th step of the induction process. 

Without more delays we proceed to  the case 2 in our inductive process. This  step  is summarized in the following diagram
 \begin{center}\label{e2}
 	{\small
 		\begin{tikzcd}
 		\partial_{x}^{2}u\chi_{\epsilon, b}^{2} \arrow[r, blue] \arrow[d,"D_{x}^{\alpha/2}"] & D_{x}^{\alpha/2}\partial_{x}^{2}u(\chi_{\epsilon, b}^{2})' \arrow[dl,dashrightarrow, red]\\
 		D_{x}^{\alpha/2}\partial_{x}^{2}u\chi_{\epsilon, b}^{2} \arrow[r, blue]& D_{x}^{\alpha}\partial_{x}^{2}u(\chi_{\epsilon, b}^{2})'.\\
 		\end{tikzcd}}
 \end{center}
 The columns on the left hand side indicates the propagation of regularity, which is carried out by steps of length $\alpha/2,$ except for the final step that will be exemplified later. Instead, the columns on the right hand side furnish  the smoothing effect obtained at that level of   propagation.  
  
  Finally, the diagonal lines are to indicate the dependence of the smoothing effect in the next level of propagation.  
 
% We emphasize that this  corresponds to the first induction process, i.e we move from 2 to 3 derivatives (propagated) by iterating $\alpha/2$ derivatives $2k-1$ times, except for the final process, where a change is made. This will be exemplified further.  
% 
 Now that the strategy has been explained we proceed to estimate.
\begin{flushleft}
	\fbox{{\sc Step 2:}}
\end{flushleft}
After apply  the operator $D_{x}^{\frac{\alpha}{2}}\partial_{x}^{2}$ to the equation in  \eqref{e1} and multiply the resulting equation  by $D_{x}^{\frac{\alpha}{2}}\partial_{x}^{2}u\chi_{\epsilon, b}^{2}(x+vt)$ one gets
 \begin{equation*}
 \begin{split}
 &\frac{1}{2}\frac{\mathrm{d}}{\mathrm{d}t}\int_{\mathbb{R}}\left(\partial_{x}^{2}D_{x}^{\frac{\alpha }{2}}u\right)^{2}\chi_{\epsilon, b}^{2}\,\mathrm{d}x-\underbrace{\frac{v}{2}\int_{\mathbb{R}}\left(\partial_{x}^{2}D_{x}^{\frac{\alpha }{2}}u\right)^{2}\left(\chi_{\epsilon, b}^{2}\right)'\,\mathrm{d}x}_{B_{1}(t)}\\
 &
 -\underbrace{\int_{\mathbb{R}}\left(\partial_{x}^{2}D_{x}^{\frac{\alpha }{2}}D_{x}^{\alpha}\partial_{x}u\right)\partial_{x}^{2}D_{x}^{\frac{\alpha }{2}}u\chi_{\epsilon, b}^{2}\,\mathrm{d}x}_{B_{2}(t)}+\underbrace{\int_{\mathbb{R}} \left(
 	\partial_{x}^{2}D_{x}^{\frac{\alpha }{2}}\left(u\partial_{x}u\right)\right)\partial_{x}^{2}D_{x}^{\frac{\alpha }{2}}u\chi_{\epsilon, b}^{2}\,\mathrm{d}x}_{B_{3}(t)}=0.
 \end{split}
 \end{equation*}
\S.1\quad   As we indicated in the diagram above, to handle  $\|B_{1}\|_{1}$   it is only  required to use  \eqref{eqfkdv1}. Then,    
\begin{equation}\label{eq1.4}
\begin{split}
\int_{0}^{T}|B_{1}(t)|\,\mathrm{d}t&\leq \frac{|v|}{2}\int_{0}^{T}\int_{\mathbb{R}} \left(\partial_{x}^{2}D_{x}^{\frac{\alpha }{2}}u\right)^{2}\left(\chi_{\epsilon, b}^{2}\right)'\,\mathrm{d}x\,\mathrm{d}t\leq c^{*}_{2,1}.
\end{split}
\end{equation}
\S.2 The term $B_{2}$ shall be rewritten  as was done before in order to obtain the corresponding  smoothing effect in  this step. Thus, we  apply integration by parts and Plancherel's identity to get 
\begin{equation}\label{kdv2}
\begin{split}
B_{2}(t) %=-\frac{1}{2}\int_{\mathbb{R}}\partial_{x}^{2}D_{x}^{\frac{\alpha }{2}}u\left[\mathcal{H}D_{x}^{\alpha+1};\chi_{\epsilon, b}^{2}\right]\partial_{x}^{2}D_{x}^{\frac{\alpha }{2}}u\,\mathrm{d}x\\
&=-\frac{1}{2}\int_{\mathbb{R}}D_{x}^{2+\frac{\alpha }{2}}u\left[\mathcal{H}D_{x}^{\alpha+1};\chi_{\epsilon, b}^{2}\right]D_{x}^{2+\frac{\alpha }{2}}u\,\mathrm{d}x.
\end{split}
\end{equation}
By using    the commutator decomposition \eqref{eqdkv.2} into \eqref{kdv2}we have 
\begin{equation*}
\begin{split}
B_{2}(t)
&=\frac{1}{4}\int_{\mathbb{R}}D_{x}^{2+\frac{\alpha }{2}}uR_{n}(\alpha+1)D_{x}^{2+\frac{\alpha }{2}}u\,\mathrm{d}x+\frac{1}{4}\int_{\mathbb{R}}D_{x}^{2+\frac{\alpha }{2}}uP_{n}(\alpha+1)D_{x}^{2+\frac{\alpha }{2}}u\,\mathrm{d}x\\
&\quad -\frac{1}{4}\int_{\mathbb{R}}D_{x}^{2+\frac{\alpha }{2}}u\mathcal{H}P_{n}(\alpha+1)\mathcal{H}D_{x}^{2+\frac{\alpha }{2}}u\,\mathrm{d}x\\
&=B_{2,1}(t)+B_{2,2}(t)+B_{2,3}(t),
\end{split}
\end{equation*}
for some positive integer $n$  to be fixed.

From this term we obtain the smoothing effect by using a similar analysis as the used in the previous step. This argument  allow us to fix $n=2,$ and for this particular value, 
%In virtue of  Proposition \ref{propo2}, for $n\in \mathbb{Z}^{+}$  satisfying 
%\begin{equation}\label{eq1.3}
%2n+1\leq \alpha+1+2\left(2+\frac{\alpha }{2}\right)\leq 2n+3
%\end{equation}
 the operator $R_{2}(\alpha+1)$ is bounded from $L^{2}(\mathbb{R})$ into $L^{2}(\mathbb{R}).$
 % more precisely it    verifies 
% \begin{equation*}
% \left\|D_{x}^{2+\frac{\alpha }{2}}R_{n}(\alpha+1)D_{x}^{2+\frac{\alpha }{2}}f\right\|_{2}\lesssim \|f\|_{2}\left\|\widehat{D_{x}^{2\alpha+5}\left(\chi_{\epsilon, b}^{2}\right)}\right\|_{1}
% \end{equation*}
% for $f$ in a suitable class.
 
 %Choosing $n$ in  this way, we obtain  $n=2.$
 
 Hence  replacing $P_{2}(\alpha+1)$ into $B_{2,2}$  and $B_{2,3}$  lead us to,  in first place that
\begin{equation*}
\begin{split}
B_{2,2}(t)&=
\left(\frac{\alpha+1}{4}\right)\int_{\mathbb{R}}\left(D_{x}^{2+\alpha}u\right)^{2}(\chi_{\epsilon,b }^{2})'\,\mathrm{d}x-c_{3}\left(\frac{\alpha+1}{16}\right)\int_{\mathbb{R}}\left(D_{x}^{1+\alpha}u\right)^{2}(\chi_{\epsilon, b}^{2})^{(3)}\mathrm{d}x\\
&\quad +c_{5}\left(\frac{\alpha+1}{64}\right)\int_{\mathbb{R}}\left(D_{x}^{\alpha}u\right)^{2}(\chi_{\epsilon, b}^{2})^{(5)}\mathrm{d}x\\
&=B_{2,2,1}(t)+B_{2,2,2}(t)+B_{2,2,3}(t),
\end{split}
\end{equation*}
and in second place 
\begin{equation*}
\begin{split}
B_{2,3}(t)&=
\left(\frac{\alpha+1}{4}\right)\int_{\mathbb{R}}\left(\mathcal{H}D_{x}^{2+\alpha}u\right)^{2}(\chi_{\epsilon, b}^{2})'\mathrm{d}x
-c_{3}\left(\frac{\alpha+1}{16}\right)\int_{\mathbb{R}}\left(\mathcal{H}D_{x}^{1+\alpha}u\right)^{2}(\chi_{\epsilon, b}^{2})^{(3)}\mathrm{d}x\\
&\quad +c_{5}\left(\frac{\alpha+1}{64}\right)\int_{\mathbb{R}}\left(\mathcal{H}D_{x}^{\alpha}u\right)^{2}\left(\chi_{\epsilon, b}^{2}\right)^{(5)}\mathrm{d}x\\
&= B_{2,3,1}(t)+B_{2,3,2}(t)+B_{2,3,3}(t).
\end{split}
\end{equation*}
The terms  $B_{2,2,1}$ and $B_{2,3,1}$ are positives and give us the smoothing effect.
We estimate next   $B_{2,2,2}$ and $B_{2,3,2}$. To avoid repetitions,  we  only show the  procedures how to bound $B_{2,3,2}.$        A similar analysis  is applied to  $B_{2,2,2}.$

Before  carry on,  we shall remember that for any  $\epsilon>0$ and $b\geq 5\epsilon$ the function  $\varphi_{\epsilon,b}=\sqrt{\chi_{\epsilon, b}'}$
is smooth.

Besides,
\begin{equation}
\begin{split}
\varphi_{\epsilon/3,b+\epsilon}D_{x}^{1+\alpha}u&=D_{x}^{1+\alpha}(u\varphi_{\epsilon/3,b+\epsilon})-\left[D_{x}^{1+\alpha};\varphi_{\epsilon/3,b+\epsilon}\right]u,
\end{split}
\end{equation}
then a combination of Lemma \ref{kdvlem1} and interpolation produce
\begin{equation}
\begin{split}
\left\|\varphi_{\epsilon/3,b+\epsilon}D_{x}^{1+\alpha}u\right\|_{2}&\lesssim \left\|D_{x}^{1+\alpha}(u\varphi_{\epsilon/3,b+\epsilon})\right\|_{2}+ \|\varphi_{\epsilon/3,b+\epsilon}\|_{l,2}\|u\|_{\alpha,2}\\
&\lesssim\left\|\partial_{x}^{2}\left(u\varphi_{\epsilon/3,b+\epsilon}\right)\right\|_{2}+\|\varphi_{\epsilon/3,b+\epsilon}\|_{l,2}\|u\|_{\alpha,2}\\
&\lesssim \left\|\partial_{x}^{2}u\varphi_{\epsilon/3,b+\epsilon}\right\|_{2}+\|u\|_{s_{\alpha},2}+\|\varphi_{\epsilon/3,b+\epsilon}\|_{l,2}\|u\|_{\alpha,2}.
\end{split}
\end{equation}
Thus, 
 \begin{equation*}
 \begin{split}
\int_{0}^{T}|B_{2,2,2}(t)|\,\mathrm{d}t&=c\left\|\varphi_{\epsilon/3,b+\epsilon}D_{x}^{1+\alpha}u\right\|_{L^{2}_{T}L^{2}_{x}}^{2}\\
&\lesssim \left\|\partial_{x}^{2}u\varphi_{\epsilon/3,b+\epsilon}\right\|_{L^{2}_{T}L^{2}_{x}}^{2}+\|u\|_{L^{\infty}_{T}H^{s_{\alpha}}_{x}}^{2}\\
&\leq c\left(\|u_{0}\|_{s_{\alpha},2};\epsilon;v;T\right),
 \end{split}
 \end{equation*}
 where the last inequality is a consequence of the local theory, interpolation and \eqref{kdv6}.
 
Similarly, 
  \begin{equation*}
 \begin{split}
 \int_{0}^{T}|B_{2,3,2}(t)|\,\mathrm{d}t&\lesssim  \int_{0}^{T}\int_{\mathbb{R}} \left(\mathcal{H}D_{x}^{1+\alpha}u\right)^{2}\chi_{\epsilon/3,b+\epsilon}'\mathrm{d}x\,\mathrm{d}t\\
 & \leq c\left(\|u_{0}\|_{s_{\alpha},2};\epsilon;v;T\right).
 \end{split}
 \end{equation*}
  After integrate in  time 
 \begin{equation*}
 \int_{0}^{T}|B_{2,2+l,3}(t)|\,\mathrm{d}t<\infty\quad \mbox{for}\quad l\in \{0,1\}.
 \end{equation*}
\S.3  Finally,  we  deal  with  the term $B_{3},$ which   corresponds to the nonlinear part of the equation  in \eqref{e1}.

First, we decompose the nonlinearity as follows   
 \begin{equation*}
\begin{split}
\partial_{x}^{2}D_{x}^{\frac{\alpha}{2}}\left(u\partial_{x}u\right)\chi_{\epsilon,b }&=-D_{x}^{2+\frac{\alpha }{2}}(u\partial_{x}u)\chi_{\epsilon,b }\\
&=\frac{1}{2}\left[D_{x}^{2+\frac{\alpha }{2}};\chi_{\epsilon, b}\right]\partial_{x}\left((u\chi_{\epsilon, b})^{2}+(u\widetilde{\phi_{\epsilon, b}})^{2}+u^{2}\psi_{\epsilon}\right)\\
&\quad - \left[D_{x}^{2+\frac{\alpha}{2}};u\chi_{\epsilon, b}\right]\partial_{x}\left(u\chi_{\epsilon, b}+u\phi_{\epsilon, b}+u\psi_{\epsilon}\right)+u\chi_{\epsilon, b}\partial_{x}^{2}D_{x}^{\frac{\alpha }{2}}\partial_{x}u\\
&=\widetilde{B_{3,1}}(t)+\widetilde{B_{3,2}}(t)+\widetilde{B_{3,3}}(t)+\widetilde{B_{3,4}}(t)+\widetilde{B_{3,5}}(t)+\widetilde{B_{3,6}}(t)+\widetilde{B_{3,7}}(t).
\end{split}
\end{equation*}
To estimate $B_{3}$ is sufficient  the $L^{2}(\mathbb{R})-$norm of the terms $\widetilde{B_{3,l}}$ for $l=1,2,\dots,7.$ 

 Combining  Lemma \ref{lema1}    and Lemma \ref{kdvlem1} it is obtained
\begin{equation}
\begin{split}
\left\|\widetilde{B_{3,1}}\right\|_{2}&\lesssim \|\partial_{x}\chi_{\epsilon, b}\|_{l,2}\left\|\left(u\chi_{\epsilon, b}\right)^{2}\right\|_{2+\alpha/2,2}\\
&\lesssim \left\|\left(
u\chi_{\epsilon, b}\right)^{2}\right\|_{2}+\left\|D_{x}^{2+\frac{\alpha}{2}}\left(\left(u\chi_{\epsilon, b}\right)^{2}\right)\right\|_{2}\\
&\lesssim \|u\|_{\infty}\left(\|u_{0}\|_{2}+\left\|D_{x}^{2+\frac{\alpha}{2}}\left(u\chi_{\epsilon, b}\right)\right\|_{2} \right)
\end{split}
\end{equation}
and 
\begin{equation}
\begin{split}
\|\widetilde{B_{3,2}}\|_{2}&\lesssim \|\partial_{x}\chi_{\epsilon, b}\|_{l,2}\left\|\left(u\phi_{\epsilon, b}\right)^{2}\right\|_{2+\alpha/2,2}\\
%&\lesssim \left\|\left(
%u\phi_{\epsilon, b}\right)^{2}\right\|_{2}+\left\|D_{x}^{2+\frac{\alpha}{2}}\left(\left(u\phi_{\epsilon, b}\right)^{2}\right)\right\|_{2}\\
&\lesssim \|u\|_{\infty}\left(\|u_{0}\|_{2}+\left\|D_{x}^{2+\frac{\alpha}{2}}\left(u\phi_{\epsilon, b}\right)\right\|_{2} \right).
\end{split}
\end{equation}
Since the weighted functions   $\chi_{\epsilon, b}$ and $\psi_{\epsilon}$  satisfy hypothesis of Lemma \ref{lemma1}, 
%\begin{equation*}
%\dist\left(\supp(\chi_{\epsilon, b}),\supp(\psi_{\epsilon})\right)\geq \frac{\epsilon}{2}>0,
%\end{equation*}
then 
\begin{equation*}
\begin{split}
\|\widetilde{B_{3,3}}\|_{2}&=\left\|\chi_{\epsilon, b}D_{x}^{2+\frac{\alpha }{2}}(u^{2}\psi_{\epsilon})\right\|_{2}\lesssim \|u\|_{\infty}\|u_{0}\|_{2},
\end{split}
\end{equation*}
and 
\begin{equation*}
\begin{split}
\|\widetilde{B_{3,6}}\|_{2}&=\left\|u\chi_{\epsilon, b}D_{x}^{2+\frac{\alpha }{2}}(u\psi_{\epsilon})\right\|_{2}\lesssim\|u_{0}\|_{2}\|u\|_{\infty}.
\end{split}
\end{equation*}
To  handle the terms $\widetilde{B_{3,4}}$ and $\widetilde{B_{3,1}},$ we use  the commutator estimate \eqref{kpdl} to yield
\begin{equation}
\begin{split}
\|\widetilde{B_{3,4}}\|_{2}\lesssim \left\|D_{x}^{2+\frac{\alpha }{2}}(u\chi_{\epsilon, b})\right\|_{2}\left\|\partial_{x}(u\chi_{\epsilon, b})\right\|_{\infty}\\
\end{split}
\end{equation}
and 
\begin{equation}
\|\widetilde{B_{3,5}}\|_{2}\lesssim \left\|D_{x}^{2+\frac{\alpha }{2}}(u\phi_{\epsilon, b})\right\|_{2}\left\|\partial_{x}(u\chi_{\epsilon, b})\right\|_{\infty}+\left\|D_{x}^{2+\frac{\alpha }{2}}(u\chi_{\epsilon,b })\right\|_{2}\left\|\partial_{x}(u\phi_{\epsilon, b})\right\|_{\infty}.
\end{equation}
%%%%%
The nonlocal character of the operator $D_{x}^{s},$ for  $s\notin 2\mathbb{N},$ implies that several terms above shall be estimated. First, an immediate  application of   Theorem \ref{thm11} yield
\begin{equation*}
\begin{split}
\left\|D_{x}^{2+\frac{\alpha}{2}}\left(u\phi_{\epsilon, b}
\right)\right\|_{2}
&\lesssim  \left\|D_{x}^{2+\frac{\alpha}{2}}\phi_{\epsilon, b}\right\|_{\mathrm{BMO}}\|u\|_{2}+\left\|\phi_{\epsilon, b} D_{x}^{2+\frac{\alpha}{2}}u \right\|_{2}+\left\|\partial_{x}\phi_{\epsilon, b}\mathcal{H}D_{x}^{1+\frac{\alpha}{2}}u\right\|_{2}\\
&\quad+\left\|\partial_{x}^{2}\phi_{\epsilon, b}D_{x}^{\frac{\alpha}{2}}u\right\|_{2}\\
&\lesssim \left\|D_{x}^{\frac{\alpha}{2}}\partial_{x}^{2}\phi_{\epsilon, b}\right\|_{\infty}\|u_{0}\|_{2}+\left\|
\chi_{\epsilon/8,b+\epsilon/4}'D_{x}^{2+\frac{\alpha}{2}}u\right\|_{2}+\|u\|_{s_{\alpha},2}\\
&\lesssim \left\|
\chi_{\epsilon/8,b+\epsilon/4}'D_{x}^{2+\frac{\alpha}{2}}u\right\|_{2}+\|u\|_{s_{\alpha},2}.
\end{split}
\end{equation*}
Notice that the second term on the right hand side is bounded by local theory. By using  the weighted functions  properties combined with  \eqref{eqfkdv1} it follows
\begin{equation}
\begin{split}
\int_{0}^{T}\left\|D_{x}^{2+\frac{\alpha}{2}}u \chi_{\epsilon/8,b+\epsilon/4}\right\|_{2}\mathrm{d}t&\leq T^{1/2}\left\|D_{x}^{2+\frac{\alpha}{2}}u\chi_{\epsilon/8,b+\epsilon/4}'\right\|_{L^{2}_{T}L^{2}_{x}}\\
&\lesssim T^{1/2}\left\|D_{x}^{2+\frac{\alpha}{2}}u\eta_{\epsilon/24,b+7\epsilon/24}\right\|_{L^{2}_{T}L^{2}_{x}}\\
&\lesssim \left(c^{*}_{2,1}\right)^{1/2}.
\end{split}
\end{equation}
Similarly
\begin{equation}
\left\|D_{x}^{2+\frac{\alpha}{2}}\left(u\widetilde{\phi_{\epsilon,b}}\right)\right\|_{L^{1}_{T}L^{2}_{x}}\lesssim \left(c^{*}_{2,1}\right)^{1/2}.
\end{equation} 
Finally, to estimate 
$
\left\|D_{x}^{2+\frac{\alpha}{2}}(u\chi_{\epsilon, b})\right\|_{2},$
	we write  
\begin{equation}
\begin{split}
D_{x}^{2+\frac{\alpha}{2}}(u\chi_{\epsilon, b})&=-\chi_{\epsilon, b}\partial_{x}^{2}D_{x}^{\frac{\alpha}{2}}u+\left[D_{x}^{2+\frac{\alpha}{2}};\chi_{\epsilon, b}\right]\left(u\chi_{\epsilon, b}+u\phi_{\epsilon,b}+u\psi_{\epsilon}\right).
\end{split}
\end{equation}
Thus, a  combination of  Lemma \ref{kdvlem1} and interpolation lead us to 
\begin{equation} 
\left\|D_{x}^{2+\frac{\alpha}{2}}(u\chi_{\epsilon,b})\right\|_{2}\lesssim \left\|\chi_{\epsilon, b}\partial_{x}^{2}D_{x}^{\frac{\alpha}{2}}u\right\|_{2}+\|u\|_{s_{\alpha},2}.
\end{equation}
Notice that the first term on the right hand side is the quantity to be estimated by Gronwall's inequality.

Concerning to  $B_{3,7}$ we obtain after apply integration by parts  
\begin{equation*}
\begin{split}
B_{3,7}(t)&= -\frac{1}{2}\int_{\mathbb{R}}\partial_{x}u\chi_{\epsilon,b }^{2}\left(\partial_{x}^{2}D_{x}^{\frac{\alpha }{2}}u\right)^{2}\,\mathrm{d}x-\frac{1}{2}\int_{\mathbb{R}}u\left(\chi_{\epsilon,b }^{2}\right)'\left(\partial_{x}^{2}D_{x}^{\frac{\alpha }{2}}u\right)^{2}\mathrm{d}x\\
&=B_{3,7,1}(t)+B_{3,7,2}(t).
\end{split}
\end{equation*}
On one hand, we have 
\begin{equation*}
\begin{split}
|B_{3,7,1}(t)|&\lesssim \|\partial_{x}u(t)\|_{\infty}\int_{\mathbb{R}}\left(\partial_{x}^{2}D_{x}^{\frac{\alpha }
	{2}}u\right)^{2}\chi_{\epsilon, b}^{2}\mathrm{d}x,
\end{split}
\end{equation*}
 the integral expression on the right hand side  will be estimate by  Gronwall's inequality and  by   Theorem \ref{lt} we have      $\partial_{x}u \in L^{1}_{T}L^{\infty}_{x}.$

On the other hand,  by Sobolev's embedding it follows that  
\begin{equation*}
\begin{split}
|B_{3,7,2}(t)|&\lesssim \|u(t)\|_{\infty}\int_{\mathbb{R}}\left(\chi_{\epsilon, b}^{2}\right)'(D_{x}^{2+\frac{\alpha }{2}}u)^{2}\,\mathrm{d}x\\
&\lesssim \sup_{0\leq t\leq  T}\|u(t)\|_{s_{\alpha},2}\int_{\mathbb{R}}\left(\chi_{\epsilon, b}^{2}\right)'\left(\partial_{x}^{2}D_{x}^{\frac{\alpha}{2}}u\right)^{2}\,\mathrm{d}x.
\end{split}
\end{equation*}

Integrating in time  yields   
\begin{equation*}
 \int_{0}^{T}|B_{3,7,2}(t)|\,\mathrm{d}t\lesssim \sup_{0\leq t\leq  T}\|u(t)\|_{s_{\alpha},2}\int_{0}^{T}\int_{\mathbb{R}}\left(\chi_{\epsilon, b}^{2}\right)'\left(\partial_{x}^{2}D_{x}^{\frac{\alpha }{2}}u\right)^{2}\mathrm{d}x\,\mathrm{d}t,
  \end{equation*} 
  where the integral expression  corresponds to the $B_{1}$ term already estimated  in \eqref{eq1.4}.

Gathering the estimates corresponding to $B_{1},B_{2}$ and $B_{3}$ combined with   Gronwall's inequality  and integration in time yields that for any $\epsilon>0, b\geq 5\epsilon$ and $v\geq 0$
\begin{equation}
\begin{split}
&\sup_{0\leq t\leq T}\left\|\partial_{x}^{2}D_{x}^{\frac{\alpha }{2}}u\chi_{\epsilon, b}(\cdot+vt)\right\|_{2}^{2}+\left\|D_{x}^{2+\alpha}u\eta_{\epsilon, b}\right\|_{L^{2}_{T}L^{2}_{x}}^{2}+\left\|\mathcal{H}D_{x}^{2+\alpha}u\eta_{\epsilon, b}\right\|_{L^{2}_{T}L^{2}_{x}}^{2}\leq c^{*}_{2,2},
\end{split}
\end{equation}
where  ${\displaystyle c^{*}_{2,2}=c^{*}_{2,2}\left(\alpha; \epsilon; T;v; \|u_{0}\|_{s_{\alpha},2}; \left\|D_{x}^{\frac{\alpha}{2}}\partial_{x}^{2}u_{0}\chi_{\epsilon, b}\right\|_{2}\right)>0.}$ 
%%%%%%%%%%%%%%%%%%%%%%%%%%%%%%%%%%%%%%%%%%%%%%%%%%%%%%%%%%%%%%%%%%%%%%%

At this point, we shall determine the number of steps in the second inductive process that allow us to reach the next integer. For this reason we will consider the following cases:

 %The way to determine it  depends strongly on the dispersion present in the IVP \eqref{e1}. It is expected  that in the cases when  ''higher'' dispersion is present in  the IVP \eqref{e1}, a  less quantity of steps shall be required, and for ''lower'' dispersion  the contrary. However, even when is  evident this idea, the problem on this assertion is determine what is high and what is low.

%The key factor,  is  the smoothing effect obtained in every step. Notice that the local gain of  regularity in every step is $\alpha/2$, so that, if we keep   moving  forward with a step at this rate, lets say $j$ times, then we obtain    $j\alpha/2-$ local derivatives. The main problem is to  determine the optimal number $j,$ in order to avoid extra steps in the inductive process. For such reason, we will consider the following cases:% we require to apply an extra step of $1-\alpha/2-$derivatives  to  attain the next integer. To do this, we require that 
%$j\alpha/2>1-\alpha/2,$ thus, the parameter dealing with the dispersion shall satisfy the inequality  $\alpha>\frac{2}{j+1}.$ Although, to avoid repetitions  we will consider the following 
%localization:
\begin{itemize}
	\item[(a)] If $\frac{2}{2k+1}\leq\alpha<\frac{1}{k}$ for some positive integer $k,$ then are required $2k+1$ steps in the second inductive process.
	\item[(b)] Instead, if $ \frac{1}{k+1}\leq \alpha<\frac{2}{2k+1}$ for some positive integer $k,$ then  are required $2k+2$ steps in the second inductive process.
\end{itemize}
A detailed description of the number of steps can be obtained directly from (a)-(b).More precisely, there are required $\ceil{\frac{2}{\alpha}}$ steps in both cases, where $\ceil{\cdot}$ denotes the ceiling function.  %This is: $2k+1=\ceil{\frac{2}{\alpha}}$ steps in case (a) and $2k+2=\ceil{\frac{2}{\alpha}}$ in case (b).
 The different considerations are made to differentiate when the number of steps   are even or odd.  %decomposition above  describe the cases when are required an even or odd quantity of steps. %Actually, this allow us to describe precisely $k$ in every case.
 A graphical description of the process described here is presented below. 

Henceforth, for comfort in the notation  we will consider $\alpha$ satisfying the condition (a).
 
As part of the  second inductive process  we shall assume that 
\begin{equation}\label{eq6}
\begin{split}
&\sup_{0\leq t\leq  T}\left\|\partial_{x}^{2}D_{x}^{\frac{\alpha j}{2}}u\chi_{\epsilon, b}(\cdot+vt)\right\|_{2}^{2}+\left\|D_{x}^{2+\alpha\left(\frac{j+1}{2}\right)}u\eta_{\epsilon, b}\right\|_{L^{2}_{T}L^{2}_{x}}^{2}+\left\|\mathcal{H}D_{x}^{2+\alpha\left(\frac{j+1}{2}\right)}u\eta_{\epsilon, b}\right\|_{L^{2}_{T}L^{2}_{x}}^{2}\leq c^{*}_{2,j+1}
\end{split}
\end{equation}
for every $\epsilon>0,b\geq 5\epsilon,\, v\geq 0,$
 and $j=0,1,\cdots, 2k-1.$ 
 
 In fact, the previous cases in the induction process are summarized in the following diagram:
  \begin{center}
 	{\small
 		\begin{tikzcd}
 		\partial_{x}^{2}u\chi_{\epsilon, b}^{2} \arrow[r, blue] \arrow[d,"D_{x}^{\alpha/2}"] & D_{x}^{\alpha/2}\partial_{x}^{2}u(\chi_{\epsilon, b}^{2})' \arrow[dl,dashrightarrow, red]\\
 		D_{x}^{\alpha/2}\partial_{x}^{2}u\chi_{\epsilon, b}^{2} \arrow[r, blue]\arrow[d,"D_{x}^{\alpha/2}"]& D_{x}^{\alpha}\partial_{x}^{2}u(\chi_{\epsilon, b}^{2})'\arrow[dl,dashrightarrow, red]\\
 		D_{x}^{\alpha}\partial_{x}^{2}u\chi_{\epsilon, b}^{2} \arrow[r, blue]\arrow[d,"D_{x}^{\alpha/2}"]& D_{x}^{3\alpha/2}\partial_{x}^{2}u(\chi_{\epsilon, b}^{2})'  \arrow[dl,dashrightarrow, red]\\
 		D_{x}^{3\alpha/2}\partial_{x}^{2}u\chi_{\epsilon, b}^{2}\arrow[d,"D_{x}^{\alpha/2}"]\arrow[r, blue]& D_{x}^{2\alpha}\partial_{x}^{2}u(\chi_{\epsilon, b}^{2})'\arrow[dl,dashrightarrow, red]  \\
 		\vdots\arrow[d,"D_{x}^{\alpha/2}"] & \vdots\\
 		D_{x}^{(2k-1)\alpha/2}\partial_{x}^{2}u\chi_{\epsilon, b}^{2}\arrow[r, blue]\arrow[d,"D_{x}^{1-\alpha/2}"]& D_{x}^{\alpha k}\partial_{x}^{2}u(\chi_{\epsilon, b}^{2})'\arrow[dl,dashrightarrow, red]\\
 		D_{x}^{1-\alpha/2}\partial_{x}^{2}u\chi_{\epsilon, b}^{2}\arrow[r, blue]& \partial_{x}^{3}u(\chi_{\epsilon, b}^{2})'\\
 		\end{tikzcd}}
 \end{center}
 The last before the last  case in the diagram is the one that we will carry out in the next step, we will present all the issues corresponding to this case. 
	\begin{flushleft}
		\fbox{{\sc Step $2k+1$:}}
	\end{flushleft}
A standard argument lead us to the energy identity
%\begin{equation*}
%\begin{split}
%&\partial_{x}^{2}D_{x}^{1-\frac{\alpha}{2}}\partial_{t}u\,\partial_{x}^{2}D_{x}^{1-\frac{\alpha}{2}}u\chi_{\epsilon, b}^{2}-\partial_{x}^{2}D_{x}^{1-\frac{\alpha}{2}}\partial_{x}D_{x}^{\alpha}u\partial_{x}^{2}D_{x}^{1-\frac{\alpha}{2}}u\chi_{\epsilon, b}^{2}\\
%&+\partial_{x}^{2}D_{x}^{1-\frac{\alpha}{2}}\left(u\partial_{x}u\right)\partial_{x}^{2}D_{x}^{1-\frac{\alpha}{2}}u\chi_{\epsilon, b}^{2}=0
%\end{split}
%\end{equation*}
%Later, we integrate in the $x-$variable to obtain the energy identity
\begin{equation}\label{kdv8}
\begin{split}
&\frac{1}{2}\frac{\mathrm{d}}{\mathrm{d}t}\int_{\mathbb{R}}\left(\partial_{x}^{2}D_{x}^{1-\frac{\alpha}{2}}u\right)^{2}\chi_{\epsilon, b}^{2}\mathrm{d}x\underbrace{-\frac{v}{2}\int_{\mathbb{R}}\left(\partial_{x}^{2}D_{x}^{1-\frac{\alpha}{2}}u\right)^{2}(\chi_{\epsilon, b}^{2})'\,\mathrm{d}x}_{B_{1}(t)}\\
&\underbrace{-\int_{\mathbb{R}}\left(\partial_{x}^{2}D_{x}^{1-\frac{\alpha}{2}}D_{x}^{\alpha}\partial_{x}u\right)\partial_{x}^{2}D_{x}^{1-\frac{\alpha}{2}}u\chi_{\epsilon, b}^{2}\,\mathrm{d}x}_{B_{2}(t)}\\
&+\underbrace{\int_{\mathbb{R}} \left(\partial_{x}^{2}D_{x}^{1-\frac{\alpha}{2}}\left(u\partial_{x}u\right)\right)\partial_{x}^{2}D_{x}^{1-\frac{\alpha}{2}}u\chi_{\epsilon, b}^{2}\,\mathrm{d}x}_{B_{3}(t)}=0.
\end{split}
\end{equation}
Notice that estimating  the term $\|B_{1}\|_{2}$  is straightforward in the previous cases, because it is a direct consequence   from the former cases. However, this  cannot be done in this step, instead a new approach is required.

Since the smoothing effect obtained in the case $j=2k-1$  is $2+ \alpha k,$ then by hypothesis
\begin{equation}\label{eq9}
2+\alpha k\geq 3-\frac{\alpha}{2}\quad\mbox{for}\quad  \alpha\in \left[\frac{2}{2k+1},\frac{1}{k}\right).
\end{equation} 
Therefore, the regularity obtained in  \eqref{eq6} is enough to provide a bound for $\|B_{1}\|_{1}.$
To  do this, we first decompose  the term $\partial_{x}^{2}D_{x}^{\alpha k}u$ as follows:
\begin{equation}\label{kdv7}
\begin{split}
\partial_{x}^{2}D_{x}^{\alpha k}\left(u\eta_{\epsilon,b}\right)&=\eta_{\epsilon,b}\partial_{x}^{2}D_{x}^{\alpha k}u-\left[D_{x}^{2+\alpha k};\eta_{\epsilon,b}\right]\left(u\chi_{\epsilon, b}+u\phi_{\epsilon, b}+u\psi_{\epsilon}\right).
\end{split}
\end{equation}
Then, by Lemma \ref{kdvlem1}
\begin{equation}\label{eq10}
\begin{split}
\left\|\partial_{x}^{2}D_{x}^{\alpha k}\left(u\eta_{\epsilon,b}\right)\right\|_{2}&\leq \left\|\partial_{x}^{2}D_{x}^{\alpha k}u\eta_{\epsilon,b}\right\|_{2}+\left\|\left[D_{x}^{2+\alpha k};\eta_{\epsilon,b}\right]\left(u\chi_{\epsilon, b}+u\phi_{\epsilon, b}+u\psi_{\epsilon}\right)\right\|_{2}\\
&\lesssim \left\|\partial_{x}^{2}D_{x}^{\alpha k}u\eta_{\epsilon,b}\right\|_{2}+ \|\eta_{\epsilon,b}\|_{l,2}\left(\left\|u\chi_{\epsilon, b}\right\|_{1+\alpha k, 2}+ \left\|u\phi_{\epsilon, b}\right\|_{1+\alpha k,2}\right)\\
&\quad +\left\|\eta_{\epsilon,b}D_{x}^{2+\alpha k}(u\psi_{\epsilon})\right\|_{2}\\
&=I_{1}+I_{2}+I_{3}+I_{4}.
\end{split}
\end{equation}
Notice that after integrate in time, the  inequality \eqref{eq6} implies that
\begin{equation*}
\|I_{1}\|_{L^{2}_{T}}=\left(\int_{0}^{T}\int_{\mathbb{R}}\left(D_{x}^{2+\alpha k}u\right)^{2}\eta_{\epsilon,b}^{2}\,\mathrm{d}x\,\mathrm{d}t\right)^{1/2}\lesssim (c^{*}_{2,j+1})^{1/2}.
\end{equation*}
The terms $I_{2}$ and $I_{3}$  can be bounded   via  Lemma \ref{lema2} and Young's inequality 
\begin{equation*}
\begin{split}
\left\|D_{x}^{1+\alpha k}(u\chi_{\epsilon, b})\right\|_{2}&\lesssim \left\|\partial_{x}^{2}(u\chi_{\epsilon, b})\right\|_{2}+\|u_{0}\|_{2}\\
&\lesssim \left\|\partial_{x}^{2}u\chi_{\epsilon, b}\right\|_{2}+\|u\|_{s_{\alpha},2},
\end{split}
\end{equation*}
and 
\begin{equation*}
\begin{split}
\left\|D_{x}^{1+\alpha k}(u\phi_{\epsilon, b})\right\|_{2}&\lesssim \left\|\partial_{x}^{2}(u\phi_{\epsilon, b})\right\|_{2}+\|u_{0}\|_{2}\\
&\lesssim \left\|\partial_{x}^{2}u\phi_{\epsilon, b}\right\|_{2}+\|u\|_{s_{\alpha},2}.
\end{split}
\end{equation*}
Moreover,  
 \begin{equation*}
\|I_{2}\|_{L^{2}_{T}}=\left\|D_{x}^{1+\alpha k}(u\chi_{\epsilon, b})\right\|_{L^{2}_{T}L^{2}_{x}}\lesssim  \sup_{0\leq t \leq T}\left\|\partial_{x}^{2}u\chi_{\epsilon, b}\right\|_{2}+\|u\|_{L^{\infty}_{T}H^{s_{\alpha}}_{x}}.
\end{equation*}
  On the other hand,  there exists $\tau>0$ such that
\begin{equation*}
\phi_{\epsilon, b}(x+vt)\lesssim \mathbb{1}_{[-\tau, \tau]}(x)\qquad \mbox{for all}\quad (x,t)\in \mathbb{R}\times [0,T],
\end{equation*}
Thus, in view of Remark \ref{kse}
\begin{equation*}
\begin{split}
\left\|\partial_{x}^{2}u\phi_{\epsilon, b}\right\|_{L^{2}_{T}L^{2}_{x}}&\lesssim \left(\int_{0}^{T}\int_{\mathbb{R}} \left(\partial_{x}^{2}u\right)^{2}\phi_{\epsilon, b}^{2}\mathrm{d}x\,\mathrm{d}t\right)^{1/2} \\
&\lesssim\left(\int_{0}^{T}\int_{-\tau
}^{\tau} \left(\partial_{x}^{2}u\right)^{2}\mathrm{d}x\,\mathrm{d}t\right)^{1/2}\\
&\leq c\left(\|u_{0}\|_{s_{\alpha},2};\tau\right).
\end{split}
\end{equation*}
Next, by  Lemma \ref{lemma1}, it follows that 
\begin{equation}\label{eq11}
\left\|\eta_{\epsilon,b}D_{x}^{2+\alpha k}(u\psi_{\epsilon})\right\|_{2}\lesssim \|u_{0}\|_{2}.
\end{equation}
Gathering together the estimates above lead us to  
\begin{equation*}
\left\|D_{x}^{2+\alpha k}(u\eta_{\epsilon,b})\right\|_{L^{2}_{T}L^{2}_{x}}<\infty.
\end{equation*}
This inequality combined with    Lemma \ref{lema2} implies that 
\begin{equation}\label{eq7}
	\left \|D_{x}^{3-\frac{\alpha}{2}}\left(u\eta_{\epsilon,b}\right)\right\|_{2}\lesssim \left\|D_{x}^{2+\alpha k}(u\eta_{\epsilon,b})\right\|_{2}^{\frac{6-\alpha}{2(2+\alpha k)}}\|u\eta_{\epsilon,b}\|_{2}^{\frac{\alpha(1+2k)-2}{2(2+\alpha k)}}.
	\end{equation} 

Since we need to estimate $\partial_{x}^{2}D_{x}^{1-\frac{\alpha}{2}}u\eta_{\epsilon,b}, $ a decomposition as that in \eqref{kdv7} is sufficient to our proposes, i.e. 
\begin{equation*}
\partial_{x}^{2}D_{x}^{1-\frac{\alpha}{2}}u\eta_{\epsilon,b}=\partial_{x}^{2}D_{x}^{1-\frac{\alpha}{2}}(u\eta_{\epsilon,b})-\left[\partial_{x}^{2}D_{x}^{1-\frac{\alpha}{2}}; \eta_{\epsilon,b}\right]\left(u\chi_{\epsilon, b}+u\phi_{\epsilon, b}+u\psi_{\epsilon}\right).
\end{equation*}
For the sake of brevity, we omit the computations  behind these commutators estimates. Nevertheless,  similar arguments as those used in  \eqref{eq10}-\eqref{eq11} lead to 
\begin{equation*}
\left\|\partial_{x}^{2}D_{x}^{1-\frac{\alpha}{2}}u\eta_{\epsilon,b}\right\|_{L^{2}_{T}L^{2}_{x}}<\infty.
\end{equation*}
 The estimates above are in   fact a way to  bound the term $\|B_{1}\|_{1}$ in \eqref{kdv8}.  
 
 More precisely,
\begin{equation*}
\begin{split}
\int_{0}^{T}|B_{1}(t)|\,\mathrm{d}t=v\left\|\partial_{x}^{2}D_{x}^{1-\frac{\alpha}{2}}u\eta_{\epsilon,b}\right\|_{L^{2}_{T}L^{2}_{x}}^{2}<\infty.
\end{split}
\end{equation*}
%%%%%%%%%%%%%%%%%%%%%%%%%
\S.2\quad A previous argument (see step 1) implies that  
%\begin{equation*}
%\begin{split}
%B_{2}(t)&=\frac{1}{2}\int_{\mathbb{R}} \partial_{x}^{2}D_{x}^{1-\frac{\alpha}{2}}u\left[D_{x}^{\alpha}\partial_{x}; \chi_{\epsilon, b}^{2}\right]\partial_{x}^{2}D_{x}^{1-\frac{\alpha}{2}}u\,\mathrm{d}x\\
%&=-\frac{1}{2}\int_{\mathbb{R}} D_{x}^{3-\frac{\alpha}{2}}u\left[\mathcal{H}D_{x}^{\alpha+1}; \chi_{\epsilon, b}^{2}\right]D_{x}^{3-\frac{\alpha}{2}}u\,\mathrm{d}x.
%\end{split}
%\end{equation*}
%The commutator decomposition \eqref{} allow us to obtain 
\begin{equation}\label{kdv3}
\begin{split}
B_{2}(t)&=\frac{1}{2}\int_{\mathbb{R}}D_{x}^{3-\frac{\alpha}{2}}uR_{n}(\alpha+1)D_{x}^{3-\frac{\alpha}{2}}u\,\mathrm{d}x+\frac{1}{4}\int_{\mathbb{R}}D_{x}^{3-\frac{\alpha}{2}}uP_{n}(\alpha+1)D_{x}^{3-\frac{\alpha}{2}}u\,\mathrm{d}x\\
&\quad -\frac{1}{4}\int_{\mathbb{R}}D_{x}^{3-\frac{\alpha}{2}}u\mathcal{H}P_{n}(\alpha+1)\mathcal{H}D_{x}^{3-\frac{\alpha}{2}}u\,\mathrm{d}x\\
&=B_{2,1}(t)+B_{2,2}(t)+B_{2,3}(t).
\end{split}
\end{equation}
A similar argument to the used in step 1 allow us  to fix $n=2$ above. 

According to Proposition \ref{propo2}, the remainder term $R_{2}(\alpha+1)$ is a bounded operator in $L^{2}(\mathbb{R}).$ Therefore 
\begin{equation*}
\begin{split}
B_{2,1}(t)&=\frac{1}{2}\int_{\mathbb{R}}uD_{x}^{3-\frac{\alpha}{2}}R_{2}(\alpha+1)D_{x}^{3-\frac{\alpha}{2}}u\,\mathrm{d}x,
\end{split}
\end{equation*}
and by  H\"{o}lder's inequality  
\begin{equation*}
\begin{split}
|B_{2,1}(t)|\lesssim \|u_{0}\|_{2}\left\|D_{x}^{3-\frac{\alpha}{2}}R_{2}(\alpha+1)D_{x}^{3-\frac{\alpha}{2}}u\right\|_{2}\lesssim \|u_{0}\|_{2}^{2}\left\|\widehat{D_{x}^{7}\left(\chi_{\epsilon, b}^{2}\right)}\right\|_{1}.
\end{split}
\end{equation*}
Thus,
\begin{equation*}
\int_{0}^{T}|B_{2,1}(t)|\,\mathrm{d}t\lesssim T\|u_{0}\|_{2}^{2}\left\|\widehat{D_{x}^{7}\left(\chi_{\epsilon, b}^{2}\right)}\right\|_{1}.
\end{equation*}
If we decompose the terms in \eqref{kdv3} it follows that  
\begin{equation*}
\begin{split}
B_{2,2}(t)&=\left(\frac{\alpha+1}{4}\right)\int_{\mathbb{R}}\left(\mathcal{H}\partial_{x}^{3}u\right)^{2}\left(\chi_{\epsilon, b}^{2}\right)'\mathrm{d}x-c_{3}\left(\frac{\alpha+1}{16}\right)\int_{\mathbb{R}} \left(\partial_{x}^{2}u\right)^{2}\left(\chi_{\epsilon, b}^{2}\right)^{(3)}\,\mathrm{d}x\\
&\quad +c_{5}\left(\frac{\alpha+1}{64}\right)\int_{\mathbb{R}}\left(\mathcal{H}\partial_{x}u\right)^{2}\left(\chi_{\epsilon, b}^{2}\right)^{(5)}\,\mathrm{d}x\\
&=B_{2,2,1}(t)+B_{2,2,2}(t)+B_{2,2,3}(t),
\end{split}
\end{equation*}
and
\begin{equation*}
\begin{split}
B_{2,3}(t)&=\left(\frac{\alpha+1}{4}\right)\int_{\mathbb{R}}\left(\partial_{x}^{3}u\right)^{2}\left(\chi_{\epsilon, b}^{2}\right)'\mathrm{d}x-c_{3}\left(\frac{\alpha+1}{16}\right)\int_{\mathbb{R}} \left(\mathcal{H}\partial_{x}^{2}u\right)^{2}\left(\chi_{\epsilon, b}^{2}\right)^{(3)}\,\mathrm{d}x\\
&\quad +c_{5}\left(\frac{\alpha+1}{64}\right)\int_{\mathbb{R}}\left(\partial_{x}u\right)^{2}\left(\chi_{\epsilon, b}^{2}\right)^{(5)}\,\mathrm{d}x\\
&=B_{2,3,1}(t)+B_{2,3,2}(t)+B_{2,3,3}(t).
\end{split}
\end{equation*}
As before $B_{2,2,1}$ and $B_{2,3,1}$  provide  the smoothing effect after integration in time.

The terms $B_{2,2,2}$ and $B_{2,3,2}$  are easily handled by using Proposition \ref{kse} and  the arguments used in \eqref{kdv5}-\eqref{kdv6}. Meanwhile,   the low regularity in the terms $B_{2,2,3}$ and $B_{2,3,3}$ is handled by using local theory.

\S.3 \quad To finish this step, we turn our attention to the term involving the nonlinear part of the equation 
\begin{equation*}
\begin{split}
\partial_{x}^{2}D_{x}^{1-\frac{\alpha}{2}}\left(u\partial_{x}u\right)\chi_{\epsilon,b }&=-D_{x}^{3-\frac{\alpha }{2}}(u\partial_{x}u)\chi_{\epsilon,b }\\
&=\frac{1}{2}\left[D_{x}^{3-\frac{\alpha }{2}};\chi_{\epsilon, b}\right]\partial_{x}\left((u\chi_{\epsilon, b})^{2}+(u\widetilde{\phi_{\epsilon, b}})^{2}+u^{2}\psi_{\epsilon}\right)\\
&\quad - \left[D_{x}^{3-\frac{\alpha}{2}};u\chi_{\epsilon, b}\right]\partial_{x}\left(u\chi_{\epsilon, b}+u\phi_{\epsilon, b}+u\psi_{\epsilon}\right)+u\chi_{\epsilon, b}\partial_{x}^{2}D_{x}^{1-\frac{\alpha }{2}}\partial_{x}u\\
&=\widetilde{B_{3,1}}(t)+\widetilde{B_{3,2}}(t)+\widetilde{B_{3,3}}(t)+\widetilde{B_{3,4}}(t)+\widetilde{B_{3,5}}(t)+\widetilde{B_{3,6}}(t)+\widetilde{B_{3,7}}(t).
\end{split}
\end{equation*} 
In the decomposition above is sufficient to  estimate the $L^{2}_{x}-$ norm of the terms $\widetilde{B_{3,l}}$ for $l=1,2,\cdots,6.$ 

Combining  \eqref{eq5}   and Lemma \ref{kdvlem1} it follows that 
\begin{equation*}
\begin{split}
\|\widetilde{B_{3,1}}\|_{2}&\lesssim \left\|\chi_{\epsilon,b}'\right\|_{l,2}\left\|\left(u\chi_{\epsilon, b}\right)^{2}\right\|_{3-\frac{\alpha}{2},2}\\
&\lesssim \left\|u^{2}\right\|_{2}+\left\|D_{x}^{3-\frac{\alpha }{2}}\left(\left(u\chi_{\epsilon, b}\right)^{2}\right)\right\|_{2}\\
&\lesssim  \|u\|_{\infty}\left(\|u_{0}\|_{2}+\left\|D_{x}^{3-\frac{\alpha}{2}}(u\chi_{\epsilon, b})\right\|_{2}\right)
\end{split}
\end{equation*}
and 
\begin{equation*}
\begin{split}
\|\widetilde{B_{3,2}}\|_{2}&% \lesssim \left\|\chi_{\epsilon,b}'\right\|_{l,2}\left\|\left(u\phi_{\epsilon, b}\right)^{2}\right\|_{3-\frac{\alpha}{2},2}\\
%&\lesssim \left\|u^{2}\right\|_{2}+\left\|D_{x}^{3-\frac{\alpha }{2}}\left(\left(u\phi_{\epsilon, b}\right)^{2}\right)\right\|_{2}\\
\lesssim  \|u\|_{\infty}\left(\|u_{0}\|_{2}+\left\|D_{x}^{3-\frac{\alpha}{2}}(u\phi_{\epsilon, b})\right\|_{2}\right).
\end{split}
\end{equation*}
Since   $\chi_{\epsilon, b}$ and $\psi_{\epsilon}$ satisfy
\begin{equation*}
\dist\left(\supp(\chi_{\epsilon, b}),\supp(\psi_{\epsilon})\right)\geq \frac{\epsilon}{2}>0,
\end{equation*}
then  by  Lemma \ref{lemma1}, it follows that  
\begin{equation*}
\begin{split}
\|\widetilde{B_{3,3}}\|_{2}&=\left\|\chi_{\epsilon, b}D_{x}^{3-\frac{\alpha }{2}}(u^{2}\psi_{\epsilon})\right\|_{2}\lesssim \|u\|_{\infty}\|u_{0}\|_{2},
\end{split}
\end{equation*}
and 
\begin{equation*}
\begin{split}
\|\widetilde{B_{3,6}}\|_{2}&=\left\|u\chi_{\epsilon, b}D_{x}^{3-\frac{\alpha }{2}}(u\psi_{\epsilon})\right\|_{2}\lesssim\|u_{0}\|_{2}\|u\|_{\infty}.
\end{split}
\end{equation*}
Instead, the terms $\widetilde{B_{3,4}}$ and $\widetilde{B_{3,1}}$ are handled by using  \eqref{kpdl}  
\begin{equation*}
\begin{split}
\|\widetilde{B_{3,4}}\|_{2}\lesssim \left\|D_{x}^{3-\frac{\alpha }{2}}(u\chi_{\epsilon, b})\right\|_{2}\left\|\partial_{x}(u\chi_{\epsilon, b})\right\|_{\infty}\\
\end{split}
\end{equation*}
and 
\begin{equation*}
\|\widetilde{B_{3,5}}\|_{2}\lesssim \left\|D_{x}^{3-\frac{\alpha }{2}}(u\phi_{\epsilon, b})\right\|_{2}\left\|\partial_{x}(u\chi_{\epsilon, b})\right\|_{\infty}+\left\|D_{x}^{3-\frac{\alpha }{2}}(u\chi_{\epsilon,b })\right\|_{2}\left\|\partial_{x}(u\phi_{\epsilon, b})\right\|_{\infty}.
\end{equation*}

Notice that 
\begin{equation}
\begin{split}
\partial_{x}^{2}D_{x}^{1-\frac{\alpha}{2}}(u\chi_{\epsilon, b})&=\chi_{\epsilon, b}\partial_{x}^{2}D_{x}^{1-\frac{\alpha}{2}}u- \left[D_{x}^{3-\frac{\alpha}{2}};\chi_{\epsilon,b }\right]\left(u\chi_{\epsilon, b}+u\phi_{\epsilon,b}+u\psi_{\epsilon}\right),
\end{split}
\end{equation}
so that a combination of interpolation, local theory and Lemma \ref{kdvlem1} yield
\begin{equation*}
\begin{split}
\left\|\partial_{x}^{2}D_{x}^{1-\frac{\alpha}{2}}(u\chi_{\epsilon, b})\right\|_{2} %&\leq \left\|\chi_{\epsilon,b }\partial_{x}^{2}D_{x}^{1-\frac{\alpha}{2}}u\right\|_{2}+\left\|\left[D_{x}^{3-\frac{\alpha}{2}};\chi_{\epsilon,b }\right]\left(u\chi_{\epsilon, b}+u\phi_{\epsilon,b}+u\psi_{\epsilon}\right)\right\|_{2}\\
%&\lesssim \left\|\chi_{\epsilon,b }\partial_{x}^{2}D_{x}^{1-\frac{\alpha}{2}}u\right\|_{2}+ \left\|\chi_{\epsilon, b}'\right\|_{l,2}\left(\|u\chi_{\epsilon, b}\|_{2-\frac{\alpha}{2},2}+\|u\phi_{\epsilon,b}\|_{2-\frac{\alpha}{2},2}\right)\\
%&\quad  +\left\|\chi_{\epsilon,b }D_{x}^{3-\frac{\alpha}{2}}\left(u\psi_{\epsilon}\right)\right\|_{2}\\
%&\lesssim \left\|\chi_{\epsilon,b }\partial_{x}^{2}D_{x}^{1-\frac{\alpha}{2}}u\right\|_{2}+ \|u_{0}\|_{2}+ \left\|D_{x}^{2-\frac{\alpha}{2}}(u\chi_{\epsilon, b})\right\|_{2} +\left\|D_{x}^{2-\frac{\alpha}{2}}(u\phi_{\epsilon,b})\right\|_{2}\\
%&\lesssim \left\|\chi_{\epsilon,b }\partial_{x}^{2}D_{x}^{1-\frac{\alpha}{2}}u\right\|_{2}+\left\|\partial_{x}^{2}(u\chi_{\epsilon, b})\right\|_{2} +\left\|\partial_{x}^{2}(u\phi_{\epsilon,b})\right\|_{2}+\|u_{0}\|_{2}\\
&\lesssim \left\|\chi_{\epsilon,b }\partial_{x}^{2}D_{x}^{1-\frac{\alpha}{2}}u\right\|_{2}+\left\|\partial_{x}^{2}u\chi_{\epsilon,b}\right\|_{2}+\left\|\partial_{x}^{2}u\phi_{\epsilon,b}\right\|_{2}+\|u\|_{s_{\alpha},2}.
\end{split}
\end{equation*}
An immediate application of Theorem \ref{thm11} implies that
\begin{equation*}
\begin{split}
&\left\|D_{x}^{2+\alpha k}(u\phi_{\epsilon, b})\right\|_{2}\\
&\lesssim \left\|D_{x}^{2+\alpha k
}\phi_{\epsilon, b}\right\|_{\mathrm{BMO}}\|u\|_{2}+\left\|\phi_{\epsilon, b}D_{x}^{2+\alpha k}u\right\|_{2}+ \left\|\partial_{x}\phi_{\epsilon, b}\mathcal{H}D_{x}^{1+\alpha k}u\right\|_{2}+\left\|\partial_{x}^{2}\phi_{\epsilon, b}D_{x}^{\alpha k}u\right\|_{2}\\
&\lesssim \|u_{0}\|_{2}+\left\|\chi_{\epsilon/8,b+\epsilon/4}'D_{x}^{2+\alpha k}u\right\|_{2}+ \left\|\chi_{\epsilon/8,b+\epsilon/4}'\mathcal{H}D_{x}^{1+\alpha k}u\right\|_{2}+\left\|\chi_{\epsilon/8,b+\epsilon/4}'D_{x}^{\alpha k}u\right\|_{2}\\
&=A_{1}+A_{2}+A_{3}+A_{4}.
\end{split}
\end{equation*}
The first and fourth term on the right hand side above are bounded after integrate in time. The second and third term require extra arguments to  control them after integrate in time. 

Since  $\alpha$ satisfies   
\begin{equation}
\frac{2}{2k+1}\leq \alpha<\frac{1}{k}\qquad  \mbox{for}\quad k\in \mathbb{Z}^{+},
\end{equation}
 it is  clear that $\alpha k$ is positive and strictly less than one for all positive integer $k,$ then is straightforward to see that   $\|A_{4}\|_{1}<\infty.$  
 
 Interpolation combined with  Calderon's commutator estimate \eqref{eq31} implies that 
\begin{equation}
\begin{split}
A_{3}&=\left\|\partial_{x}\phi_{\epsilon, b}\mathcal{H}D_{x}^{1+\alpha k}u\right\|_{2}\\
&=\left\|\mathcal{H}\left(\partial_{x}\phi_{\epsilon,b}D_{x}^{1+\alpha k}u\right)
-\left[\mathcal{H};\partial_{x}\phi_{\epsilon, b}\right]D_{x}^{1+\alpha k}u\right\|_{2}\\
%&\lesssim \left\|\partial_{x}\phi_{\epsilon,b}D_{x}^{1+\alpha k}u\right\|_{2}+\left\|\partial_{x}^{2}\phi_{\epsilon,b}\right\|_{\infty}\left\|D_{x}^{\alpha k}u\right\|_{2}\\
&\lesssim\left\|\partial_{x}\phi_{\epsilon,b}D_{x}^{1+\alpha k}u\right\|_{2}+\left\|u\right\|_{s_{\alpha},2},
\end{split}
\end{equation}
and 
\begin{equation}
\begin{split}
\left\|\partial_{x}\phi_{\epsilon,b}D_{x}^{1+\alpha k}u\right\|_{2}&\lesssim \left\|\partial_{x}^{2}\left(u\partial_{x}\phi_{\epsilon,b}\right)\right\|_{2}+\|u\|_{s_{\alpha,2}}\\
&\lesssim \left\|\partial_{x}^{2}u\chi_{\epsilon/8,b+\epsilon/4}'\right\|_{2}+\|u\|_{s_{\alpha,2}}.
%&\lesssim c\left(\|u_{0}\|_{s^{\alpha},2};\epsilon;T;v\right)+\|u\|_{s_{\alpha,2}}.
\end{split}
\end{equation}
The last inequality above is obtained combining interpolation and the properties of the weighted functions.

Hence,
\begin{equation}
\|A_{3}\|_{1}\lesssim c\left(\|u_{0}\|_{s_{\alpha},2};\epsilon;T;v\right)+\|u\|_{L^{\infty}_{T}H^{s_{\alpha}}_{x}}.
\end{equation}
Next, the properties of weighted functions and \eqref{eq6} imply 
\begin{equation}
\|A_{2}\|_{1}\lesssim T^{1/2}\,\left(c^{*}_{2, 2k}\right)^{1/2}.
\end{equation} 
Applying similar arguments   allow us to  estimate  $\left\|D_{x}^{2+\alpha k}\left(
u\widetilde{\phi_{\epsilon, b}}\right)\right\|_{2}.$ 

Finally,  integration by parts give us 
\begin{equation*}
\begin{split}
B_{3,7}(t)&= -\frac{1}{2}\int_{\mathbb{R}}\partial_{x}u\chi_{\epsilon,b }^{2}\left(\partial_{x}^{2}D_{x}^{1-\frac{\alpha }{2}}u\right)^{2}\mathrm{d}x-\frac{1}{2}\int_{\mathbb{R}}u\left(\chi_{\epsilon,b }^{2}\right)'\left(\partial_{x}^{2}D_{x}^{1-\frac{\alpha }{2}}u\right)^{2}\mathrm{d}x\\
&=B_{3,7,1}(t)+B_{3,7,2}(t).
\end{split}
\end{equation*}
Observe that
\begin{equation*}
\begin{split}
|B_{3,7,1}(t)|&\lesssim \|\partial_{x}u(t)\|_{\infty}\int_{\mathbb{R}}\left(\partial_{x}^{2}D_{x}^{1-\frac{\alpha }
	{2}}u\right)^{2}\chi_{\epsilon, b}^{2}\mathrm{d}x,
\end{split}
\end{equation*}
By the local theory    $\partial_{x}u \in L^{1}_{T}L^{\infty}_{x},$ and 
the integral expression will be estimated by means of Gronwall's inequality.

By Sobolev's embedding it follows that  
\begin{equation*}
\begin{split}
|B_{3,7,2}(t)|&\lesssim \|u(t)\|_{\infty}\int_{\mathbb{R}}\left(\chi_{\epsilon, b}^{2}\right)'\left(\partial_{x}^{2}D_{x}^{1-\frac{\alpha }{2}}u\right)^{2}\,\mathrm{d}x\\
&\lesssim \left(\sup_{0\leq t\leq  T}\|u(t)\|_{s_{\alpha},2}\right)\int_{\mathbb{R}}\left(\chi_{\epsilon, b}^{2}\right)'\left(\partial_{x}^{2}D_{x}^{\frac{\alpha}{2}}u\right)^{2}\,\mathrm{d}x
\end{split}
\end{equation*}
and noticing  that after integrating  in time 
we obtain 
\begin{equation*}
\int_{0}^{T}|B_{3,7,2}(t)|\,\mathrm{d}t\lesssim \left(\sup_{0\leq t\leq  T}\|u(t)\|_{s_{\alpha},2}\right)\int_{0}^{T}\int_{\mathbb{R}}\left(\chi_{\epsilon, b}^{2}\right)'\left(\partial_{x}^{2}D_{x}^{\frac{\alpha }{2}}u\right)^{2}\mathrm{d}x\,\mathrm{d}t,
\end{equation*} 
where the integral expression  corresponds to the $B_{1}$ term,  which was already estimated in \eqref{eq1.4}.

 We conclude this  step gathering the  estimates corresponding to $B_{1},B_{2}$ and $B_{3}$ combined with   Gronwall's inequality  and integration in time  to obtain for any $\epsilon>0, b\geq 5\epsilon$ and $v\geq 0,$
\begin{equation}\label{se2}
\begin{split}
&\sup_{0\leq t\leq T}\left\|\partial_{x}^{2}D_{x}^{1-\frac{\alpha }{2}}u\chi_{\epsilon, b}(\cdot+vt)\right\|_{2}^{2}+\left\|\partial_{x}^{3}u\eta_{\epsilon, b}^{2}\right\|_{L^{2}_{T}L^{2}_{x}}^{2}+\left\|\mathcal{H}\partial_{x}^{3}u\eta_{\epsilon, b}\right\|_{L^{2}_{T}L^{2}_{x}}^{2}\leq c^{*}_{2,2k},
\end{split}
\end{equation}
where  ${\displaystyle c^{*}_{2,2k+1}=c^{*}_{2,2k+1}\left(\alpha; \epsilon; T;v; \|u_{0}\|_{s_{\alpha},2}; \left\|D_{x}^{1-\frac{\alpha}{2}}\partial_{x}^{2}u_{0}\chi_{\epsilon, b}\right\|_{2}\right)>0.}$

The inequality above finishes the  first induction argument.
	\begin{flushleft}
	\textbf{\underline{{\sc Case $m$}}}
\end{flushleft}

Next,  our argument will combine two induction process at the same time to reach our goal. 
 
 The first induction process consists in to  assume   that $m$ derivatives are propagated  to then prove that $m+1$ derivatives are propagated; this is summarized in the diagram below. 
\begin{center}
\begin{tikzcd}
\partial_{x}^{2}u\chi_{\epsilon, b} \arrow[r]
& \partial_{x}^{3}u\chi_{\epsilon, b} \arrow[r]
\arrow[d, phantom, ""{coordinate, name=Z}]
& \partial_{x}^{4}u\chi_{\epsilon, b} \arrow[dll,
"\partial_{x}",
rounded corners,
to path={ -- ([xshift=2ex]\tikztostart.east)
	|- (Z) [near end]\tikztonodes
	-| ([xshift=-2ex]\tikztotarget.west)
	-- (\tikztotarget)}] \\
\partial_{x}^{m-1}u\chi_{\epsilon, b} \arrow[r]
& \partial_{x}^{m}u\chi_{\epsilon, b} \arrow[r]
& \partial_{x}^{m+1}u\chi_{\epsilon, b}
\end{tikzcd}
\end{center}
%Nevertheless, as occurred  to move from $m$ derivatives propagated to $m+1,$ other induction process enter in the game,as was done  in the case $m=2.$

More precisely, it is described in the following figure
 \begin{center}
	{\small
		\begin{tikzcd}
		\partial_{x}^{m}u\chi_{\epsilon, b}^{2} \arrow[r, blue] \arrow[d,"D_{x}^{\alpha/2}"] & D_{x}^{\alpha/2}\partial_{x}^{m}u(\chi_{\epsilon, b}^{2})' \arrow[dl,dashrightarrow, red]\\
		D_{x}^{\alpha/2}\partial_{x}^{m}u\chi_{\epsilon, b}^{2} \arrow[r, blue]\arrow[d,"D_{x}^{\alpha/2}"]& D_{x}^{\alpha}\partial_{x}^{m}u(\chi_{\epsilon, b}^{2})'\arrow[dl,dashrightarrow, red]\\
		D_{x}^{\alpha}\partial_{x}^{m}u\chi_{\epsilon, b}^{2} \arrow[r, blue]\arrow[d,"D_{x}^{\alpha/2}"]& D_{x}^{3\alpha/2}\partial_{x}^{m}u(\chi_{\epsilon, b}^{2})'  \arrow[dl,dashrightarrow, red]\\
		D_{x}^{3\alpha/2}\partial_{x}^{m}u\chi_{\epsilon, b}^{2}\arrow[d,"D_{x}^{\alpha/2}"]\arrow[r, blue]& D_{x}^{2\alpha}\partial_{x}^{m}u(\chi_{\epsilon, b}^{2})'\arrow[dl,dashrightarrow, red]  \\
		\vdots\arrow[d,"D_{x}^{\alpha/2}"] & \vdots\\
		D_{x}^{(2k-1)\alpha/2}\partial_{x}^{m}u\chi_{\epsilon, b}^{2}\arrow[r, blue]\arrow[d,"D_{x}^{1-\alpha/2}"]& D_{x}^{\alpha k}\partial_{x}^{m}u(\chi_{\epsilon, b}^{2})'\arrow[dl,dashrightarrow, red]\\
		D_{x}^{1-\alpha/2}\partial_{x}^{m}u\chi_{\epsilon, b}^{2}\arrow[r, blue]& \partial_{x}^{m+1}u(\chi_{\epsilon, b}^{2})'\\
		\end{tikzcd}}
\end{center}
As part of this induction process, we shall  assume that for any $\epsilon>0,\, b\geq 5\epsilon,$ $v\geq 0,$ the following holds:
\begin{itemize}
\item[(I)] for   $l=2,3,\cdots,m$\, and \, $j=0,1,\cdots,2k-1$
\begin{equation}\label{kdv9}
\begin{split}
&\sup_{0\leq t\leq T}\left\|\partial_{x}^{l}D_{x}^{\frac{\alpha j }{2}}u\chi_{\epsilon, b}(\cdot+vt)\right\|_{2}^{2}+\left\|D_{x}^{\alpha\left(\frac{j+1}{2}\right)
}\partial_{x}^{l}u\eta_{\epsilon, b}\right\|_{L^{2}_{T}L^{2}_{x}}^{2}+\left\|\mathcal{H}D_{x}^{\alpha\left(
	\frac{j+1}{2}\right)}\partial_{x}^{l}u\eta_{\epsilon, b}^{2}\right\|_{L^{2}_{T}L^{2}_{x}}^{2}\leq c^{*}_{n,j}
\end{split}
\end{equation}
where  ${\displaystyle c^{*}_{l,j}=c^{*}_{l,j}\left(\alpha; \epsilon; T;v;j; \|u_{0}\|_{s_{\alpha},2}; \left\|\partial_{x}^{l}D_{x}^{\frac{\alpha j}{2}}u_{0}\chi_{\epsilon, b}\right\|_{2}\right)>0}.$  
\item[(II)] For $l=2,3,\cdots,m-1$
\begin{equation}\label{kdv30}
\begin{split}
&\sup_{0\leq t\leq T}\left\|\partial_{x}^{l}D_{x}^{1-\frac{\alpha  }{2}}u\chi_{\epsilon, b}(\cdot+vt)\right\|_{2}^{2}+\left\|\partial_{x}^{l+1}u\eta_{\epsilon,b}\right\|_{L^{2}_{T}L^{2}_{x}}^{2}+\left\|\mathcal{H}\partial_{x}^{l}u\eta_{\epsilon,b}\right\|_{L^{2}_{T}L^{2}_{x}}^{2}\leq c^{*}_{l,2k},
\end{split}
\end{equation}
 as usual we indicate the dependence of the parameters involved behind  the constant in the right hand side above. More precisely, 
    $${\displaystyle c^{*}_{l,j}=c^{*}_{l,j}\left(\alpha; \epsilon; T;v;j; \|u_{0}\|_{s_{\alpha},2}; \left\|\partial_{x}^{l}D_{x}^{1-\frac{\alpha}{2}}u_{0}\chi_{\epsilon, b}\right\|_{2}\right)>0}.$$  
\end{itemize}
\begin{flushleft}
	\fbox{{\sc Step 2:}}
\end{flushleft}
Once our induction assumptions are fixed, we proceed as before, that is, we  obtain the weighted energy estimates
\begin{equation}\label{fkdv2}
\begin{split}
&\frac{1}{2}\frac{\mathrm{d}}{\mathrm{d}t}\int_{\mathbb{R}}\left(D_{x}^{\frac{\alpha}{2}}\partial_{x}^{m}u\right)^{2}\chi_{\epsilon, b}^{2}\,\mathrm{d}x
\underbrace{-\frac{v}{2}\int_{\mathbb{R}}\left(D_{x}^{\frac{\alpha}{2}}\partial_{x}^{m}u\right)^{2}\left(\chi_{\epsilon, b}^{2}\right)'\mathrm{d}x}_{B_{1}(t)}\\
&\underbrace{-\int_{\mathbb{R}}\left(D_{x}^{\frac{\alpha}{2}}\partial_{x}^{m}D_{x}^{\alpha}\partial_{x}u\right)D_{x}^{\frac{\alpha}{2}}\partial_{x}^{m}u\chi_{\epsilon, b}^{2}\,\mathrm{d}x}_{B_{2}(t)}\\
&
\underbrace{+\int_{\mathbb{R}}\left(D_{x}^{\frac{\alpha}{2}}\partial_{x}^{m}\left(u\partial_{x}u\right)\right)
D_{x}^{\frac{\alpha}{2}}\partial_{x}^{m}u \chi_{\epsilon, b}^{2}\,\mathrm{d}x=}_{B_{3}(t)}0.
\end{split}
\end{equation}
\S.1 \quad First, by  the induction hypothesis \eqref{kdv9}, we take  $l=m$ and $j=2k-1,$ then $\|B_{1}\|_{1}$ is bounded. More precisely,
\begin{equation}
\begin{split}
\int_{0}^{T}|B_{1}(t)|\mathrm{d}t&\lesssim\int_{0}^{T}\int_{\mathbb{R}}\left(D_{x}^{\frac{\alpha}{2}}\partial_{x}^{m}u\right)^{2}\left(\chi_{\epsilon, b}^{2}\right)'\mathrm{d}x\lesssim c^{*}_{m,1}.
\end{split}
\end{equation}

\S.2\quad Next, as usual in our argument to handle $B_{2},$  we first rewrite it as follows
\begin{equation*}
\begin{split}
B_{2}(t)&=\frac{1}{2}\int_{\mathbb{R}}D_{x}^{\frac{\alpha}{2}}\partial_{x}^{m}u\left[D_{x}^{\alpha}\partial_{x};\chi_{\epsilon, b}^{2}\right]D_{x}^{\frac{\alpha}{2}}\partial_{x}^{m}u\,\mathrm{d}x\\
&=-\frac{1}{2}\int_{\mathbb{R}}D_{x}^{\frac{\alpha}{2}}\partial_{x}^{m}u\left[\mathcal{H}D_{x}^{\alpha+1};\chi_{\epsilon, b}^{2}\right]D_{x}^{\frac{\alpha}{2}}\partial_{x}^{m}u\,\mathrm{d}x.
\end{split}
\end{equation*}
At this point several remarks  shall  be made. More precisely, $B_{2}$ has different representations  according the number $m$  be even or odd.  The full description of these issues is presented below.

\begin{itemize}
	\item[(I)] If $m\in \mathbb{Q}_{1},$  then 
		\begin{equation*}
		B_{2}(t)=-\frac{1}{2}\int_{\mathbb{R}}D_{x}^{m+\frac{\alpha}{2}}u\left[\mathcal{H}D_{x}^{\alpha+1};\chi_{\epsilon, b}^{2}\right]D_{x}^{m+\frac{\alpha}{2}}u\,\mathrm{d}x.
		\end{equation*}
			\item[(II)] If $m\in \mathbb{Q}_{2},$ then 
		\begin{equation*}
			B_{2}(t)=-\frac{1}{2}\int_{\mathbb{R}}\mathcal{H}D_{x}^{m+\frac{\alpha}{2}}u\left[\mathcal{H}D_{x}^{\alpha+1};\chi_{\epsilon, b}^{2}\right]\mathcal{H}D_{x}^{m+\frac{\alpha}{2}}u\,\mathrm{d}x.
		\end{equation*}
\end{itemize}
We  will only  focus our attention on  (II). Nevertheless, the way to proceed when  (I) holds was described in the  case $m=2.$  
 
 As was done in the previous cases, we incorporate the commutator decomposition \eqref{eq8} to obtain 
\begin{equation}\label{kdv10}
\begin{split}
B_{2}(t)&=\frac{1}{2}\int_{\mathbb{R}}\mathcal{H}D_{x}^{m+\frac{\alpha}{2}}uR_{n}(\alpha+1)\mathcal{H}D_{x}^{m+\frac{\alpha}{2}}u\,\mathrm{d}x+\frac{1}{4}\int_{\mathbb{R}}\mathcal{H}D_{x}^{m+\frac{\alpha}{2}}uP_{n}(\alpha+1)\mathcal{H}D_{x}^{m+\frac{\alpha}{2}}u\,\mathrm{d}x\\
&\quad -\frac{1}{4}\int_{\mathbb{R}}D_{x}^{m+\frac{\alpha}{2}}\mathcal{H}u\mathcal{H}P_{n}(\alpha+1)\mathcal{H}D_{x}^{m+\frac{\alpha}{2}}\mathcal{H}u\,\mathrm{d}x\\
&=B_{2,1}(t)+B_{2,2}(t)+B_{2,3}(t),
\end{split}
\end{equation}
for some positive integer $n.$

We fix $n$  satisfying 
\begin{equation*}
2n+1\leq \alpha +1 +2\left(m+\frac{\alpha}{2}\right)\leq2n+3,
\end{equation*}
from this  we obtain $n=m.$ Thus, in view of Plancherel's identity,  H\"{o}lder's inequality and Proposition \ref{propo2} 
\begin{equation}
\begin{split}
|B_{2,1}(t)|&\lesssim  \|\mathcal{H}u(t)\|_{2}\left\|D_{x}^{m+\frac{\alpha}{2}}R_{m}(\alpha+1)D_{x}^{m+\frac{\alpha}{2}}\mathcal{H}u(t)\right\|_{2}\\
&\lesssim\left\|u_{0}\right\|_{2}\left\|\widehat{D_{x}^{2m+2\alpha+1}\left(\chi_{\epsilon,b}^{2}\right)}\right\|_{1}.
\end{split}
\end{equation}
Therefore,
\begin{equation*}
\int_{0}^{T} |B_{2,1}(t)|\,\mathrm{d}t\lesssim T\|u_{0}\|_{2}.
\end{equation*}

In addition to this, the terms in the decomposition \eqref{kdv10} can be written as  
\begin{equation*}
\begin{split}
B_{2,2}(t) %\left(\frac{\alpha+1}{4}\right)\sum_{d=0}^{m}\frac{c_{2d+1}(-1)^{d}}{4^{d}}\int_{\mathbb{R}}\left(\mathcal{H}D_{x}^{m-d+\frac{\alpha}{2}}u\right)^{2}(\chi_{\epsilon, b}^{2})^{(2d+1)}\,\mathrm{d}x\\
&=\left(\frac{\alpha+1}{4}\right)\int_{\mathbb{R}} \left(\mathcal{H}D_{x}^{m+\alpha}u\right)^{2}\left(\chi_{\epsilon, b}^{2}\right)'\,\mathrm{d}x\\
&\quad +\left(\frac{\alpha+1}{4}\right)\sum_{d=1}^{m-1}\frac{c_{2d+1}(-1)^{d}}{4^{d}}\int_{\mathbb{R}}\left(\mathcal{H}D_{x}^{m-d+\alpha}u\right)^{2}\left(\chi_{\epsilon, b}^{2}\right)^{(2d+1)}\,\mathrm{d}x\\
&\quad -\left(\frac{\alpha+1}{4}\right)\frac{c_{2m+1}}{4^{m}}\int_{\mathbb{R}}\left(\mathcal{H}D_{x}^{\alpha}u\right)^{2}\left(\chi_{\epsilon, b}^{2}\right)^{(2m+1)}\,\mathrm{d}x\\
%&=B_{2,2,1}(t)+\sum_{d=1}^{m-1}B_{2,2,d}(t)+B_{2,2,m}(t),\\
&=B_{2,2,1}(t)+\sum_{d\in\mathbb{Q}_{1}(m-1)}B_{2,2,d}(t)+\sum_{d\in\mathbb{Q}_{2}(m-1)}B_{2,2,d}(t)+B_{2,2,m}(t)
\end{split}
\end{equation*}
and 
\begin{equation*}
\begin{split}
B_{2,3}(t) %&=\left(\frac{\alpha+1}{4}\right)\sum_{d=0}^{m}\frac{c_{2d+1}(-1)^{d}}{4^{d}}\int_{\mathbb{R}}\left(D_{x}^{m-d+\frac{\alpha}{2}}u\right)^{2}(\chi_{\epsilon, b}^{2})^{(2d+1)}\,\mathrm{d}x\\
&=\left(\frac{\alpha+1}{4}\right)\int_{\mathbb{R}} \left(D_{x}^{m+\alpha}u\right)^{2}\left(\chi_{\epsilon, b}^{2}\right)'\,\mathrm{d}x\\
&\quad +\left(\frac{\alpha+1}{4}\right)\sum_{d=1}^{m-1}\frac{c_{2d+1}(-1)^{d}}{4^{d}}\int_{\mathbb{R}}\left(D_{x}^{m-d+\alpha}u\right)^{2}\left(\chi_{\epsilon, b}^{2}\right)^{(2d+1)}\,\mathrm{d}x\\
&\quad -\left(\frac{\alpha+1}{4}\right)\frac{c_{2m+1}}{4^{m}}\int_{\mathbb{R}}\left(D_{x}^{\alpha}u\right)^{2}\left(\chi_{\epsilon, b}^{2}\right)^{(2m+1)}\,\mathrm{d}x\\
&=B_{2,3,1}(t)+\sum_{d\in  \mathbb{Q}_{1}(m-1)}B_{2,3,d}(t)+\sum_{d\in  \mathbb{Q}_{2}(m-1)}B_{2,3,d}(t)+B_{2,3,m}(t).
\end{split}
\end{equation*}
The terms $B_{2,2,1}$ and $B_{2,3,1}$ are positive and  represent the smoothing effect after integrate in time. Besides, $\|B_{2,2,m}\|_{1}$ and $\|B_{2,3,m}\|_{1}$ are easily  controlled by means of the local theory.

The remainder terms require extra arguments besides local theory to be estimated, we proceed to describe how to handle these terms.

First,  for $d\in \mathbb{Q}_{1}(m-1)$
\begin{equation}\label{kdv11}
\begin{split}
	\int_{0}^{T}|B_{2,2,d}(t)|\,\mathrm{d}t&\lesssim \int_{0}^{T}\int_{\mathbb{R}}\left(\mathcal{H} D_{x}^{m-d+\alpha}u\right)^{2}\chi_{\epsilon/3,b+\epsilon}'\,\mathrm{d}x\,\mathrm{d}t\\
	&\lesssim\int_{0}^{T}\int_{\mathbb{R}}\left(\mathcal{H} D_{x}^{m-d+\alpha}u\right)^{2}\left(\chi_{\epsilon/9,b+10\epsilon/9}^{2}\right)'\,\mathrm{d}x\,\mathrm{d}t\\
	&\leq c^{*}_{m-d,2},
	\end{split}
\end{equation}
while for  $d\in \mathbb{Q}_{2}(m-1)$ 
\begin{equation}\label{kdv12}
\begin{split}
\int_{0}^{T}|B_{2,2,d}(t)|\,\mathrm{d}t&\lesssim \int_{0}^{T}\int_{\mathbb{R}}\left( \partial_{x}^{m-d}D_{x}^{\alpha}u\right)^{2}\chi_{\epsilon/3,b+\epsilon}'\,\mathrm{d}x\,\mathrm{d}t\\
	&\lesssim\int_{0}^{T}\int_{\mathbb{R}}\left( \partial_{x}^{m-d}D_{x}^{\alpha}u\right)^{2}\left(\chi_{\epsilon/9,b+10\epsilon/9}^{2}\right)'\,\mathrm{d}x\,\mathrm{d}t\\
&\leq c^{*}_{m-d,2}.
\end{split}
\end{equation}
The inequalities in \eqref{kdv11} and \eqref{kdv12} are obtained combining  the properties of the weighted functions and \eqref{kdv9}.

Similar arguments can be applied to obtain 
\begin{equation*}
\begin{split}
\int_{0}^{T}|B_{2,3,d}(t)|\,\mathrm{d}t&\lesssim c^{*}_{m-d,2}\quad\mbox{for any}\quad d\in \mathbb{Q}_{1}(m-1)\cup\mathbb{Q}_{2}(m-1).
\end{split}
\end{equation*}
\S.3\quad  The term $B_{3}$ is handled  as in the previous steps, this is, first we decompose it according to the case. This is:
\begin{itemize}
	\item[(I)] if $m$ is even, then 
	\begin{equation*}
	\begin{split} 
	\partial_{x}^{m}D_{x}^{\frac{\alpha}{2}}\left(u\partial_{x}u\right)\chi_{\epsilon,b }&=c_{m}D_{x}^{m+\frac{\alpha }{2}}(u\partial_{x}u)\chi_{\epsilon,b }\\
	&=\frac{c_{m}}{2}\left[D_{x}^{m+\frac{\alpha }{2}};\chi_{\epsilon, b}\right]\partial_{x}\left((u\chi_{\epsilon, b})^{2}+\left(u\widetilde{\phi_{\epsilon, b}}\right)^{2}+u^{2}\psi_{\epsilon}\right)\\
	&\quad -c_{m}\left[D_{x}^{m+\frac{\alpha}{2}};u\chi_{\epsilon, b}\right]\partial_{x}\left(u\chi_{\epsilon, b}+u\phi_{\epsilon, b}+u\psi_{\epsilon}\right)+u\chi_{\epsilon, b}\partial_{x}^{m}D_{x}^{\frac{\alpha }{2}}\partial_{x}u\\
	&=\widetilde{B_{3,1}}(t)+\widetilde{B_{3,2}}(t)+\widetilde{B_{3,3}}(t)+\widetilde{B_{3,4}}(t)+\widetilde{B_{3,5}}(t)+\widetilde{B_{3,6}}(t)+\widetilde{B_{3,7}}(t);
	\end{split}
	\end{equation*}
	\item[(II)] instead for  $m$  odd,
	\begin{equation*}
	\begin{split} 
	\partial_{x}^{m}D_{x}^{\frac{\alpha}{2}}\left(u\partial_{x}u\right)\chi_{\epsilon,b } &=-c_{m}\mathcal{H}D_{x}^{m+\frac{\alpha }{2}}(u\partial_{x}u)\chi_{\epsilon,b }\\
	&=\frac{c_{m}}{2}\left[\mathcal{H}D_{x}^{m+\frac{\alpha }{2}};\chi_{\epsilon, b}\right]\partial_{x}\left((u\chi_{\epsilon, b})^{2}+\left(u\widetilde{\phi_{\epsilon, b}}\right)^{2}+u^{2}\psi_{\epsilon}\right)\\
	&\quad - c_{m}\left[\mathcal{H}D_{x}^{m+\frac{\alpha}{2}};u\chi_{\epsilon, b}\right]\partial_{x}\left(u\chi_{\epsilon, b}+u\phi_{\epsilon, b}+u\psi_{\epsilon}\right)+u\chi_{\epsilon, b}\partial_{x}^{m}D_{x}^{\frac{\alpha }{2}}\partial_{x}u\\
	&=\frac{c_{m}}{2}\mathcal{H}\left[D_{x}^{m+\frac{\alpha }{2}};\chi_{\epsilon, b}\right]\partial_{x}\left((u\chi_{\epsilon, b})^{2}+\left(u\widetilde{\phi_{\epsilon, b}}\right)^{2}+u^{2}\psi_{\epsilon}\right)\\
	&\quad + \frac{c_{m}}{2}\left[\mathcal{H};\chi_{\epsilon, b}\right]D_{x}^{m+\frac{\alpha }{2}}\partial_{x}\left((u\chi_{\epsilon, b})^{2}+\left(u\widetilde{\phi_{\epsilon, b}}\right)^{2}+u^{2}\psi_{\epsilon}\right)\\
	&\quad -c_{m}\mathcal{H}\left[D_{x}^{m+\frac{\alpha}{2}};u\chi_{\epsilon, b}\right]\partial_{x}\left(u\chi_{\epsilon, b}+u\phi_{\epsilon, b}+u\psi_{\epsilon}\right)\\
	&\quad -c_{m}\left[\mathcal{H};u\chi_{\epsilon, b}\right]D_{x}^{m+\frac{\alpha}{2}}\partial_{x}\left(u\chi_{\epsilon, b}+u\phi_{\epsilon, b}+u\psi_{\epsilon}\right)+u\chi_{\epsilon, b}\partial_{x}^{m}D_{x}^{\frac{\alpha }{2}}\partial_{x}u\\
	&=\widetilde{B_{3,1}}(t)+\widetilde{B_{3,2}}(t)+\widetilde{B_{3,3}}(t)+\widetilde{B_{3,4}}(t)+\widetilde{B_{3,5}}(t)+\widetilde{B_{3,6}}(t)+\widetilde{B_{3,7}}(t)\\
	&\quad + \widetilde{B_{3,8}}(t)+\widetilde{B_{3,9}}(t)+\widetilde{B_{3,10}}(t)+\widetilde{B_{3,11}}(t)+\widetilde{B_{3,12}}(t)+ \widetilde{B_{3,13}}(t).
	\end{split}
	\end{equation*}
\end{itemize}
We will focus in the harder case to estimate i.e. when $m$ is odd. The case $m$ even is is simpler. 

For the sake of simplicity we will gather  the terms  according to the  tools used for its estimation without matter the order.

First,   by Lemma \ref{lemma1} is clear that
\begin{equation}
\|\widetilde{B_{3,3l}}\|_{2}\lesssim \|u(t)\|_{\infty}\|u_{0}\|_{2}\quad\mbox{for}\quad l=1,2,3,4.
\end{equation}
A combination of the Lemma  \ref{lema1}, Lemma \ref{kdvlem1}, interpolation and the Calderon's commutator estimate \eqref{eq31}, give us 
\begin{equation}
\begin{split}
\|\widetilde{B_{3,1}}\|_{2}\lesssim \left\|u\right\|_{\infty}\left(\|u_{0}\|_{2}+\left\|D_{x}^{m+\frac{\alpha}{2}}(u\chi_{\epsilon, b})\right\|_{2}\right),
\end{split}
\end{equation} 
\begin{equation}
\begin{split}
\|\widetilde{B_{3,2}}\|_{2}\lesssim \left\|u\right\|_{\infty}\left(\|u_{0}\|_{2}+\left\|D_{x}^{m+\frac{\alpha}{2}}(u\widetilde{\phi_{\epsilon, b}})\right\|_{2}\right),
\end{split}
\end{equation} 
and 
\begin{equation}
\|\widetilde{B_{3,4+l}}\|_{2}\lesssim \left\|u\right\|_{\infty}\|u\|_{s_{\alpha,2}} \quad \mbox{for}\quad l\in \{0,1\}.
\end{equation}
Next,  applying commutator's estimate \eqref{kpdl} leads to  
\begin{equation}\label{kdv26}
\begin{split}
\|\widetilde{B_{3,7}}\|_{2}
%&\lesssim \left\|\left[\mathcal{H} ;u\chi_{\epsilon, b}\right]D_{x}^{m+1-\frac{\alpha }{2}}\partial_{x}\left((u\chi_{\epsilon, b})\right)\right\|_{2}\\
&\lesssim\|\partial_{x}(u\chi_{\epsilon, b})\|_{\infty}\left\|D_{x}^{m+\frac{\alpha}{2}}(u\chi_{\epsilon, b})\right\|_{2}
\end{split}
\end{equation}
and  
\begin{equation}\label{kdv25}
\begin{split}
\|\widetilde{B_{3,8}}\|_{2}
%&=|c_{m}|\left\|\mathcal{H}\left[D_{x}^{m+1-\alpha \left(\frac{ j+1}{2}\right)};u\chi_{\epsilon, b}\right]\partial_{x}(u\phi_{\epsilon, b})\right\|_{2}\\
%&\lesssim \left\|\left[D_{x}^{m+1-\alpha \left(\frac{ j+1}{2}\right)};u\chi_{\epsilon, b}\right]\partial_{x}(u\phi_{\epsilon, b})\right\|_{2} \\
&\lesssim \left\|D_{x}^{m+\frac{\alpha}{2}}(u\phi_{\epsilon, b})\right\|_{2}\|\partial_{x}(u\chi_{\epsilon, b})\|_{\infty}+ \left\|D_{x}^{m+\frac{ \alpha}{2}}(u\chi_{\epsilon, b})\right\|_{2}\|\partial_{x}(u\phi_{\epsilon, b})\|_{\infty}.
\end{split}
\end{equation}
In order to bound the remainder terms we  use  \eqref{kpdl} from where we get that 
\begin{equation}\label{kdv24}
\begin{split}
\|\widetilde{B_{3,10}}\|_{2}
%&=|c_{m}|\left\|\mathcal{H}\left[D_{x}^{m+\frac{\alpha j}{2}};u\chi_{\epsilon, b}\right]\partial_{x}(u\chi_{\epsilon, b})\right\|_{2}\\
%&\lesssim \left\|\left[D_{x}^{m+1-\frac{\alpha}{2}};u\chi_{\epsilon, b}\right]\partial_{x}(u\chi_{\epsilon, b})\right\|_{2}\\
&\lesssim \|\partial_{x}(u\chi_{\epsilon, b})\|_{\infty}\left\|D_{x}^{m+\frac{\alpha}{2}}(u\chi_{\epsilon, b})\right\|_{2},
\end{split}
\end{equation}
and
\begin{equation}\label{kdv27}
\begin{split}
\|\widetilde{B_{3,11}}\|_{2}
%&=|c_{m}|\left\|\mathcal{H}\left[D_{x}^{m+\frac{\alpha j}{2}};u\chi_{\epsilon, b}\right]\partial_{x}(u\chi_{\epsilon, b})\right\|_{2}\\
%&\lesssim \left\|\left[D_{x}^{m+1-\frac{\alpha}{2}};u\chi_{\epsilon, b}\right]\partial_{x}(u\chi_{\epsilon, b})\right\|_{2}\\
&\lesssim \|\partial_{x}(u\chi_{\epsilon, b})\|_{\infty}\left\|D_{x}^{m+\frac{\alpha}{2}}(u\phi_{\epsilon, b})\right\|_{2}.
\end{split}
\end{equation}
To finish with the estimates above, we replace   $\widetilde{B_{3,13}}$ into \eqref{fkdv2}, from that   it is obtained a term which can be estimated by using  integration by parts, Gronwall's inequality and the Strichartz's estimate in Theorem \ref{lt}.

Notwithstanding all the terms above have been estimated, several issues have to be clarified. First, the non-local behavior of operator $D_{x}^{s}$ for any $s\notin2\mathbb{N}$ do not allow us to know where  force us to localize the function $u$ and its derivatives in the places where the regularity is available.
 
For this reason, the term $\left\|D_{x}^{m+\frac{\alpha}{2}}(u\chi_{\epsilon, b})\right\|_{2}$ appearing  in the estimates above must be bounded, we do this  by  using   the decomposition trick,  this is 
\begin{equation}\label{kdv28}
\begin{split}
D_{x}^{m+\frac{\alpha}{2}}(u\chi_{\epsilon, b})=\chi_{\epsilon, b}D_{x}^{m+\frac{\alpha}{2}}u+\left[
D_{x}^{m+\frac{\alpha}{2}};\chi_{\epsilon, b}\right]\left(u\chi_{\epsilon, b}+u\phi_{\epsilon,b}+u\psi_{\epsilon}\right),
\end{split}
\end{equation}
and if we assume that $m$ is odd the decomposition \eqref{kdv28} have to be rewritten as follows
\begin{equation}
\begin{split}
D_{x}^{m+\frac{\alpha}{2}}(u\chi_{\epsilon, b})&=c_{m}\left[\mathcal{H}; \chi_{\epsilon, b}\right]\partial_{x}^{m}D_{x}^{\frac{\alpha}{2}}u+c_{m}\mathcal{H}\left(\chi_{\epsilon, b}\partial_{x}^{m}D_{x}^{\frac{\alpha}{2}}u\right)\\
&\quad +
\left[
D_{x}^{m+\frac{\alpha}{2}};\chi_{\epsilon, b}\right]\left(u\chi_{\epsilon, b}+u\phi_{\epsilon,b}+u\psi_{\epsilon}\right),
\end{split}
\end{equation} 
where $c_{m}$ is a non-null constant.
A direct application of Calderon's commutator estimate \eqref{eq31}, Lemma \ref{kdvlem1} and   interpolation  imply that 
\begin{equation}
\begin{split}
\left\|D_{x}^{m+\frac{\alpha}{2}}(u\chi_{\epsilon, b})\right\|_{2}&\lesssim \left\|u\right\|_{s_{\alpha},2} +\left\|\chi_{\epsilon, b}\partial_{x}^{m}D_{x}^{\frac{\alpha}{2}}u\right\|_{2}+\left\|D_{x}^{m-1+\frac{\alpha}{2}}(u\chi_{\epsilon, b})\right\|_{2}+\left\|D_{x}^{m-1+\frac{\alpha}{2}}(u\phi_{\epsilon, b})\right\|_{2}\\
&\lesssim \left\|u\right\|_{s_{\alpha},2} +\left\|\chi_{\epsilon, b}\partial_{x}^{m}D_{x}^{\frac{\alpha}{2}}u\right\|_{2}+\left\|\partial_{x}^{m}(u\chi_{\epsilon, b})\right\|_{2}+\left\|\partial_{x}^{m}(u\phi_{\epsilon, b})\right\|_{2}\\
&=A_{1}+A_{2}+A_{3}+A_{4}.
\end{split}
\end{equation}
It is straightforward to see that $A_{1}$ is bounded, this is  a consequence of the local theory.  The term, $A_{3}$  is handled by  first noticing that
\begin{equation}\label{kdv29}
\chi_{\epsilon/5, \epsilon}(x)\chi_{\epsilon, b}(x)=\chi_{\epsilon,b }(x)\quad \mbox{for all}\quad x\in \mathbb{R},
\end{equation}
then combining Lemma \ref{lema2} and Young's inequality
\begin{equation}
\begin{split}
\left\|\partial_{x}^{m}(u\chi_{\epsilon, b})\right\|_{2}\lesssim \left\|\partial_{x}^{m}u\chi_{\epsilon, b}\right\|_{2}+\sum_{d=2}^{m-1} \gamma_{m,d}\left\|\partial_{x}^{d}u\chi_{\epsilon/5, \epsilon}\right\|_{2}+\|u\|_{s_{\alpha},2}.
\end{split}
\end{equation} 
Considering a  slightly modification of the   weights involved in  \eqref{kdv29} it is possible to  follow the same arguments as above  to estimate $A_{4},$ however,  in order to avoid repetitions we will omit the calculations. 

Further, after integrate in time and use the inductive hypothesis, this is  \eqref{kdv9}, then  
\begin{equation*}
\begin{split}
\left\|A_{3}\right\|_{1}\lesssim \sum_{d=2}^{m}\widetilde{\gamma_{m,d}}\,\left(c^{*}_{d,0}\right)^{1/2}+ \|u\|_{L^{\infty}_{T}H^{s_{\alpha}}_{x}}.
\end{split}
\end{equation*}
Similarly, combining properties of the weighted functions  and the  induction hypothesis \eqref{kdv30} produces 
\begin{equation*}
\left\|A_{4}\right\|_{1}\lesssim \sum_{d=2}^{m}\lambda_{m,d}\,\left(c^{*}_{d,2k}\right)^{1/2}+ \|u\|_{L^{\infty}_{T}H^{s_{\alpha}}_{x}}.
\end{equation*}

Finally,  to estimate   $\left\|D_{x}^{m+\frac{\alpha}{2}}(u\phi_{\epsilon, b})\right\|_{2}$
we decouple it by using  Theorem \ref{thm11} as follows
\begin{equation}
\begin{split}
\left\|D_{x}^{m+\frac{\alpha}{2} }(u\phi_{\epsilon,b})\right\|_{L^{2}_{T}L_{x}^{2}}&\lesssim \left\|D_{x}^{m+\frac{\alpha}{2}}\phi_{\epsilon, b}\right\|_{L^{\infty}_{T}L^{4}_{x}}\|u_{0}\|_{2}+\sum_{d\in \mathbb{Q}_{1}(m)}\frac{1}{d!}\left\|\partial_{x}^{d}\phi_{\epsilon, b}D_{x}^{m-d+\frac{\alpha}{2}}u\right\|_{L^{2}_{T}L^{2}_{x}}\\
&\quad  +\sum_{d\in \mathbb{Q}_{2}(m)}\frac{1}{d!}\left\|\partial_{x}^{d}\phi_{\epsilon, b}\mathcal{H}D_{x}^{m-d+\frac{\alpha}{2}}u\right\|_{L^{2}_{T}L^{2}_{x}}\\
&\lesssim \|u_{0}\|_{2}+\sum_{d\in \mathbb{Q}_{1}(m)}\frac{1}{d!}\left\|\varphi_{\epsilon/8,b+\epsilon/4}D_{x}^{m-d+\frac{\alpha}{2}}u\right\|_{L^{2}_{T}L^{2}_{x}}\\
&\quad +\sum_{d\in \mathbb{Q}_{2}(m)}\frac{1}{d!}\left\|\varphi_{\epsilon/8,b+\epsilon/4}\mathcal{H}D_{x}^{m-d+\frac{\alpha}{2}}u\right\|_{L^{2}_{T}L^{2}_{x}}\\
&\lesssim \|u_{0}\|_{2}+\sum_{d\in \mathbb{Q}_{1}(m)}\frac{1}{d!}\left\|\eta_{\epsilon/24,b+7\epsilon/24}D_{x}^{m-d+\frac{\alpha}{2}}u\right\|_{L^{2}_{T}L^{2}_{x}}\\
&\quad +\sum_{d\in \mathbb{Q}_{2}(m)}\frac{1}{d!}\left\|\eta_{\epsilon/24,b+7\epsilon/24}\mathcal{H}D_{x}^{m-d+\frac{\alpha}{2}}u\right\|_{L^{2}_{T}L^{2}_{x}}\\
&\lesssim \|u_{0}\|_{2} +\sum_{d=0}^{m}\frac{1}{d!}\left(c^{*}_{m-d,1}\right)^{1/2},
\end{split}
\end{equation}
where the last inequality is a consequence of inductive hypothesis \eqref{kdv9}.

This step finish gathering together the estimates corresponding to $B_{1}, B_{2}$ and $B_{3}$ that combined with Gronwall's inequality and integration in time
\begin{equation}
\begin{split}
&\sup_{0\leq t\leq T}\left\|\partial_{x}^{m}D_{x}^{\frac{\alpha}{2}}u\chi_{\epsilon, b}(\cdot+vt)\right\|_{2}^{2}+\left\|D_{x}^{\alpha}\partial_{x}^{m}u\eta_{\epsilon,b}\right\|_{L^{2}_{T}L^{2}_{x}}^{2}+ \left\|\mathcal{H}D_{x}^{\alpha}\partial_{x}^{m}u\eta_{\epsilon,b}\right\|_{L^{2}_{T}L^{2}_{x}}^{2}\lesssim c^{*}_{m,2},
\end{split}
\end{equation}
where as usual we indicate the full dependence of the parameters behind the constant i.e. 
${\displaystyle c^{*}_{m,2}=c^{*}_{m,2}\left(\alpha; \epsilon; T;v; \|u_{0}\|_{s_{\alpha},2}; \left\|D_{x}^{\frac{\alpha}{2}}\partial_{x}^{m}u_{0}\chi_{\epsilon, b}\right\|_{2}\right)>0.}$ 
\begin{flushleft}
	\fbox{{\sc Step $2k+1:$}}
\end{flushleft}
The corresponding energy  weighted estimate for this step is
\begin{equation}\label{kdv23}
\begin{split}
&\frac{1}{2}\frac{\mathrm{d}}{\mathrm{d}t}\int_{\mathbb{R}}\left(\partial_{x}^{m}D_{x}^{1-\frac{\alpha}{2}}u\right)^{2}\chi_{\epsilon, b}^{2}\,\mathrm{d}x\underbrace{-\frac{v}{2}\int_{\mathbb{R}}\left(\partial_{x}^{m}D_{x}^{1-\frac{\alpha}{2}}u\right)^{2}\left(\chi_{\epsilon, b}^{2}\right)'\,\mathrm{d}x}_{B_{1}(t)}\\
&\underbrace{-\int_{\mathbb{R}}\left(\partial_{x}^{m}D_{x}^{1-\frac{\alpha}{2}}D_{x}^{\alpha}\partial_{x}u\right)\partial_{x}^{m}D_{x}^{1-\frac{\alpha}{2}}u \chi_{\epsilon, b}^{2}\,\mathrm{d}x}_{B_{2}(t)}\\
&\underbrace{+\int_{\mathbb{R}}\left(\partial_{x}^{m}D_{x}^{1-\frac{\alpha}{2}}(u\partial_{x}u)\right)\partial_{x}^{m}D_{x}^{1-\frac{\alpha}{2}}u \chi_{\epsilon, b}^{2}\,\mathrm{d}x}_{B_{3}(t)}=0.
\end{split}
\end{equation}
\S.1\quad We shall point out that   the representation of $B_{1}$ may change according to the case. More precisely,
%\begin{equation*}
%\begin{split}
%B_{1}(t)&=-\frac{v}{2}\int_{\mathbb{R}}\left(\partial_{x}^{m}D_{x}^{1-\frac{\alpha}{2}}u\right)^{2}\left(\chi_{\epsilon, b}^{2}\right)'\mathrm{d}x
%\end{split}
%\end{equation*}
\begin{itemize}
\item[(I)] If $m\in \mathbb{Q}_{1}$ then 
\begin{equation*}
B_{1}(t)= -v\int_{\mathbb{R}}\left(D_{x}^{m+1-\frac{\alpha}{2}}u\right)^{2}\eta_{\epsilon,b}^{2}\,\mathrm{d}x=-v\left\|D_{x}^{m+1-\frac{\alpha}{2}}u\eta_{\epsilon,b}\right\|_{2}^{2}.
\end{equation*}
\item[(II)]If $m\in \mathbb{Q}_{2}$ then 
\begin{equation*}
B_{1}(t)= -v\int_{\mathbb{R}}\left(\mathcal{H}D_{x}^{m+1-\frac{\alpha}{2}}u\right)^{2}\eta_{\epsilon,b}^{2}\,\mathrm{d}x=-v\left\|\mathcal{H}D_{x}^{m+1-\frac{\alpha}{2}}u\eta_{\epsilon,b}\right\|_{2}^{2}.
\end{equation*}
\end{itemize}

For the sake of simplicity we will study the case (I), it will be clear from the context how to proceed in the case (II).

Since 
\begin{equation*}
\begin{split}
D_{x}^{m+\alpha k}(u\eta_{\epsilon,b})&=D_{x}^{m+\alpha k}u\eta_{\epsilon,b}+\left[D_{x}^{m+\alpha k}; \eta_{\epsilon,b}\right](u\chi_{\epsilon, b}+u\phi_{\epsilon, b}+u\psi_{\epsilon}),\\
%&=\mathcal{H}D_{x}^{m+\alpha\left(\frac{j+1}{2}\right)}u\eta_{\epsilon,b}+\mathcal{H}\left[D_{x}^{m+\alpha\left(\frac{j+1}{2}\right)}; \eta_{\epsilon,b}\right](u\chi_{\epsilon, b}+u\phi_{\epsilon, b}+u\psi_{\epsilon})\\
%&+\left[\mathcal{H}; \eta_{\epsilon,b}\right]D_{x}^{m+\alpha\left(\frac{j+1}{2}\right)}\left(u\chi_{\epsilon, b}+u\phi_{\epsilon, b}+u\psi_{\epsilon}\right)
\end{split}
\end{equation*}
then combining  \eqref{kpdl}, Lemma \ref{lemma1} and interpolation 
\begin{equation*}
\begin{split}
\left\|D_{x}^{m+\alpha k }(u\eta_{\epsilon,b})\right\|_{2}&\leq \left\|D_{x}^{m+\alpha k }u\eta_{\epsilon,b}\right\|_{2}+\left\|\left[D_{x}^{m+\alpha k };\eta_{\epsilon,b}\right](u\chi_{\epsilon, b}+u\phi_{\epsilon, b}+u\psi_{\epsilon})\right\|_{2}\\
%&\lesssim \left\|D_{x}^{m+1-\frac{\alpha}{2}}u\eta_{\epsilon,b}\right\|_{2}+ \left\|D_{x}^{m-\frac{\alpha}{2}}(u\chi_{\epsilon, b})\right\|_{2}+\left\|D_{x}^{m-\frac{\alpha}{2}}(u\phi_{\epsilon, b})\right\|_{2} \\
%&\quad +\left\|\eta_{\epsilon,b}D_{x}^{m+1-\frac{\alpha}{2}}(u\psi_{\epsilon})\right\|_{2}\\
&\lesssim\left\|D_{x}^{m+\alpha k }u\eta_{\epsilon,b}\right\|_{2}+\left\|\partial_{x}^{m}(u\chi_{\epsilon,b })\right\|_{2}+\left\|\partial_{x}^{m}(u\phi_{\epsilon,b })\right\|_{2}+\|u_{0}\|_{2}.
\end{split}
\end{equation*}
Since 
\begin{equation}\label{kdv21}
\chi_{\epsilon/5, \epsilon}(x)\chi_{\epsilon, b}(x)=\chi_{\epsilon,b }(x)\quad \mbox{for all}\quad x\in \mathbb{R},
\end{equation}
 then combining Lemma \ref{lema2} and Young's inequality give us 
\begin{equation}
\begin{split}
\left\|\partial_{x}^{m}(u\chi_{\epsilon,b })\right\|_{2}& %\leq \sum_{d=0}^{m}c_{m,d}\left\|\partial_{x}^{d}u\chi_{\epsilon/5, \epsilon}\partial_{x}^{m-d}\chi_{\epsilon, b}\right\|_{2}\\
%&=\left\|\partial_{x}^{m}u\chi_{\epsilon, b}\right\|_{2}+\sum_{d=0}^{m-1}c_{m,d}\left\|\partial_{x}^{d}u\chi_{\epsilon/5, \epsilon}\partial_{x}^{m-d}\chi_{\epsilon,b }\right\|_{2} \\
%&\leq \left\|\partial_{x}^{m}u\chi_{\epsilon, b}\right\|_{2}+ \sum_{d=0}^{m-1}c_{m,d}\left\|\partial_{x}^{m-d}\chi_{\epsilon, b}\right\|_{\infty}\left\|\partial_{x}^{d}u\chi_{\epsilon/5, \epsilon}\right\|_{2}\\
\lesssim \left\|\partial_{x}^{m}u\chi_{\epsilon, b}\right\|_{2}+\sum_{d=2}^{m-1}\gamma_{m,d}\left\|\partial_{x}^{d}u\chi_{\epsilon/5, \epsilon}\right\|_{2}+\|u\|_{s_{\alpha},2}.
\end{split}
\end{equation}
Hence, by \eqref{kdv9}
\begin{equation}
\left\|\partial_{x}^{m}u\chi_{\epsilon,b }\right\|_{L^{\infty}_{T}L^{2}_{x}}\lesssim \left(c^{*}_{m,1}\right)
^{1/2}+\sum_{d=0}^{m-1}\gamma_{m,d}\left(c^{*}_{d,1}\right)^{1/2}+\|u\|_{L^{\infty}_{T}H^{s_{\alpha}}_{x}}.
\end{equation}
and 

\begin{equation}
\begin{split}
\left\|\partial_{x}^{m}(u\phi_{\epsilon, b})\right\|_{2}&\lesssim \sum_{d=2}^{m}\gamma_{m,d}\left\|\partial_{x}^{d}u\chi_{\epsilon/20, \epsilon/4}\right\|_{2}+\|u\|_{s_{\alpha},2}.
\end{split}
\end{equation}
Analogously
\begin{equation}\label{kdv22}
\left\|\partial_{x}^{m}(u\phi_{\epsilon, b})\right\|_{L^{\infty}_{T}L^{2}_{x}}\lesssim \sum_{d=2}^{m}\gamma_{m,d} \left(c^{*}_{d,1}\right)^{1/2}+\|u\|_{L^{\infty}_{T}H^{s_{\alpha}}_{x}}.
\end{equation}

Gathering the estimated terms above allow us to  obtain 
\begin{equation}
\left\|D_{x}^{m+\alpha k }(u\eta_{\epsilon,b})\right\|_{L^{2}_{T}L^{2}_{x}}<\infty.
\end{equation}

Therefore, by interpolation 
\begin{equation}
\left\|\partial_{x}^{m}D_{x}^{1-\frac{\alpha}{2}}(u\eta_{\epsilon,b})\right\|_{L^{2}_{T}L^{2}_{x}}\lesssim \|u_{0}\|_{2}+\left\|D_{x}^{m+\alpha k }(u\eta_{\epsilon,b})\right\|_{L^{2}_{T}L^{2}_{x}}.
\end{equation}
With  this information at hand, we can  estimate the term with the regularity required, this is  achieved localizing the commutator expression as follows
\begin{equation*}
\partial_{x}^{m}D_{x}^{1-\frac{\alpha}{2}}u\eta_{\epsilon,b}=\partial_{x}^{m}D_{x}^{1-\frac{\alpha}{2}}(u\eta_{\epsilon,b})-\left[\partial_{x}^{m}D_{x}^{1-\frac{\alpha}{2}}; \eta_{\epsilon,b}\right]\left(u\chi_{\epsilon, b}+u\phi_{\epsilon, b}+u\psi_{\epsilon}\right),
\end{equation*}
that  by using a similar argument as before  we obtain 
\begin{equation}
\begin{split}
\left\|\partial_{x}^{m}D_{x}^{1-\frac{\alpha}{2}}u \eta_{\epsilon,b}\right\|_{L^{2}_{T}L^{2}_{x}}<\infty.
\end{split}
\end{equation}

So that,
\begin{equation}
\begin{split}
\int_{0}^{T}|B_{1}(t)|\mathrm{d}t&=v\int_{0}^{T}\left\|\partial_{x}^{m}D_{x}^{1-\frac{\alpha}{2}}u\eta_{\epsilon,b}\right\|_{2}^{2}\,\mathrm{d}t<\infty.\\
\end{split}
\end{equation}
\S.2\quad As was evidenced at the beginning of this step, several case shall be considered. In fact, the term $B_{2}$ is represented in different ways according  be the case 
\begin{itemize}
	\item[(I)] If $m\in \mathbb{Q}_{1},$  then 
	\begin{equation*}
	B_{2}(t)=-\frac{1}{2}\int_{\mathbb{R}}D_{x}^{m+1-\frac{\alpha}{2}}u\left[\mathcal{H}D_{x}^{\alpha+1};\chi_{\epsilon, b}^{2}\right]D_{x}^{m+1-\frac{\alpha}{2}}u\,\mathrm{d}x.
	\end{equation*}
		\item[(II)] If $m\in \mathbb{Q}_{2}$ then 
	\begin{equation*}
	B_{2}(t)=-\frac{1}{2}\int_{\mathbb{R}}\mathcal{H}D_{x}^{m+1-\frac{\alpha}{2}}u\left[\mathcal{H}D_{x}^{\alpha+1};\chi_{\epsilon, b}^{2}\right]\mathcal{H}D_{x}^{m+1-\frac{\alpha}{2}}u\,\mathrm{d}x
	\end{equation*}
\end{itemize}
To show  how to proceed in the case that $m$ is odd, we will use the expression above and the commutator decomposition as we have previously described  to obtain 
\begin{equation*}
\begin{split}
B_{2}(t) 
&=\frac{1}{2}\int_{\mathbb{R}}\mathcal{H}D_{x}^{m+1-\frac{\alpha}{2}}uR_{n}(\alpha+1)\mathcal{H}D_{x}^{m+1-\frac{\alpha}{2}}u\,\mathrm{x}\\
&\quad +\frac{1}{4}\int_{\mathbb{R}}\mathcal{H}D_{x}^{m+1-\frac{\alpha}{2}}uP_{n}(\alpha+1)\mathcal{H}D_{x}^{m+1-\frac{\alpha}{2}}u\,\mathrm{d}x\\
&\quad-\frac{1}{4}\int_{\mathbb{R}} \mathcal{H}D_{x}^{m+1-\frac{\alpha}{2}}u\mathcal{H}P_{n}(\alpha+1)\mathcal{H}D_{x}^{m+1-\frac{\alpha}{2}}\mathcal{H}u\\
&=B_{2,1}(t)+B_{2,2}(t)+B_{2,3}(t).
\end{split}
\end{equation*}
 Is choosing $n$  according to the rule
\begin{equation*}
2n+1\leq \alpha +1 +2\left(m+1-\frac{\alpha}{2}\right) \leq 2n+3
\end{equation*}
 which clearly implies $n=m.$
  
 Next, an application of Proposition \ref{propo2} implies  that the remainder term $R_{m}(\alpha+1)$ is bounded in $L^{2}_{x}$ which let us handled $B_{2,1}$ as follows
 \begin{equation*}
 \begin{split}
 B_{2,1}(t)&=\frac{1}{2}\int_{\mathbb{R}} \mathcal{H}D_{x}^{m+1-\frac{\alpha}{2}}uR_{m}(\alpha+1)D_{x}^{m+1-\frac{\alpha}{2}}\mathcal{H}u\,\mathrm{d}x\\
 &= \frac{1}{2}\int_{\mathbb{R}} \mathcal{H}uD_{x}^{m+1-\frac{\alpha}{2}}R_{m}(\alpha+1)D_{x}^{m+1-\frac{\alpha}{2}}\mathcal{H}u\,\mathrm{d}x\\
 \end{split}
 \end{equation*} 
 Combining H\"{o}lder's inequality and 
 \begin{equation*}
 \begin{split}
 |B_{2,1}(t)|& \lesssim \|\mathcal{H}u(t)\|_{2}\left\|D_{x}^{m+1-\frac{\alpha}{2}}R_{m}(\alpha+1)D_{x}^{m+1-\frac{\alpha}{2}}\mathcal{H}u\right\|_{2}\\
 &\lesssim \|u_{0}\|_{2}^{2}\left\|\widehat{D_{x}^{2m+3}(\chi_{\epsilon, b}^{2})}\right\|_{1}.
  \end{split}
 \end{equation*}
 From this it is easy to  obtain  that 
 \begin{equation*}
 \int_{0}^{T}|B_{2,1}(t)|\,\mathrm{d}t\leq c.
 \end{equation*}
 The control in $B_{2,1}$ allow us to  fix the  number of terms in $B_{2,2}$ and $B_{2,3}.$ 
 \begin{equation}\label{kdv13}
\begin{split}
B_{2,2}(t) %&=\left(\frac{\alpha+1}{4}\right)\sum_{d=0}^{m} \frac{c_{2d+1}(-1)^{d}}{4^{d}}\int_{\mathbb{R}} \left(D_{x}^{m+1-d}u\right)^{2}(\chi_{\epsilon, b}^{2})^{(2d+1)} \,\mathrm{d}x\\
&= \left(\frac{\alpha+1}{4}\right)\int_{\mathbb{R}} \left(D_{x}^{m+1}u\right)^{2}\left(\chi_{\epsilon, b}^{2}\right)' \,\mathrm{d}x\\
&\quad +\left(\frac{\alpha+1}{4}\right)\sum_{d=1}^{m-1}\frac{c_{2d+1}(-1)^{d}}{4^{d}}\int_{\mathbb{R}} \left(D_{x}^{m+1-d}u\right)^{2}\left(\chi_{\epsilon, b}^{2}\right)^{(2d+1)} \,\mathrm{d}x\\
&\quad -\left(\frac{\alpha+1}{4}\right)\left(\frac{c_{2m+1}}{4^{m}}\right)\int_{\mathbb{R}}\left(D_{x}u\right)^{2}\left(\chi_{\epsilon, b}^{2}\right)^{(2m+1)}\,\mathrm{d}x\\
&=B_{2,2,1}(t)+\sum_{d=2}^{m-1}B_{2,2,d}(t)+B_{2,2,m+1}(t)
\end{split}
\end{equation}
 and 
\begin{equation}\label{kdv14}
\begin{split}
B_{2,3}(t) %&=
%\left(\frac{\alpha+1}{4}\right)\sum_{d=0}^{m} \frac{c_{2d+1}(-1)^{d}}{4^{d}}\int_{\mathbb{R}} \left(\mathcal{H}D_{x}^{m+1-d}u\right)^{2}(\chi_{\epsilon, b}^{2})^{(2d+1)} \,\mathrm{d}x\\
&= \left(\frac{\alpha+1}{4}\right)\int_{\mathbb{R}} \left(\mathcal{H}D_{x}^{m+1}u\right)^{2}\left(\chi_{\epsilon, b}^{2}\right)' \,\mathrm{d}x\\
&\quad +\left(\frac{\alpha+1}{4}\right)\sum_{d=1}^{m-1}\frac{c_{2d+1}(-1)^{d}}{4^{d}}\int_{\mathbb{R}} \left(\mathcal{H}D_{x}^{m+1-d}u\right)^{2}\left(\chi_{\epsilon, b}^{2}\right)^{(2d+1)} \,\mathrm{d}x\\
&\quad -\left(\frac{\alpha+1}{4}\right)\left(\frac{c_{2m+1}}{4^{m}}\right)\int_{\mathbb{R}}\left(\mathcal{H}D_{x}u\right)^{2}\left(\chi_{\epsilon, b}^{2}\right)^{(2m+1)}\,\mathrm{d}x\\
&=B_{2,3,1}(t)+\sum_{d=2}^{m-1}B_{2,3,d}(t)+B_{2,3,m+1}(t).
\end{split}
\end{equation}
Concerning the terms  $B_{2,2,1}$ and $B_{2,3,1},$  after integrate in time they grant the smoothing effect. The remainders terms in \eqref{kdv13} and \eqref{kdv14}  have to be estimated separately.

In this sense,  for $d\in \left\{1,2,\cdots,m-1\right\}$
\begin{equation*}
\begin{split}
\int_{0}^{T}|B_{2,l+1,d}(t)|\,\mathrm{d}t& \lesssim \int_{0}^{T}\int_{\mathbb{R}}\left(\mathcal{H}^{l}D_{x}^{m+1-d}u\right)^{2}\,\chi_{\epsilon/3, b+\epsilon}'\,\mathrm{d}x\,\mathrm{d}t\quad \mbox{for}\quad l\in \{0,1\}.\\
\end{split}
\end{equation*}
Then a combination of  \eqref{kdv9} and  properties of the weighted functions imply that 
\begin{equation*}
\begin{split}
\int_{0}^{T}|B_{2,l+1,d}(t)|\,\mathrm{d}t\lesssim c^{*}_{m+1-d,2k+1}\quad\mbox{for any}\quad l\in \{0,1\},\quad d=2,3,\cdots,m-2.
\end{split}
\end{equation*}
The cases not considered in the estimation above are handled by means of the local theory.

\S.3\quad 
%\begin{equation*}
%\begin{split}
%&\partial_{x}^{m}D_{x}^{1-\frac{\alpha}{2}}(u\partial_{x}u)\chi_{\epsilon,b }\\
%&=-\frac{1}{2}\left[\partial_{x}^{m}D_{x}^{1- \frac{ \alpha}{2}};\chi_{\epsilon, b}\right]\partial_{x}\left((u\chi_{\epsilon, b})^{2}+(u\widetilde{\phi_{\epsilon, b}})^{2}+u^{2}\psi_{\epsilon}\right)\\
%&\quad + \left[\partial_{x}^{m}D_{x}^{1-\frac{\alpha}{2}};u\chi_{\epsilon, b}\right]\partial_{x}\left((u\chi_{\epsilon, b})+(u\phi_{\epsilon, b})+(u\psi_{\epsilon})\right)+u\chi_{\epsilon, b}\partial_{x}^{m}D_{x}^{1- \frac{\alpha}{2}}\partial_{x}u
%\end{split}
%\end{equation*}
 The term   $B_{2},$ can be rewritten according be the case:
\begin{itemize}
	\item[(I)] If $ m\in \mathbb{Q}_{1},$ then  there exists a  non-null constant $c=c(m)$ such that 
	\begin{equation*}
	\begin{split}
	\partial_{x}^{m}D_{x}^{1- \frac{ \alpha}{2}}(u\partial_{x}u)\chi_{\epsilon,b }
	&=c_{m}\left[D_{x}^{m+1-\frac{ \alpha}{2}};\chi_{\epsilon, b}\right]\partial_{x}\left((u\chi_{\epsilon, b})^{2}+(u\widetilde{\phi_{\epsilon, b}})^{2}+u^{2}\psi_{\epsilon}\right)\\
	&\quad + c_{m}\left[D_{x}^{m+1-\frac{ \alpha}{2}};u\chi_{\epsilon, b}\right]\partial_{x}\left((u\chi_{\epsilon, b})+(u\phi_{\epsilon, b})
	+(u\psi_{\epsilon})\right)\\
	&\quad +u\chi_{\epsilon, b}\partial_{x}^{m}D_{x}^{1-\frac{\alpha}{2}}\partial_{x}u.
	\end{split}
	\end{equation*}
	\item[(II)] If $m\in \mathbb{Q}_{2}$ then  there exists a non-null constant  $c=c(m)$ such that

	\begin{equation*}
	\begin{split}
	\partial_{x}^{m}D_{x}^{1-\frac{\alpha}{2}}(u\partial_{x}u)\chi_{\epsilon,b }
	&=c_{m}\mathcal{H}\left[D_{x}^{m+1-\frac{\alpha}{2}};\chi_{\epsilon, b}\right]\partial_{x}\left((u\chi_{\epsilon, b})^{2}+\left(u\widetilde{\phi_{\epsilon, b}}\right)^{2}+u^{2}\psi_{\epsilon}\right)\\
	&\quad +c_{m}\left[\mathcal{H};\chi_{\epsilon, b}\right]D_{x}^{m+1-\frac{\alpha}{2}}\partial_{x}\left((u\chi_{\epsilon, b})^{2}+\left(u\widetilde{\phi_{\epsilon, b}}\right)^{2}+u^{2}\psi_{\epsilon}\right)\\
	&\quad + c_{m}\mathcal{H}\left[D_{x}^{m+1-\frac{\alpha}{2}};u\chi_{\epsilon, b}\right]\partial_{x}\left((u\chi_{\epsilon, b})+(u\phi_{\epsilon, b})+(u\psi_{\epsilon})\right)\\
	&\quad + c_{m}\left[\mathcal{H} ;u\chi_{\epsilon, b}\right]D_{x}^{m+1-\frac{\alpha}{2}}\partial_{x}\left((u\chi_{\epsilon, b})+(u\phi_{\epsilon, b})+(u\psi_{\epsilon})\right)\\
	&\quad+u\chi_{\epsilon, b}\partial_{x}^{m}D_{x}^{1-\frac{\alpha}{2}}\partial_{x}u\\
	&=\widetilde{B_{3,1}}(t)+\widetilde{B_{3,2}}(t)+\widetilde{B_{3,3}}(t)+\widetilde{B_{3,4}}(t)+\widetilde{B_{3,5}}(t)+\widetilde{B_{3,6}}(t)\\
	&\quad +\widetilde{B_{3,7}}(t)+\widetilde{B_{3,8}}(t)+\widetilde{B_{3,9}}(t)+\widetilde{B_{3,10}}(t)+\widetilde{B_{3,11}}(t)+\widetilde{B_{3,12}}(t)\\
	&\quad +\widetilde{B_{3,13}}(t).
	\end{split}
	\end{equation*}
\end{itemize}
To show how to proceed in the case $m$ odd we 
combine Lemma \ref{lema1} and Lemma \ref{kdvlem1} to obtain
\begin{equation}\label{kdv15}
\begin{split}
\|\widetilde{B_{3,1}}\|_{2}
%&=|c_{m}|\left\|\mathcal{H}\left[D_{x}^{m+1-\frac{\alpha}{2}};\chi_{\epsilon, b}\right]\partial_{x}\left((u\chi_{\epsilon, b})^{2}\right)\right\|_{2}\\
%&\lesssim \left\|\left[D_{x}^{m+1-\frac{\alpha}{2}};\chi_{\epsilon, b}\right]\partial_{x}\left(\left(u\chi_{\epsilon, b}\right)^{2}\right)\right\|_{2}\\
%&\lesssim \|\partial_{x}\chi_{\epsilon, b}\|_{l,2}\left\|\left(u\chi_{\epsilon, b}\right)^{2}\right\|_{m+1-\frac{\alpha}{2},2}\\
&\lesssim \left\|u\right\|_{\infty}\left(\|u_{0}\|_{2}+\left\|D_{x}^{m+1-\frac{\alpha}{2}}(u\chi_{\epsilon, b})\right\|_{2}\right),
\end{split}
\end{equation}
and 
\begin{equation}\label{kdv16}
\begin{split}
\|\widetilde{B_{3,2}}\|_{2} %&=|c_{m}|\left\|\mathcal{H}\left[D_{x}^{m+1-\frac{\alpha}{2}};\chi_{\epsilon, b}\right]\partial_{x}\left(\left(u\widetilde{\phi_{\epsilon, b}}\right)^{2}\right)\right\|_{2}\\
&\lesssim\left\|u\right\|_{\infty}\left(\|u_{0}\|_{2}+\left\|D_{x}^{m+1-\frac{\alpha}{2}}(u\widetilde{\phi_{\epsilon, b}})\right\|_{2}\right).
\end{split}
\end{equation}
Since the weighted functions $\chi_{\epsilon, b}$ and $\psi_{\epsilon}$ satisfy the hypothesis of Lemma  \ref{lemma1} then   
%\begin{equation*}
%\dist\left(\supp(\chi_{\epsilon, b}),\supp(\psi_{\epsilon})\right)\geq \frac{\epsilon}{2}>0,
%\end{equation*}
\begin{equation*}
\begin{split}
\|\widetilde{B_{3,3l}}\|_{2} %&=|c_{m}|\left\|\mathcal{H}\left[D_{x}^{m+1-\frac{\alpha}{2}};\chi_{\epsilon, b}\right]\partial_{x}\left(u^{2}\psi_{\epsilon}\right)\right\|_{2}\\
%&\lesssim \left\|\left[D_{x}^{m+1-\frac{\alpha}{2}};\chi_{\epsilon, b}\right]\partial_{x}\left(u^{2}\psi_{\epsilon}\right)\right\|_{2}\\
%&=c\left\|\chi_{\epsilon, b}D_{x}^{m+1-\frac{\alpha}{2}}\left(u^{2}\psi_{\epsilon}\right)\right\|_{2}\\
&\lesssim \|u\|_{\infty}\|u_{0}\|_{2}\quad \mbox{for }\quad l =1,2,3,4.
\end{split}
\end{equation*}
%\begin{equation*}
%\begin{split}
%\|\widetilde{B_{3,6}}\|_{2} %&\lesssim \left\|\mathcal{H}\left(\chi_{\epsilon, b}\partial_{x}D_{x}^{m+1-\frac{ \alpha}{2}}\left(u^{2}\psi_{\epsilon}\right)\right)\right\|_{2}+\left\|\chi_{\epsilon, b}D_{x}^{m+1-\frac{\alpha}{2}}(u^{2}\psi_{\epsilon})\right\|_{2}\\
%%&\lesssim \left\|\chi_{\epsilon, b}\partial_{x}D_{x}^{m+1-\frac{\alpha}{2}}\left(u^{2}\psi_{\epsilon}\right)\right\|_{2}+\|u\|_{\infty}\|u_{0}\|_{2}\\
%&\lesssim\|u\|_{\infty}\|u_{0}\|_{2}.
%\end{split}
%\end{equation*}
On the other hand,  the  Calderon's commutator estimate \eqref{eq31} and Lemma \ref{lema1} lead to
\begin{equation*}
\begin{split}
\|\widetilde{B_{3,4}}\|_{2}
%&=|c_{m}|\left\|\left[\mathcal{H};\chi_{\epsilon, b}\right]\partial_{x}^{m}D_{x}^{1-\frac{\alpha }{2}}\mathcal{H}\left((u\chi_{\epsilon, b})^{2}\right)\right\|_{2}\\
%&\lesssim \left\|\partial_{x}^{m}\chi_{\epsilon, b}\right\|_{\infty}\left\|\mathcal{H}D_{x}^{1-\frac{ \alpha}{2}}\left((u\chi_{\epsilon, b})^{2}\right)\right\|_{2}\\
%&\lesssim \left\|D_{x}^{1-\frac{\alpha }{2}}(u\chi_{\epsilon, b})\right\|_{2}\|u\chi_{\epsilon, b}\|_{\infty}\\
&\lesssim \|u\|_{1,2}\|u\|_{\infty},
\end{split}
\end{equation*}
and
\begin{equation*}
\begin{split}
\|\widetilde{B_{3,5}}\|_{2}
%&=|c_{m}|\left\|\left[\mathcal{H};\chi_{\epsilon, b}\right]\partial_{x}^{m}D_{x}^{1-\frac{\alpha}{2}}\mathcal{H}\left(\left(u\widetilde{\phi_{\epsilon, b}}\right)^{2}\right)\right\|_{2}\\
%&\lesssim \left\|\partial_{x}^{m+1}\chi_{\epsilon, b}\right\|_{\infty}\left\|\mathcal{H}D_{x}^{\frac{\alpha j}{2}}\left(\left(u\chi_{\epsilon, b}\right)^{2}\right)\right\|_{2}\\
%&\lesssim \left\|D_{x}^{\frac{\alpha j}{2}}(u\widetilde{\phi_{\epsilon, b}})\right\|_{2}\left\|u\widetilde{\phi_{\epsilon, b}}\right\|_{\infty}\\
&\lesssim \|u\|_{1,2}\|u\|_{\infty}.
\end{split}
\end{equation*}
Moreover, by the commutator estimate  \eqref{kpdl} we have  
\begin{equation}\label{kdv17}
\begin{split}
\|\widetilde{B_{3,7}}\|_{2}
%&\lesssim \left\|\left[\mathcal{H} ;u\chi_{\epsilon, b}\right]D_{x}^{m+1-\frac{\alpha }{2}}\partial_{x}\left((u\chi_{\epsilon, b})\right)\right\|_{2}\\
&\lesssim\|\partial_{x}(u\chi_{\epsilon, b})\|_{\infty}\left\|D_{x}^{m+1-\frac{\alpha}{2}}(u\chi_{\epsilon, b})\right\|_{2}
\end{split}
\end{equation}
and
\begin{equation}\label{kdv18}
\begin{split}
\|\widetilde{B_{3,8}}\|_{2}
%&=|c_{m}|\left\|\mathcal{H}\left[D_{x}^{m+1-\alpha \left(\frac{ j+1}{2}\right)};u\chi_{\epsilon, b}\right]\partial_{x}(u\phi_{\epsilon, b})\right\|_{2}\\
%&\lesssim \left\|\left[D_{x}^{m+1-\alpha \left(\frac{ j+1}{2}\right)};u\chi_{\epsilon, b}\right]\partial_{x}(u\phi_{\epsilon, b})\right\|_{2} \\
&\lesssim \left\|D_{x}^{m+1- \frac{\alpha}{2}}(u\phi_{\epsilon, b})\right\|_{2}\|\partial_{x}(u\chi_{\epsilon, b})\|_{\infty}+ \left\|D_{x}^{m+1-\frac{ \alpha}{2}}(u\chi_{\epsilon, b})\right\|_{2}\|\partial_{x}(u\phi_{\epsilon, b})\|_{\infty}.
\end{split}
\end{equation}
Furthermore, an application of \eqref{kpdl} yields  
\begin{equation}\label{kdv19}
\begin{split}
\|\widetilde{B_{3,10}}\|_{2}
%&=|c_{m}|\left\|\mathcal{H}\left[D_{x}^{m+\frac{\alpha j}{2}};u\chi_{\epsilon, b}\right]\partial_{x}(u\chi_{\epsilon, b})\right\|_{2}\\
%&\lesssim \left\|\left[D_{x}^{m+1-\frac{\alpha}{2}};u\chi_{\epsilon, b}\right]\partial_{x}(u\chi_{\epsilon, b})\right\|_{2}\\
&\lesssim \|\partial_{x}(u\chi_{\epsilon, b})\|_{\infty}\left\|D_{x}^{m+1-\frac{\alpha}{2}}(u\chi_{\epsilon, b})\right\|_{2},
\end{split}
\end{equation}
and
\begin{equation}\label{kdv20}
\begin{split}
\|\widetilde{B_{3,11}}\|_{2}
%&=|c_{m}|\left\|\mathcal{H}\left[D_{x}^{m+\frac{\alpha j}{2}};u\chi_{\epsilon, b}\right]\partial_{x}(u\chi_{\epsilon, b})\right\|_{2}\\
%&\lesssim \left\|\left[D_{x}^{m+1-\frac{\alpha}{2}};u\chi_{\epsilon, b}\right]\partial_{x}(u\chi_{\epsilon, b})\right\|_{2}\\
&\lesssim \|\partial_{x}(u\chi_{\epsilon, b})\|_{\infty}\left\|D_{x}^{m+1-\frac{\alpha}{2}}(u\phi_{\epsilon, b})\right\|_{2}.
\end{split}
\end{equation}
Notice that in the inequalities \eqref{kdv15}-\eqref{kdv16} and \eqref{kdv17}-\eqref{kdv20} there are several terms which have been not  estimated yet. 

Firstly, 
\begin{equation*}
\begin{split}
D_{x}^{m+1-\frac{\alpha}{2}}(u\chi_{\epsilon, b})&=c_{m}\partial_{x}^{m}\mathcal{H}D_{x}^{1-\frac{\alpha}{2}}u\chi_{\epsilon, b}+\left[D_{x}^{m+1-\frac{\alpha}{2}}; \chi_{\epsilon,b }\right]\left(u\chi_{\epsilon, b}+u\phi_{\epsilon, b}+u\psi_{\epsilon}\right)\\
&=\mathcal{H}\left(\chi_{\epsilon, b}\partial_{x}^{m}D_{x}^{1-\frac{\alpha}{2}}u\right)
-\left[\mathcal{H}; \chi_{\epsilon,b}\right]\partial_{x}^{m}D_{x}^{1-\frac{\alpha}{2}}u\\
&\quad+\left[D_{x}^{m+1-\frac{\alpha}{2}}; \chi_{\epsilon,b }\right]\left(u\chi_{\epsilon, b}+u\phi_{\epsilon, b}+u\psi_{\epsilon}\right),
\end{split}
\end{equation*}
where $c_{m}$ is a non-null constant.

Since the arguments  to estimate the expression on the right  have been  previously used, we will only indicate the tools used.

Combining the Calderon's commutator estimate \eqref{eq31}, \eqref{kpdl} and interpolation imply that 
\begin{equation*}
\begin{split}
\left\|D_{x}^{m+1-\frac{\alpha}{2}}(u\chi_{\epsilon, b})\right\|_{2}&\lesssim \left\|\partial_{x}^{m}D_{x}^{1-\frac{\alpha}{2}}u\chi_{\epsilon, b}\right\|_{2}+\|u\|_{s_{\alpha},2}+ \|\partial_{x}^{m}\left(u\chi_{\epsilon, b}\right)\|_{2}+\|\partial_{x}^{m}\left(u\phi_{\epsilon, b}\right)
\|_{m}.
\end{split}
\end{equation*}
Notice that the first term on the right hand side is the quantity to be estimated. The third and fourth term  are handled by using similar arguments as those  described in \eqref{kdv21}-\eqref{kdv22}. 

Next, the condition 
\begin{equation}
\alpha\in \left[ \frac{2}{2k+1},\frac{1}{k}\right)\Rightarrow m+1-\frac{\alpha}{2}\leq m+\alpha k
\end{equation}
is used to obtain information from the previous cases. This is achieved  combining Theorem \ref{thm11} and \eqref{kdv9}, that is,  
\begin{equation}
\begin{split}
\left\|D_{x}^{m+\alpha k}(u\phi_{\epsilon,b})\right\|_{L^{2}_{T}L_{x}^{2}}&\lesssim \left\|D_{x}^{m+\alpha k}\phi_{\epsilon, b}\right\|_{L^{\infty}_{T}L^{4}_{x}}\|u_{0}\|_{2}+\sum_{d\in \mathbb{Q}_{1}(m)}\frac{1}{d!}\left\|\partial_{x}^{d}\phi_{\epsilon, b}D_{x}^{m+\alpha k-d}u\right\|_{L^{2}_{T}L^{2}_{x}}\\
&\quad  +\sum_{d\in \mathbb{Q}_{2}(m)}\frac{1}{d!}\left\|\partial_{x}^{d}\phi_{\epsilon, b}\mathcal{H}D_{x}^{m+\alpha k-d}u\right\|_{L^{2}_{T}L^{2}_{x}}\\
&\lesssim \|u_{0}\|_{2}+\sum_{d\in \mathbb{Q}_{1}(m)}\frac{1}{d!}\left\|\varphi_{\epsilon/8,b+\epsilon/4}D_{x}^{m+\alpha k-d}u\right\|_{L^{2}_{T}L^{2}_{x}}\\
&\quad +\sum_{d\in \mathbb{Q}_{2}(m)}\frac{1}{d!}\left\|\varphi_{\epsilon/8,b+\epsilon/4}\mathcal{H}D_{x}^{m+\alpha k-d}u\right\|_{L^{2}_{T}L^{2}_{x}}\\
&\lesssim \|u_{0}\|_{2}+\sum_{d\in \mathbb{Q}_{1}(m)}\frac{1}{d!}\left\|\eta_{\epsilon/24,b+7\epsilon/24}D_{x}^{m+\alpha k-d}u\right\|_{L^{2}_{T}L^{2}_{x}}\\
&\quad +\sum_{d\in \mathbb{Q}_{2}(m)}\frac{1}{d!}\left\|\eta_{\epsilon/24,b+7\epsilon/24}\mathcal{H}D_{x}^{m+\alpha k-d}u\right\|_{L^{2}_{T}L^{2}_{x}}\\
&\lesssim \|u_{0}\|_{2} +\sum_{d=0}^{m}\frac{1}{d!}\left(c^{*}_{m+\alpha k -d,2k-1}\right)^{1/2}.
\end{split}
\end{equation}
So that  interpolation provide the  bound
\begin{equation}
\left\|\partial_{x}^{m}D_{x}^{1-\frac{\alpha}{2}}\left(u\phi_{\epsilon, b}\right)\right\|_{L^{2}_{T}L^{2}_{x}}\lesssim \|u_{0}\|_{2}+ \left\|D_{x}^{m+\alpha k}(u\phi_{\epsilon,b})\right\|_{L^{2}_{T}L_{x}^{2}}.
\end{equation}
Analogously is estimated $\left\|\partial_{x}^{m}D_{x}^{1-\frac{\alpha}{2}}\left(u\widetilde{\phi_{\epsilon, b}}\right)\right\|_{L^{2}_{T}L^{2}_{x}}.$

Inserting   $\widetilde{B_{3,13}}$ into \eqref{kdv23}  it is obtained a term which can be estimated by integration by parts, Gronwall's inequality and the Strichartz's estimate in Theorem \ref{lt}, i.e $\partial_{x}u\in L^{1}\left([0,T]:L^{\infty}(\mathbb{R})\right).$

Finally, gathering together the estimates  corresponding to this step combined  with  Gronwall's inequality and integration in time,  we obtain the desired estimate, this is, for any $\epsilon>0,\, b\geq 5\epsilon$ and $v\geq 0,$
\begin{equation*}
\begin{split}
&\sup_{0\leq t \leq T}\left\|\partial_{x}^{m}D_{x}^{1-\frac{\alpha}{2}}u\chi_{\epsilon, b}(\cdot+vt)\right\|_{2}^{2}+\left\|\partial_{x}^{m+1}u\eta_{\epsilon, b}\right\|_{L^{2}_{T}L^{2}_{x}}^{2}+\left\|\mathcal{H}\partial_{x}^{m+1}u\eta_{\epsilon, b}\right\|_{L^{2}_{T}L^{2}_{x}}^{2}\lesssim c^{*}_{m,2k+1},
\end{split}
\end{equation*}
where ${\displaystyle c^{*}_{m,2k+1}=c^{*}_{m,2k+1}\left(\alpha;k; \epsilon; T;v; \|u_{0}\|_{s_{\alpha},2}; \left\|D_{x}^{1-\frac{\alpha}{2}}\partial_{x}^{m}u_{0}\chi_{\epsilon, b}\right\|_{2}\right)>0.}$

This estimate finish the 2-induction process.

Throughout the proof we  have always  assumed as much regularity as possible on the function $u$, nevertheless  the way to proceed in the general case  involves the  device of regularization of Bona, Smith \cite{BS}.

More precisely,  we consider  initial data  $u_{0}\in H^{s_{\alpha}^{+}}(\mathbb{R}),$  then   we regularize the initial data $u_{0}^{\mu}=\rho_{\mu}*u_{0},$ where $\rho$ is a positive smooth function with compact support, more precisely $\rho\in C^{\infty}_{0}(\mathbb{R})$ with $\supp(\rho)\subset (-1,1)$ and $\|\rho\|_{1}=1.$

For $\mu>0$ we   define the family  
\begin{equation}
\rho_{\mu}(x)=\mu^{-1}\rho\left(\frac{x}{\mu}\right)\quad x\in \mathbb{R}.
\end{equation}

The solution $u^{\mu}$ associated to the IVP \eqref{e1} with initial data $u_{0}^{\mu}$ satisfies 
\begin{equation}
u^{\mu}\in C\left([0,T]: H^{\infty}(\mathbb{R})\right).
\end{equation} 

Applying the results obtained in the section of the 2-inductive process to the function $u^{\mu}$ allows us conclude   that, for any $\epsilon>0,$\, $b\geq 5\epsilon,$\, $v\geq 0,$  the following holds:
\begin{itemize}  
\item[(a)] for $l=2,3,\cdots,m$\, and \, $j=0,1,\cdots,2k-1$ 
 \begin{equation}
\begin{split}
&\sup_{0\leq t \leq T}\left\|\partial_{x}^{l}D_{x}^{\frac{\alpha j}{2}}u^{\mu}\chi_{\epsilon,b }(\cdot+vt)\right\|_{2}^{2}+\left\|\partial_{x}^{l}D_{x}^{\alpha\left(\frac{j+1}{2}\right)}u^{\mu}\eta_{\epsilon,b}\right\|_{L^{2}_{T}L^{2}_{x}}^{2}\\
&\quad +\left\|\partial_{x}^{l}\mathcal{H}D_{x}^{\alpha\left(\frac{j+1}{2}\right)}u^{\mu}\eta_{\epsilon,b}\right\|_{L^{2}_{T}L^{2}_{x}}^{2}\leq c^{*}_{l,j}
\end{split}
\end{equation}
where ${\displaystyle c^{*}_{l,j}=c^{*}_{l,j}\left(\alpha;k; \epsilon; T;v;l;j; \left\|u_{0}^{\mu}\right\|_{s_{\alpha},2}; \left\|\partial_{x}^{l}D_{x}^{\frac{\alpha j}{2}}u_{0}^{\mu}\chi_{\epsilon, b}\right\|_{2}\right)>0};$ 
 
\item[(b)] for $l=2,3,\cdots,m$
\begin{equation}
\begin{split}
&\sup_{0\leq t\leq T} \left\|\partial_{x}^{l}D_{x}^{1-\frac{\alpha}{2}}u^{\mu}(t)\chi_{\epsilon, b}(\cdot+vt)\right\|_{2}+\left\|\partial_{x}^{l+1}u^{\mu}\eta_{\epsilon,b}\right\|_{L^{2}_{T}L^{2}_{x}}^{2}\\
&\quad +\left\|\mathcal{H}\partial_{x}^{l+1}u^{\mu}\eta_{\epsilon,b}\right\|_{L^{2}_{T}L^{2}_{x}}^{2}\leq c^{*}_{l,2k+1},
\end{split}
\end{equation}
where  ${\displaystyle c^{*}_{l,2k+1}=c^{*}_{l,2k+1}\left(\alpha; k;\epsilon; T;v;n; \left\|u_{0}^{\mu}\right\|_{s_{\alpha},2}; \left\|\partial_{x}^{n}D_{x}^{1-\frac{\alpha}{2}} \partial_{x}^{2}u_{0}^{\mu}\chi_{\epsilon, b}\right\|_{2}\right)>0.}$
\end{itemize}

Next, we prove that the constants  $c^{*}_{l,j} $ and  $c^{*}_{l,2k+1}$  are independent of the parameter $\mu.$ 

First notice  that 
\begin{equation}
\left\|u_{0}^{\mu}\right\|_{s_{\alpha},2}\leq \|u_{0}\|_{s_{\alpha},2}\left\|\widehat{\rho_{\mu}}\right\|_{\infty}\leq \|u_{0}\|_{s_{\alpha},2}.
\end{equation}
 Next,   the  weighted function $\chi_{\epsilon, b}(x)=0$  for  $x\leq \epsilon,$ therefore  when we restrict   $\mu\in  (0,\epsilon)$ it follows by Young's inequality that
 \begin{equation}
 \begin{split}
 \int_{\epsilon}^{\infty}\left(\partial_{x}^{l}D_{x}^{\frac{\alpha j }{2}}u_{0}^{\mu}\right)^{2}\mathrm{d}x&= \int_{\epsilon}^{\infty}\left(\rho_{\mu}*\partial_{x}^{l}D_{x}^{\frac{\alpha j }{2}}u_{0}\mathbb{1}_{[0,\infty)}\right)^{2}\mathrm{d}x\\
 &\leq \|\rho_{\mu}\|_{1}\left\|\partial_{x}^{l}D_{x}^{\frac{\alpha j}{2}}u_{0}\right\|_{L^{2}((0,\infty))}\\
 &=\left\|\partial_{x}^{l}D_{x}^{\frac{\alpha j}{2}}u_{0}\right\|_{L^{2}((0,\infty))},
 \end{split}
 \end{equation}
  $l=2,3,\cdots,m$\, and \, $j=0,1,\cdots,2k-1.$
 
 Similarly can be proved that for $\mu\in  (0,\epsilon)$ the following inequality holds 
 \begin{equation}
 \left\|\partial_{x}^{l}D_{x}^{1-\frac{\alpha}{2}} \partial_{x}^{2}u_{0}^{\mu}\chi_{\epsilon, b}\right\|_{2}^{2}\leq \left\|\partial_{x}^{l}D_{x}^{1-\frac{\alpha}{2}} \partial_{x}^{2}u_{0}\right\|_{L^{2}((0,\infty))}\quad \mbox{for}\quad l=2,3,\cdots,m.
 \end{equation}
  Also, as part of the argument, the continuous dependence upon the initial data is  used, this is 
 \begin{equation}
 \sup_{0\leq t \leq T }\left\|u^{\mu}(t)-u(t)\right\|_{s_{\alpha},2} \underrel[c]{\mu \to 0}{\longrightarrow} 0.
 \end{equation}
 The proof finish combining  the remarks underlined above, the independence of the constants $c^{*}_{l,j} $ and  $c^{*}_{l,2k+1},$ of the parameter $\mu,$ these joint with a  weak compactness  argument and Fatou's Lemma allow  to conclude the proof.  
\section{Acknowledgments}
The results of this paper are part of the author's Ph.D dissertation at  IMPA-Brazil. He gratefully  acknowledges  the encouragement and assistance  of his advisor, Prof. F. Linares. He also express appreciation for the careful reading of the manuscript done   by  O. Ria\~{n}o.

\end{document}